\input ./style/arxiv-general.cfg
\documentclass[aop,MSNbibl,dvips]{arximspdf}
\makeatletter
   \@ifpackageloaded{graphicx}{}{\usepackage{graphicx}}
\makeatother
\usepackage{mathbh}

\doi{10.1214/15-AOP1003}
\volume{44}
\issue{2}
\pubyear{2016}
\firstpage{1341}
\lastpage{1425}
\docsubty{FLA}

\makeatletter
\newcommand{\fontas}{\fontsize{8.36}{10.36}\selectfont}
\def\sfrac#1#2{#1/#2}

\newcommand{\lleft}{\left}
\newcommand{\rrvert}{\vert}
\newcommand{\rright}{\right}
\newcommand{\llvert}{\vert}
\newtheorem{teo}{Theorem}
\newtheorem{Lemma}[teo]{Lemma}
\newtheorem{Cor}[teo]{Corollary}
\newproclaim{definition}[teo]{Definition}
\newproclaim{convention}{Convention}
\newproclaim{fact}{Key Fact}
\newproclaim{remark}{Remark}
\renewcommand{\mid}{|}
\newcommand{\bR}{\mathbb{R}}
\renewcommand{\Re}{\mathbb{R}}
\newcommand{\Lap}{\Delta}
\renewcommand{\phi}{\varphi}
\newcommand{\sO}{\mathcal{O}}
\newcommand{\sE}{\mathfrak{E}}
\newcommand{\dist}{\operatorname{dist}}
\newcommand{\osc}{\operatorname{osc}}
\newcommand{\area}{\operatorname{area}}
\newcommand{\Cut}{\operatorname{Cut}}
\newcommand{\E}{\mathbb{E}}
\newcommand{\Prob}{\mathbb{P}}
\newcommand{\OldSigma}{\tau}
\newcommand{\OldTau}{\sigma}
\makeatother

\begin{document}
\begin{frontmatter}

\title{A stochastic target approach to Ricci flow on~surfaces}
\runtitle{Stochastic approach to Ricci flow}

\begin{aug}
\author[A]{\fnms{Robert W.}~\snm{Neel}\thanksref{M1,T1}\ead[label=e1]{robert.neel@lehigh.edu}}
\and
\author[B]{\fnms{Ionel}~\snm{Popescu}\corref{}\thanksref{M2,M3,T2}\ead[label=e2]{ionel.popescu@imar.ro}\ead[label=e3]{ipopescu@math.gatech.edu}}
\runauthor{R. W. Neel and I. Popescu}
\affiliation{Lehigh University\thanksmark{M1}, Georgia Institute of Technology\thanksmark{M2} and\\ ``Simion Stoilow'' Institute of Mathematics of Romanian Academy\thanksmark{M3}}
\address[A]{Department of Mathematics\\
Lehigh University\\
Christmas-Saucon Hall\\
14 East Packer Avenue\\
Bethlehem, Pennsylvania 18015\\
USA\\
\printead{e1}}
\address[B]{School of Mathematics\\
Georgia Institute of Technology\\
686 Cherry Street\\
Atlanta, Georgia 30332\\
USA\\
and\\
``Simion Stoilow'' Institute of Mathematics\\
\quad of Romanian Academy\\
21 Calea Grivi\c tei\\
Bucharest\\
Romania\\
\printead{e2}\\
\phantom{E-mail: }\printead*{e3}}
\end{aug}
\thankstext{T1}{Supported in part by an NSF Postdoctoral Research Fellowship.}
\thankstext{T2}{Supported in part by the Marie Curie Action Grant
PIRG.GA.2009.249200
and also by a grant of the Romanian National Authority for Scientific
Research CNCS---UEFISCDI, project number PN-II-RU-TE-2011-3-0259.}

%
\received{\smonth{9} \syear{2012}}
%
\revised{\smonth{1} \syear{2015}}

%
\begin{abstract}
We develop a stochastic target representation for Ricci flow and
normalized Ricci flow on smooth,
compact surfaces, analogous to Soner and Touzi's representation of mean
curvature flow.
We prove a verification/uniqueness theorem, and then consider geometric
consequences of this stochastic representation.

Based on this stochastic approach, we give a proof that, for surfaces
of nonpositive
Euler characteristic, the normalized Ricci flow converges to a constant
curvature metric
exponentially quickly in every $C^k$-norm. In the case of $C^{0}$ and
$C^{1}$-convergence,
we achieve this by coupling two particles. To get $C^{2}$-convergence
(in particular,
convergence of the curvature), we use a coupling of three particles.
This triple coupling
is developed here only for the case of constant curvature metrics on
surfaces, though we
suspect that some variants of this idea are applicable in other
situations and therefore
be of independent interest. Finally, for $k\ge3$, the
$C^{k}$-convergence follows relatively easily using induction and
coupling of two particles.

None of these techniques appear in the Ricci flow literature and thus
provide an alternative approach to the field.
\end{abstract}

%
\end{frontmatter}

\section{Introduction}\label{sec1}

In~\cite{SonerTouzi}, Soner and Touzi give a characterization of
various extrinsic geometric flows (with ambient space $\Re^n$),
including mean curvature flow, as stochastic target problems. More
specifically, they introduce the relevant target problems and then
prove associated verification theorems, namely theorems showing that if
the curvature flow has a smooth solution for an interval of time $t\in
[0,T)$, then the solution agrees with the solution to the stochastic
target problem on this interval. In the first part of this paper, we
develop a similar characterization of Ricci flow (and normalized Ricci
flow) on compact surfaces, including the relevant verification theorems
(see Theorem \ref{teoVeri}). We then briefly discuss time-dependent
bounds on the solution to both normalized and un-normalized Ricci flow
and estimates on the blow-ups of solutions to Ricci flow in the cases
of nonzero Euler characteristic, all obtained from the stochastic
formulation of the flow. In the remainder of the paper, we use this
stochastic representation to prove that, for a smooth, compact surface
of nonpositive Euler characteristic, given that a smooth solution to
the normalized Ricci flow exists for all time (which is well known from
the literature), it converges to a constant curvature metric
exponentially fast in $C^\infty$ (see Theorem \ref{teoCInfyConv} for
a precise statement).

Ricci flow on smooth, compact surfaces is essentially completely
understood as, for instance, \cite{ChKn1,CLN,Ham1}. Nonetheless, one
feature of our approach is that probability often provides an appealing
intuition, as in the case of Brownian motion and heat flow. Thus, if
Ricci flow is thought of as a kind of ``heat equation for curvature,''
it is natural to want to extend the analogy to include a diffusion
interpretation. For example, it is nice to see the convergence of a
manifold under normalized Ricci flow to a constant curvature limit as
the equi-distribution of the metric, and as a result of the curvature,
from a probabilistic perspective.

More generally, one might ask about the potential merits of developing
stochastic techniques for Ricci flow (or other curvature flows). One
obvious point to be made here is that one gets a representation of the
solution and, at least in the theory of linear second-order PDEs, this
has turned out to be extremely versatile in extracting properties of
the solutions. As we will see, the stochastic tools we employ are good
enough to give a different proof of a main result in the theory of
Ricci flow on surfaces with the bonus that we see the ``averaging
property of the curvature'' as a consequence of coupling, which is a
probabilistic manifestation of ergodicity. Another motivation for such
an endeavor is that the stochastic target formulation is fairly
insensitive to regularity, and thus potentially useful in formulating
notions of weak solutions. Indeed, in a second paper, Soner and
Touzi~\cite{ST2} show that generalized solutions to various extrinsic
curvature flows can also be understood in terms of stochastic target
problems. Also stemming from these ideas, we note that stochastic
approaches to PDEs can lend themselves to the development of
probabilistic numerical schemes (as in \cite{Touzi}), but we do not
touch this subject here.

Our framework is not the most general one. We presumably could have
worked in a little more generality, but to keep the ideas as appealing
and clear as possible, we decided to study surfaces, which are the
traditional starting point for studying Ricci flow.

We point out that, as noted in \cite{CheriditoST}, stochastic target
problems of certain kind are equivalent to second-order backward
stochastic differential equations. As discussed there, second-order
backward SDEs are natural stochastic objects to associate with fully
nonlinear PDEs. Thus, one could presumably recast the results of this
paper in those terms. Nonetheless, we have chosen to adopt the
stochastic target approach because it seems more geometrically
intuitive and visually appealing, and because it puts Ricci flow and
mean curvature flow in a similar framework.

There are few papers on stochastic analysis and Ricci flow, for
instance, \mbox{\cite{ACT,r3,KACP,r4,r2,r1}}. The ones that are somewhat
closer to our work are \cite{ACT} and \cite{KACP}. These papers
investigate the Brownian motion (and the associated parallel transport)
with respect to a time changing metric on a manifold of any dimension,
not only on surfaces. Using stochastic analysis, they also develop a
Bismut-like formula to represent the gradient of solutions to heat-type
flows with respect to the time-dependent metric. In particular, this
leads to gradient estimates for the corresponding solutions.

We use in this paper a different tool, namely couplings. Coupling is a
very useful thing and has been successfully used in a variety of
situations. There is a large body of research on this and we will point
out only some selections without any claim of completeness on the
subject. One of the most useful on is the \emph{mirror coupling}~of
Brownian motions introduced by Lindvall and Rogers in \cite{Lin-Rog} in
the Euclidean setting and by Cranston \cite{CranstonJFA} and Kendall
in \cite{coup8} on manifolds. Different types of couplings, as, for
instance, the synchronous coupling appearing in \cite
{MR1768241,MR1001020,MR1080491} and shy coupling which is treated in
\cite{Burdzy-Benjamini,coup11,MR3077525} or even fixed-distance
couplings on manifolds in \cite{fixeddist}. There are lots of
applications of the coupling in geometric and analytic problems as it
can be seen from an excerpt of the literature in \cite
{coup10,coup5,coup3,coup6,coup8,coup9,coup7,coup11,coup4,coup1,coup2}.

One of the main techniques in the present work is the mirror coupling
applied to time changed Brownian motions. It turns out to be an
extremely fruitful tool for proving estimates in the context of Ricci
flow. Though the coupling and the Bismut formula produce in several
cases similar estimates, particularly when it comes to gradient
estimates, we do not know how to get a nice and useful version of the
Bismut formula for the second-order derivatives. This is one of the
reasons we prefer to deal with the coupling techniques which reveals
its full power. The idea of dealing with the second-order derivatives
comes from \cite{CranstonTripleCouple}, where a coupling of three
particles is used to estimate second-order derivatives of harmonic
functions on Euclidean domains. This triple coupling indicated by
Cranston uses a certain symmetry to get a key cancellation in the
estimation of the Hessian. This symmetry is not surprising in the flat
case. However, there are immediate technical challenges for a similar
construction on manifolds, and the way it works in the flat case does
not seem to work on arbitrary manifolds for the time changed Brownian
motions. Nevertheless, it turns out that we can construct such a triple
coupling which has enough good properties in the case of surfaces of
constant curvature.

We continue with a few more observations about the present work. We do
not prove the existence of solutions to the target problem directly;
rather, the verification theorems proceed from the assumption that the
Ricci flow admits a smooth solution. In the case of normalized Ricci
flow, we have long-time existence as proved in \cite{Cao} and \cite
{Ham1}. However, an immediate consequence of such a verification
theorem is that the solution (to the flow) is unique.

In contrast to the standard proof of the convergence to constant
curvature, we deal directly with the metric itself (and its
derivatives), rather than introducing an auxiliary PDE satisfied by the
curvature. We use uniformization to work with an underlying metric
which has constant curvature and is in the same conformal class as the
initial metric. One might hope to extend these arguments to more
general situations, but for a first paper on this approach
uniformization makes the analysis cleaner and reveals the power of the
coupling in a nice way.

The outline of the paper is as follows. We first describe the
stochastic target problem in Section~\ref{s1} giving a fair amount of
detail, since it is a somewhat nonstandard control problem. Then, in
Section~\ref{s2} we prove the verification/uniqueness theorem, namely
that, if there is a smooth solution to the Ricci flow (or normalized
Ricci flow) on some interval of time, then it agrees with the solution
to the stochastic target problem.

Section~\ref{s3} is a short section showing how one can use the
representation to prove that the unnormalized Ricci flow develops
singularities (in certain cases) either in finite time or in infinite
time. In Section~\ref{s4}, we develop the a priori bounds for the
stochastic target problem. As a consequence, we obtain the exponential
convergence in the $C^{0}$-norm of the normalized flow in the case of
$\chi(M)<0$ [as usual, $\chi(M)$ denotes the Euler characteristic of
$M$]. We also include a short discussion of the blow up of the
unnormalized Ricci flow in the cases $\chi(M)>0$ and $\chi(M)<0$,
which is in tune with the previous section's findings, although this
time assuming uniformization.

Section~\ref{s5} introduces and proves the main result on mirror
coupling for the time changed Brownian motions associated to the target
problems. This coupling is well defined for short times, but the main
challenge is to show that the coupling extends beyond the cut locus.
This is done using the geometric structure of the cut locus on surfaces
of Euler characteristic less than or equal to $0$. We should also point
out that there is a coupling of Brownian motions constructed with
respect to time-varying metrics (such as Ricci flow) in \cite{r4}, but
it differs from our situation here.

In Section~\ref{s6}, we start the main analysis of the convergence of
normalized Ricci flow. We prove the nontrivial fact that in Euler
characteristic zero, the normalized flow converges exponentially fast
in the $C^{0}$-topology. This uses the result from the previous section
combined with the comparison of the distance process with a Bessel
process in order to estimate the coupling time. This is a fundamentally
probabilistic idea. Combining this result with those coming from the a
priori estimates proves that, for nonpositive Euler characteristic,
the flow converges in the $C^{0}$-topology exponentially fast.

The next task is to prove that the convergence takes place also in
$C^{1}$, or in other words that the gradient of the metric converges
exponentially fast. This is done in Section~\ref{s7}, again using
coupling. However, the point here is a little different. We use the
coupling for particles started close to one another and estimate the
coupling time in terms of the gradient of the metric (more precisely
the conformal factor of the metric) and the initial distance. This in
turn yields a functional inequality satisfied by the $C^{0}$-norm of
the gradient which is contained in Lemma~\ref{gel1}. It turns out
that this functional inequality is strong enough to produce the
exponential convergence.

Going forward, Section~\ref{s8} is dedicated to the triple coupling
used in a crucial way for the Hessian estimates. We exploit in an
essential way the constant curvature properties of the underlying
metric. We have two mirror coupled particles $x$ and $y$ and another
middle particle $z$ which is moving on the geodesic between them which
is described by the distance $\rho_{1}$ from $z$ to $x$, or
alternatively, the distance $\rho_{2}$ from $z$ to $y$. One of the
main interests is the symmetry with respect to swapping $\rho_{1}$ and
$\rho_{2}$. The other thing thrust of the investigation is as follows.
Assuming that $x$ and $y$ are time changed Brownian motions, we study
the conditions under which $z$ is a time changed Brownian motion with a
drift. This is a key point in the Hessian estimates.

Section~\ref{s9} covers the Hessian estimates. Here, we use the
results from the previous sections, for instance, the exponential decay
of the flow in the $C^{1}$-topology and the triple coupling. As in the
case of the gradient, we end up with a functional inequality for the
$C^{0}$-norm of the Hessian as in Lemma~\ref{hel2}. It turns out that
this suffices to conclude the exponential convergence.

The last section proves the $C^{k}$-convergence of the flow. This is
done essentially using the Ricci flow equation and induction. It is
important to mention here that in the flat case, we still use couplings.

A few words about the sphere case, which requires some finer analysis.
There are several obstacles we have to overcome. On one hand, the a
priori estimates give bounds which blow up in finite or infinite time.
However, these estimates are simply bounds of a stochastic differential
equation in terms of the ODE in which the martingale is killed off, and
eventually can likely be refined. Further, in the case of nonpositive
Euler characteristic, there is a unique stationary solution to the
normalized Ricci flow with a given volume (in a given conformal class),
and thus one has to prove that the flow converges to this uniquely
determined solution. In the case of the sphere, this is not the case,
and thus convergence is harder to establish, because we do not know
beforehand toward which stationary solution the flow wants to converge
(this is related to the issue of Ricci solitons). Therefore, the
strategy we used in this paper for $\chi(M)\le0$ needs some
refinements if it's to address the case of positive Euler characteristic.

\section{Stochastic target formulation}\label{s1}

\subsection{Ricci flow}

Consider a smooth, compact Riemannian surface $(M, h)$, that is, $M$ is
a smooth, compact manifold without boundary of dimension two and $h$ a
smooth Riemannian metric on $M$. Any other smooth metric in the same
conformal class as $h$ can be written as $g=\bar{u}h$ for some smooth,
positive function $\bar{u}$. The Ricci curvature of any metric metric
$g$ is given by
%
\begin{equation}
\label{e1} 2\operatorname{Ric}_g=R_g
g=2K_{g}g,
\end{equation}
where $R_g$ is the scalar curvature and $K_{g}$ is the Gauss curvature.
The \emph{Ricci flow} is defined as the evolution of the metric
$g_{t}$ according to
%
\begin{equation}
\label{e0r} \partial_{t}g_{ij}=-2\operatorname{Ric}_{ij},
\end{equation}
where $\operatorname{Ric}$ is the Ricci tensor. From this, it is easy to see
that the Ricci flow preserves the conformal class in two dimensions,
and thus it becomes an evolution equation for the conformal factor
$\bar{u}_{t}$. In particular, the Ricci flow corresponds to
$\bar{u}$
evolving by
%
\begin{equation}
\label{EqnRFu} \partial_t \bar{u}_t =
\Lap_h \log\bar{u}_t - 2K_h,
\end{equation}
where $K_h$ is the Gauss curvature of $(M,h)$. In passing from (\ref
{e0r}) to (\ref{EqnRFu}), we have already used the fact that if
$g=uh$, for two metrics, $g$ and $h$, then (see \cite{CLN}, Exercise 2.8)
%
\begin{equation}
\label{e02} R_{g}=\frac{1}{u}(R_{h}-
\Delta_{h}\log u),
\end{equation}
where the $\Delta_{h}$ is the Laplacian with respect to the metric $h$.

This is a nonlinear parabolic equation, and thus the usual
probabilistic methods of solution (diffusions, Feynman--Kac, etc.) do
not apply. Instead, we will adopt a stochastic target approach modeled
on the approach of \cite{SonerTouzi} to mean curvature flow, as
mentioned above.

To be more concrete, we assume that the initial metric on $M$ can be
written as $g_0=\bar{u}_0h$ for some smooth, positive $\bar{u}$ and
some metric $h$. There are two natural choices for $h$. Of course, we
can let $h=g_0$ and $\bar{u}_0\equiv1$. Alternatively, the
uniformization theorem, for instance, \cite{sergiu}, Chapter~3, implies
that there is a metric in the same conformal class as $g_0$ which has
constant curvature of $-1$, 0, or 1. Then we can take $h$ to be this
metric, in which case $\bar{u}_0$ is determined by the condition that
$g_0=\bar{u}_0h$. We will find the flexibility of this set-up to
be useful.

As usual, we also wish to introduce the \emph{normalized Ricci flow},
which is defined as
%
\begin{equation}
\label{e03} \partial_{t}g_{ij}=-2\operatorname{Ric}_{ij}+2r
g_{ij},
\end{equation}
where $r$ is the average of the Gauss curvature on $M$ with respect to
the metric $g$. Written in terms of the conformal factor, this is
\[
\partial_{t}\bar{u}_{t}=\Delta_{h}\log
\bar{u}-2K_{h}+2r_{t}\bar{u}_{t}.
\]

Under this flow, the surface is continually rescaled to preserve the
area. Indeed, the Gauss--Bonnet theorem tells us that the integral of
the scalar curvature is
\[
\int K_{g}\,dA_{g}=2\pi\chi(M),
\]
where $\chi(M)$ is the Euler characteristic of $M$ and $A_{g}$ is the
area element of the metric $g$. Consequently, if $r_{t}$ is the average
of the Gauss curvature for $g_{t}$, then
\[
r_{t}=\frac{2\pi\chi(M)}{\area(M,g_{t})},
\]
where $\area(M,g)$ stands for the area of $M$ with the metric $g$.
From here, a straightforward calculation gives that
\[
\partial_{t}\area(M,g_{t})=\partial_{t} \int
\bar{u}_{t}\,d A_{h}=\int\partial_{t}
\bar{u}_{t}\,d A_{h}=-2\int K_{h}\,d
A_{h}+2 r_{t}\int\bar{u}_{t}\,d
A_{h}=0,
\]
which shows that the area is preserved under this evolution and, in
particular, $r_{t}$ does not depend on $t$. Therefore, the flow (\ref
{e03}) preserves the area and
%
\begin{equation}
\label{e05} r=\frac{2\pi\chi(M)}{\area(M,g_{0})}.
\end{equation}
We can now translate (\ref{e03}) into an equation satisfied by the
conformal change $\bar{u}_{t}$ as (recall that $g_{t}=\bar{u}_{t}h$)
%
\begin{equation}
\label{e04} \partial_{t}\bar{u}_{t}=\Delta_{h}
\log\bar{u}-2K_{h}+2r\bar{u}_{t}
\end{equation}
with $r$ the constant from (\ref{e05}).

As is implicit in the above, we see that the set of all smooth metrics
(on $M$) in a given conformal class corresponds to the set of smooth
sections of a one-dimensional bundle over $M$. More concretely, fixing
a ``reference metric'' $h$ and writing any other (smooth) metric (in
the same conformal class) as $\bar{u}h$ induces a global
coordinate $u$
on fibers of this bundle making the total space $E$ diffeomorphic to
$M\times(0,\infty)$. Further, $\bar{u}$ is given as the composition
of the lift from $M$ to $E$ (corresponding to the section) with $u$.
This helps to explain the notation: $u$ is a coordinate on the fibers,
and $\bar{u}$ is the expression of a section in this coordinate.
Because our bundle admits natural global coordinates, we will almost
always work in these coordinates, and thus we will not have much
occasion to consider sections in a coordinate-free notation.

Viewed in this light, it is natural to introduce a new coordinate on
the fibers. Let $p=(1/2)\log u$. Then any other metric in the same
conformal class as $h$ can be written as $g = e^{2\bar{p}} h$ for some
smooth function $\bar{p}\dvtx M\rightarrow\Re$, which is given by
the composition
of the lift $M\rightarrow E$ (corresponding to the section) with $p$. This
coordinate makes the bundle into a real line bundle. In particular, the
metric $h$ corresponds to the zero section, and fiberwise addition
corresponds to composition of conformal changes. However, we will not
need the vector space structure on fibers in what follows; we really
just view the fibers as having a smooth structure. In terms of the
coordinate $p$, the Ricci flow equation becomes
%
\begin{equation}
\label{EqnRFp} \partial_t \bar{p}_t =
e^{-2\bar{p}_t} (\Lap_h \bar{p}_t
-K_h ),
\end{equation}
and the normalized Ricci flow equation becomes [see also \cite{LMa},
equation (1.3.1)]
%
%
\begin{equation}
\label{EqnNRFp} \partial_t \bar{p}_t =
e^{-2\bar{p}_t} (\Lap_h \bar{p}_t
-K_h )+r,
\end{equation}
with $r$ the constant defined in (\ref{e05}), and thus depending only
on the area of $M$ with respect to the initial metric $g_{0}$.

At this point, we see that there is a one-to-one correspondence between
metrics in the same conformal class as $h$, sections of $E$ over $M$,
and functions $\bar{p}$ (where all of these objects are assumed
to be
smooth). Further, there is a one-to-one correspondence between smooth
sections and smooth hypersurfaces of $E$ that intersect each fiber once
and do so transversely; under composition with $p$ this is the same as
the correspondence between smooth functions on $M$ and their graphs in
$M\times\Re$. Viewing metrics as hypersurfaces in the total space $E$
provides a framework for studying Ricci flow which is fairly similar to
that of mean curvature flow and well suited for the stochastic target
approach. Our next task is to define the appropriate target problem.

\subsection{The target problem}

Let $\Gamma(0)$ be the hypersurface corresponding to the initial
metric $g_0$. In spite of our previous efforts to distinguish between
sections over $M$ from their description in a particular coordinate, in
what follows we will fix the global coordinate $p$ on fibers, thus
identifying the fibers with $\Re$, and formulate everything in those
terms. In particular, $\Gamma(0)$ corresponds to the graph of
$\bar{p}_0$. The stochastic target problem is, for any time $t$,
the problem
of determining the set of points such that the controlled process,
starting from such a point, can be made to hit $\Gamma(0)$ (the
``target'') in time $t$ almost surely. Obviously, this requires
specifying the allowed controls and the processes they give rise to. We
will generally explain things for the Ricci flow and then indicate the
analogous results for the normalized Ricci flow in situations where
there are no additional complications.

We start with the infinitesimal picture in normal coordinates. We
choose any point $(q,\hat{p})\in M\times\Re$ and let $(x_1,x_2)$ be
normal coordinates around $q$. Thus, $(x_1,x_2,p)$ are coordinates on a
neighborhood of $\{q\} \times\Re$. We assume that the controlled
process is currently at $(q,\hat{p})$, say at time $\tau$. The
$(x_1,x_2)$-marginal of the controlled process will be
(infinitesimally) Brownian motion on $M$ (with fixed reference metric
$h$), time-changed by $2e^{-2\hat{p}}$. The control consists of
choosing a lift of the tangent plane to $M$ at $q$ into the tangent
space to $E$ at $(q,\hat{p})$. The controlled process has its
martingale part diffusing (infinitesimally) along this lifted plane in
the unique way that gives the right $(x_1,x_2)$-marginal, and has its
drift along the fiber at rate $e^{-2\hat{p}}K_h$ [plus an additional
$-2\pi\chi(M)/\area(M,h)$ for the normalized Ricci flow]. More
precisely, the control consists of a choice of $(a_1,a_2)\in\Re^2$,
for which the processes evolves [infinitesimally, assuming the process
is at $(q,\hat{p})$ at time $\tau$] according to
\[
\lleft[\matrix{ dx_{1,\tau}
\vspace*{3pt}\cr
dx_{2,\tau}
\vspace*{3pt}\cr
dp_{\tau}
} \rright] = \lleft[\matrix{ e^{-\hat{p}} & 0
\vspace*{3pt}\cr
0 & e^{-\hat{p}}
\vspace*{3pt}\cr
e^{-\hat{p}}a_1 & e^{-\hat{p}}a_2 } \rright] \lleft[\matrix{ \sqrt{2}
\,dW_{\tau}^1
\vspace*{3pt}\cr
\sqrt{2}
\,dW_{\tau}^2 } \rright] + \lleft[\matrix{ 0
\vspace*{3pt}\cr
0
\vspace*{3pt}\cr
e^{-2\hat{p}}K_h(q) } \rright],
\]
where $W^1$ and $W^2$ are one-dimensional Brownian motions. Here, we
have written $K_h(q)$ to emphasize that the curvature depends on the
point in $M$. The $\sqrt{2}$ factors (in front of the Brownian
differentials) are needed because the Ricci flow is defined using the
Laplacian, instead of half the Laplacian, and rather than use a
nonstandard normalization for the Ricci flow, we choose to speed up
our Brownian motions (this is analogous to the usual discrepancy
between the analyst's and the probabilist's versions of the heat
equation). This is the controlled process, at least infinitesimally,
corresponding to the Ricci flow. For the normalized Ricci flow, the set
of controls is the same, but the process evolves according to
\[
\lleft[\matrix{ dx_{1,\tau}
\vspace*{3pt}\cr
dx_{2,\tau}
\vspace*{3pt}\cr
dp_{\tau}
} \rright] = \lleft[\matrix{ e^{-\hat{p}} & 0
\vspace*{3pt}\cr
0 & e^{-\hat{p}}
\vspace*{3pt}\cr
e^{-\hat{p}}a_1 & e^{-\hat{p}}a_2 } \rright] \lleft[\matrix{ \sqrt{2}
\,dW_{\tau}^1
\vspace*{3pt}\cr
\sqrt{2}
\,dW_{\tau}^2 } \rright] + \lleft[\matrix{ 0
\vspace*{3pt}\cr
0
\vspace*{3pt}\cr
e^{-2\hat{p}} K_h(q) -r } \rright].
\]

We point out that, for both the Ricci flow and the normalized Ricci
flow, the (infinitesimal) diffusion matrix is
\[
\lleft[\matrix{ 2e^{-2p_{\tau}} & 0 & 2e^{-2p_{\tau}}a_1
\vspace*{3pt}\cr
0 & 2e^{-2p_{\tau}} & 2e^{-2p_{\tau}}a_2
\vspace*{3pt}\cr
2e^{-2p_{\tau}}a_1 & 2e^{-2p_{\tau}}a_2 &
2e^{-2p_{\tau
}}\bigl(a_1^2+a_2^2
\bigr) } \rright]
\]
in $(x_1,x_2,p)$ coordinates at $(q,\hat{p})$, of course.

Having given the infinitesimal picture, we now extend this to a global
description. While it is tempting to simply assert that this follows
immediately from the local description, we prefer to give a more
explicit formulation. There is more than one way to do this, but we
choose to use the bundle of orthonormal frames on $(M,h)$. The
immediate difficulty with extending the above local picture is that,
except in special cases (more on which below), we cannot find
coordinates which are normal at more than one point at a time, or even
a global orthonormal frame. The solution we have in mind is to use the
bundle of orthonormal frames to supply each point along the evolving
process with an orthonormal frame and its associated normal
coordinates. In particular, let $\sO(M)$ be the bundle of orthonormal
frames over $(M,h)$, consisting of points $(q,\mathfrak{e}(q))$ where
$q\in M$ and $\mathfrak{e}(q)$ is an orthonormal basis for $T_qM$ with
metric $h$. We identify $\mathfrak{e}(q)$ with the corresponding
linear isometry from $\Re^2$ to $T_qM$. Let $e_1$ and $e_2$ be the
standard basis for $\Re^2$ and let $\sE(e_i)$ be the corresponding
canonical vector fields. Further, we let $\pi\dvtx\sO(M)\rightarrow M$
be the
usual projection and $\pi_{*}\dvtx T\sO(M)\rightarrow TM$ be the induced
push-forward map on tangent spaces.

The connection with the previous infinitesimal picture comes from the
following relationship between the canonical vector fields and normal
coordinates. Choose a point $q\in M$ and a frame $\mathfrak{e}(q)$
over $q$, and let $(x_1,x_2)$ be normal coordinates [for $(M,h)$] in a
neighborhood of $q$ such that $\partial_{x_i} = \mathfrak{e}(q)(e_i)$
at $q$. Obviously, $\pi_{*} [\sE(e_i)\mid_{(q,\mathfrak
{e}(q))} ]= \partial_{x_i}\mid_q$. Moreover, let $s$ be a smooth
section of $\sO(M)$ in a neighborhood of $q$ which is equal to
$\mathfrak{e}(q)$ at $q$ and horizontal at $q$, meaning that $\partial
_{x_i} s$ are horizontal vectors at $q$. Then $\pi_{*} [\sE
(e_i)\circ s ]$ agrees with $\partial_{x_i}$ to first-order
around~$q$. (Indeed, to show that such a section $s$ exists, start with
normal coordinates and apply the Gram--Schmidt process to $\{\partial
_{x_1},\partial_{x_2}\}$ at every point in a neighborhood of~$q$.)

We also recall the connection between the bundle of orthonormal frames
and Brownian motion on $(M,h)$. We have that $ (\sE(e_1)^2+\sE
(e_2)^2 )/2$ is\break Bochner's Laplacian on $\sO(M)$, and the
corresponding martingale problem is well posed (in the sense of Stroock
and Varadhan, namely that there is a unique solution for any initial
point). We use $\tilde{B}_{\tau}$ to denote such a process.
Projecting $\tilde{B}_{\tau}$ to $M$ gives Brownian motion on $M$,
which we denote $B_{\tau}$. This is the well-known
Eells--Elworthy--Malliavin construction of Brownian motion on $M$, and
we refer the reader to \cite{Elton} or~\cite{Stroock} for a detailed
account on the subject. Moreover, the process $\tilde{B}_{\tau}$ on
$\sO(M)$ should be thought of as the horizontal lift of $B_{\tau}$ on
$M$, and thus as giving Brownian motion equipped\vspace*{1pt} with parallel
transport. In particular, this is how we will typically understand
$\tilde{B}_{\tau}$, as Brownian motion on $M$ endowed with parallel
transport. Finally, we note that the solution to the martingale problem
for Bochner's Laplacian can be realized as the (unique) strong solution
to the natural SDE driven by a standard Brownian motion on $\Re^2$, or
equivalently, two independent, one-dimensional Brownian motions. That
is, $\tilde{B}_{\tau}$ can be realized as the solution to
\[
d \tilde{B}_{\tau} = \sE(e_1)\circ dW_{\tau}^1
+\sE(e_2) \circ dW_{\tau}^2,
\]
where $\circ dW$ indicates that the differential is to be understood in
the Strato\-no\-vich sense.

We now have the necessary background to give the global formulation of
the stochastic target problem for Ricci flow (and the related target
problem for normalized Ricci flow). We write points in $E$ as $(x,p)\in
M\times\Re$ and the controlled process (for the Ricci flow) as
$Y_{\tau} = (x_{\tau},p_{\tau})$. As suggested above, the
$M$-marginal $x_{\tau}$ will be Brownian motion on $M$, time-changed
by $p$, and thus we know from the above that we have parallel transport
of frames (for $T_xM$) along the paths $x_t$ (note that the frame is
always orthonormal relative to the metric $h$). In particular, if we
choose a frame $\mathfrak{e}(x_0)$ at the starting point, then we let
$\mathfrak{e}(x_{\tau})$ denote the parallel transport of this frame
along $x_{\tau}$. Abstractly, the control consists in choosing a lift
of $T_{x_{\tau}}M$ to $T_{(x_{\tau},p_{\tau})}E$. In\vspace*{1pt} terms of our
evolving frame, such lifts can be identified with points of $\Re^2$.
This is the time to formally introduce the control process. In what
follows, $(\Omega,\mathcal{F},\Prob)$ is a probability space where
the Brownian motion $(W^{1},W^{2})$ is defined and the reference
filtration involved here is $\mathcal{F}_{\tau}$, the one generated
by the Brownian motion.

%
\begin{definition}
For a fixed time $t>0$, an admissible control process $A$ is a bounded
map $A\dvtx[0,t]\times M \times\Omega\to\Re^2$ which is continuous in
the first two coordinates, and such that for each $(x,\tau)\in M\times
[0,t]$, $A(\tau,x)\dvtx\Omega\to\Re^{2}$ is $\mathcal{F}_{\tau
}$-measurable. We write this in components $A=(a_1,a_2)$.
\end{definition}

We will explain below in the first remark of this section why we
require the control to be bounded.

If we start our process from a point $Y_0=(x_0,\bar{p}_{0})$ equipped
with a frame $\mathfrak{e}(x_0)$ of $T_{x_0}M$, then it evolves
according to the SDE (note that we are using both It\^o and
Stratonovich differentials)
%
\begin{eqnarray}
\label{e0} dx_{\tau} &=& e^{-p_{\tau}} \Biggl[ \Biggl[\sum
_{i=1}^2 \mathfrak{e}(x_{\tau})
(e_i)\sqrt{2} \circ dW_{\tau}^i \Biggr] \Biggr],
\nonumber\\[-8pt]\\[-8pt]\nonumber
dp_{\tau} &=& e^{-p_{\tau}} \Biggl[\sum_{i=1}^2
a_i \sqrt{2} \,dW_{\tau}^i \Biggr] +
e^{-2p_{\tau}} K_h(x_{\tau}) \,d\tau.
\end{eqnarray}
This equation comes with the following convention.

\begin{convention}
Whenever we have a bracket $A[\![M\circ dN]\!]$, the order of operations is
that we first write $M\circ dN=M \,dN+\frac{1}{2}\,d\langle M,N \rangle$
in It\^o form and then multiply everything by $A$. Thus, we have
\[
A[\![M\circ dN ]\!]=(AM)\,dN+\tfrac{1}{2}A\,d\langle M,N \rangle
\]
as opposed to the common writing
\[
A[M\circ dN ]=AM \,dN+\tfrac{1}{2}\,d\langle AM,N \rangle,
\]
where $\langle M,N \rangle$ is the quadratic variation of $M$ and $N$.
For the standard It\^o differentials, the meaning is the standard one, namely
\[
A[M\,dN ]=(AM)\,dN.
\]

Though we can rewrite in a more conventional way
\[
A[\![M\circ dN]\!]=M\circ(A\,dN),
\]
we prefer the notation $A[\![M\circ dN]\!]$ because it is more suggestive
in our context that $A$ is the time change of the process $M\circ dN$.
This becomes even better in the context of equation (\ref{e0}) that
$x_{\tau}$ is simply a time changed Brownian motion on $M$.
\end{convention}

Here, we see that $\mathfrak{e}(x_\tau)(e_i)$ is just the projection
onto $M$ of $\sE(e_i)$ and to ease the notation we will also use the
shortcut $\mathfrak{e}(x_\tau)(e_i)=\mathfrak{e}_{i}(x_\tau)$, or
even more simply~$\mathfrak{e}_{i}$, if there is no confusion
generated by dropping $x_{\tau}$. In particular, the horizontal lift
of $x_\tau$, which we write $\tilde{x}_{\tau}= (x_{\tau},\mathfrak
{e}(x_{\tau}))$ evolves according to
\[
d \tilde{x}_\tau= e^{-p_{\tau}} \Biggl[ \Biggl[\sum
_{i=1}^2 \sE(e_i)\sqrt{2} \circ
dW_{\tau}^i \Biggr] \Biggr]\qquad\mbox{on }\sO(M),
\]
and the first line of (\ref{e0}) is just the projection of this onto
$M$. We choose to write (\ref{e0}) in this way in order to emphasize
that we are ultimately only interested in the evolution of the surface
in $E$ and not in the frame; the frame is only used as a convenience in
order to express the control and the corresponding SDE. We do this
despite the fact that (\ref{e0}) requires evolving the frame
$\mathfrak{e}(x_{\tau})$ as well.

The mixing of It\^o and Stratonovich differentials in (\ref{e0}) is a
result of the fact that horizontal Brownian motion (or just Brownian
motion on $M$) is not easily written globally in It\^o form. To clarify
this, we give the following equivalent characterization, which is just
a consequence of It\^o's formula but one of the important properties of
the above system. For any smooth function $\phi\dvtx[0,T]\times M\times
\Re\to\Re$ (assuming that the process $(x_{\tau},p_{\tau})$ exists
for $\tau\in[0,T]$),
%
\begin{eqnarray}\label{e2}
&& d\phi(\tau,x_{\tau},p_{\tau})\nonumber
\\
&&\qquad = e^{-p_{\tau}}
\sum_{i=1}^{2}\bigl(\mathfrak{e}_{i}(x_{\tau})
\phi+a_{i}\phi'\bigr)\sqrt{2}\,dW_{\tau}^{i}
\nonumber\\[-8pt]\\[-8pt]\nonumber
&&\quad\qquad{} + \Biggl(\partial_{\tau}\phi+e^{-2p_{\tau}}K_{h}(x_{\tau})
\phi'+e^{-2p_{\tau}}\Delta_{h}\phi
\\
&&\hspace*{51pt}{} +e^{-2p_{\tau}}\sum_{i=1}^{2}a_{i}^{2}
\phi''+2e^{-2p_{\tau}}\sum
_{i=1}^{2}a_{i}\mathfrak{e}_{i}(x_{\tau})
\phi' \Biggr)\,d\tau,\nonumber
\end{eqnarray}
where all the ``inside'' functions are evaluated at $(\tau, x_{\tau
},p_{\tau})$, $\mathfrak{e}_{i}(x)\phi$ signifies the derivative
[along $\mathfrak{e}_{i}(x)$] with respect to the second variable of
$\phi$, $\partial_{\tau}\phi$ is the derivative with respect to
$\tau$ variable, and the prime is the partial derivative with respect
to $p$. Note that if we let $(x_1,x_2)$ be appropriate normal
coordinates at a point, then applying this to $x_1$, $x_2$, and $p$
shows that, at that point, this agrees with the infinitesimal picture
described above.


We now take a moment to discuss what we mean by asserting the
controlled process arises from the control via the SDEs just mentioned.
We understand these (systems of) SDEs in the weak sense, that is the
choice of driving Brownian motions $(W^1_{\tau},W^2_{\tau})$ is part
of the solution, not prescribed in advance. Of course, for an arbitrary
choice of controls, a solution need not exist, and if it does, it may
not be unique in law. We will have more to say about this later, after
we introduce the target problem.

Now that we have specified the admissible controls $A_{\tau}$ and
described the evolution of controlled process $Y_{\tau}(A)$ that a
choice of control gives rise to, it is time to explain how this gives
rise to a subset of $E$.

%
\begin{definition}
We define the reachable set at a given time $t\in[0,\infty)$, denoted
$V(t)$, to be the set of points in $E$ for which there exists an
admissible control such that the controlled process, started at this
point and with this control, is in $\Gamma(0)$ at time $t$ almost surely.
\end{definition}

We follow Soner and Touzi \cite{SonerTouzi} in calling this the
reachable set, even though it's the set of points you can reach a fixed
target from, not the set of points you can reach from a fixed starting
point. In order for this to be well defined, we need to show that
$V(t)$ does not depend on the initial choice of frame. Suppose $A_{\tau
}$ is a control such that $Y_{\tau}(A)$, started from $y\in E$ with
initial frame $\mathfrak{e}(y)$, hits $\Gamma(0)$ at time $t$ almost
surely [so that $y\in V(t)$]. If $\tilde{\mathfrak{e}}(y)$ is any
other (orthonormal) frame at $y$, then there is some $r\in O(2)$ such
that $\mathfrak{e}(y)=r\tilde{\mathfrak{e}}(y)$. It's clear that
$A_{\tau}r$ is such that $Y_{\tau}(Ar)$, started from $y\in E$ with
initial frame $\tilde{\mathfrak{e}}(y)$, hits $\Gamma(0)$ at time
$t$ almost surely. Thus, a point of $E$ is in the reachable set or not
independent of what frame we use to express the controlled process, and
so the $V(t)$ are well defined.

For a point in the reachable set, we will indicate the control in the
definition by~$\hat{A}$, if necessary indicating the point in $V(t)$
by writing $\hat{A}(x_0,\bar{p}_{0})$ or $\hat{A}(Y_0)$, and call it
a successful control (this seems linguistically more appropriate than
optimal control). In light of the fact that this depends on the initial
choice of frame, a successful control should really be thought of as a
family of controls indexed by $O(2)$. However, since the dependence on
the initial frame is so simple and not our primary focus, we will
generally gloss over this. We will also write $Y_{\tau}(\hat{A})$ as
$\hat{Y}_{\tau}$. Thus, the defining property of a point in $V(t)$
and the associated successful control is that if we start the process
at this point in $V(t)$, then $Y_t(\hat{A})\in\Gamma(0)$ almost\vspace*{1pt}
surely. This necessarily requires that, for a successful control $\hat
{A}$, there exists a solution to equation (\ref{e0}), and thus a
corresponding process $Y_{\tau}(\hat{A})$ for all time $\tau\in
[0,t]$. In particular, one might imagine that some choice of control
gives rise to a solution under which $p_{\tau}$ blows up prior to $t$
($x_{\tau}$ cannot blow up since $M$ is compact), but such a control
cannot be a successful control by definition. The definition does not
require that a successful control gives rise to a solution $Y_{\tau
}(\hat{A})$ which is unique in law, despite the fact that our notation
makes it look as though $Y_{\tau}$ is always determined by $A$. (So it
is conceivable that a successful control might give rise to another
solution $Y^{\prime}_{\tau}$ that does not almost surely hit the
target.) Nonetheless, we will see below that, as long as a smooth
solution to the Ricci flow exists, there is essentially only one choice
of successful control starting from a given point of $V(t)$, that it is
well behaved, and that this control uniquely determines $\hat{Y}_{\tau}$.

Finally, we recall that the stochastic target problem is the
determination of the reachable sets $V(t)$. We note that $V(0)=\Gamma
(0)$; understanding $V(t)$ for positive $t$ and its relationship to
Ricci flow is the topic of the next section. Looking ahead, what we
will prove is that, assuming the Ricci flow has a smooth solution for
some interval of time, that solution agrees with the solution to the
stochastic target problem in the sense that $V(t)=\Gamma(t)$ at all
times in this interval.

Naturally, we have an analogous set-up which we associate with the
normalized Ricci flow. The set of admissible controls remains the same,
but now the controlled process, which we denote $Y^n_{\tau}(A)$ (the
``$n$'' in the superscript standing for ``normalized'') evolves according to
%
\begin{eqnarray}
\label{e1b} dx_{\tau} &=& e^{-p_{\tau}} \Biggl[ \Biggl[\sum
_{i=1}^2 \mathfrak{e}(x_{\tau
})
(e_i)\sqrt{2} \circ dW_{\tau}^i \Biggr] \Biggr],
\nonumber\\[-8pt]\\[-8pt]\nonumber
dp_{\tau} &=& e^{-p_{\tau}} \Biggl[\sum_{i=1}^2
a_i \sqrt{2} \,dW_{\tau}^i \Biggr] +
\bigl(e^{-2p_{\tau}} K_h(x_{\tau})-r \bigr) \,d\tau.
\end{eqnarray}
Note that the only difference from $Y_{\tau}$ is that the drift of
$p_{\tau}$ has an extra term.

We denote the corresponding reachable sets by $V^n(t)$. We also have
the analog of equation (\ref{e2}) where $e^{-2p_{\tau}}K_{h}$
there is replaced by $e^{-2p_{\tau}}K_{h}-r$:
\begin{eqnarray}
\label{e2b}
&& d\phi(\tau,x_{\tau},p_{\tau})\nonumber
\\
&&\qquad = e^{-p_{\tau}}
\sum_{i=1}^{2}\bigl(\mathfrak{e}_{i}(x_{\tau})
\phi+a_{i}\phi'\bigr)\sqrt{2}\,dW_{\tau}^{i}
\nonumber\\[-8pt]\\[-8pt]\nonumber
&&\quad\qquad{} + \Biggl(\partial_{\tau}\phi+ \bigl(e^{-2p_{\tau}}K_{h}(x_{\tau
})-r \bigr)\phi'+e^{-2p_{\tau}}\Delta_{h}
\phi\nonumber
\\
&&\hspace*{51pt}{} +e^{-2p_{\tau}}\sum_{i=1}^{2}a_{i}^{2}
\phi''+2e^{-2p_{\tau}}\sum
_{i=1}^{2}a_{i}\mathfrak{e}_{i}(x_{\tau})
\phi' \Biggr)\,d\tau.\nonumber
\end{eqnarray}


%
\begin{remark}
We want to discuss why we insist that our control $(a_1,a_2)$ is in~$L^{\infty}$. We begin by describing a simpler situation which
illustrates the essential point. Suppose we consider a real-value
controlled process given by
\[
dx_t = a_t \,dW_t,\qquad x_0=1,
\]
where $a_t$ is an adapted real-valued function which serves as the
control. If we consider the goal to be to make the process $x_t$ hit
$0$ in within time 1 (and we stop the process when it hits $0$), then
we would like to assert that this is impossible, because, for instance,
it would violate the martingale property of $x_t$. However, without
some additional restriction on $a_t$, this will not be the case. For
example, consider the following scheme for controlling the process. For
$t\in[0,1/2)$, we let $a$ be the constant such that the process has
probability $1/2$ of hitting $0$ by time $t=1/2$. It is clear that this
is possible, since letting $a$ be constant means that $x_t$ is simply a
time-changed Brownian motion, and we know that Brownian motion almost
surely hits the origin in finite time, no matter where it is started
from. Then at $t=1/2$, the process has hit $0$ and been stopped with
probability $1/2$. If it has not, then $x_{1/2}$ is some positive
value. Again, we can find some constant value for $a$, depending only
on $x_{1/2}$, such that if we let $a_t$ equal that constant for $t\in
[1/2,3/4)$, then the process hits $0$ in that interval of time with
probability $1/2$. Thus, by time $t=3/4$, the process has hit $0$ with
probability $3/4$. Now we can iterate this procedure, at each step
using up half of the remaining time, in order to get $x_t$ to hit $0$
with probability $1$ by time $t=1$. If we do this, the resulting
process $x_t$ will no longer be a martingale on the interval $t\in
[0,1]$ but instead merely a local martingale. Part of the point is that
this is a simple trick. We can think of $a_t$ as determining a
time-change so that $x_t$ is a time-changed Brownian motion, and since
we know Brownian motion hits the origin in finite time, if we are
allowed to speed up time as much as we would like we can simply
compress the entire lifetime of the Brownian motion prior to the first
time it hits the origin into a finite interval.

We now return to the target problem we associate to Ricci flow. In
light of the above, if we assumed only that $(a_1,a_2)$ was adapted, we
could imagine a similar procedure of choosing the control to be very
large so that, from any starting point, we could cause it to hit
$\bar{p}_{t-\tau}$ (this is a moving target, but it varies in a smooth
fashion and stays bounded) by time $t$. Once it hits $\bar{p}_{t-\tau
}$, we could then ``switch'' to the successful control described in the
next section in order to hit $\bar{p}_0$ as time $t$. The result would
be that every point would be in $V(t)$, which is obviously not what we
want. Of course, what we have just described uses a discontinuous
control, but one can imagine smoothing it to get a continuous analogue.
At any rate, the underlying logic of this ``bad'' control justifies our
wish to avoid unbounded controls.

Requiring that $(a_1,a_2)$ be bounded prevents this kind of easy trick
and forces a successful control to respect the geometry of the
situation. Of course, one might imagine that there might be other, less
restrictive, ways to achieve this, such as requiring the controls to be
in some $L^p$-space for finite $p$ or requiring some natural coordinate
to be a martingale, as opposed to merely a local martingale. Indeed, if
one were to extend this stochastic target formulation to include, say,
noncompact surfaces, it seems like some weaker assumption on the
control would be appropriate. However, for the present paper, we have
no need to speculate on what other conditions one might want in other
circumstances.
\end{remark}

%
\begin{remark}
We close this section by noting that, in the case when $(M,h)$ is flat
(and thus either a torus or a Klein bottle), the orthonormal frame
bundle is unnecessary. In particular, uniformization implies that
$(M,h)$ is isometric to $\Re^2$ modulo the action of the group of Deck
transformations $\Lambda$. If we let $x_1$ and $x_2$ be the usual
Euclidean coordinates on $\Re^2$, then $h = dx_1^2 +dx_2^2$ (after
identifying $M$ with $\Re^2/\Lambda$). Further, $(W^1_{\tau},
W^2_{\tau})$ is Brownian motion on $(M,h)$, once we take it modulo
$\Lambda$. In this case, the set of controls are adapted,
time-continuous, bounded maps into $\{(a_1,a_2)\dvtx a_i \in\Re\}$, and
the controlled process simplifies, so that it is given, for both Ricci
and normalized Ricci flow, by the SDE
\[
\lleft[\matrix{ dx_{1,\tau}
\vspace*{3pt}\cr
dx_{2,\tau}
\vspace*{3pt}\cr
dp_{\tau}
} \rright] = \lleft[\matrix{ e^{-p_{\tau}} & 0
\vspace*{3pt}\cr
0 & e^{-p_{\tau}}
\vspace*{3pt}\cr
e^{-p_{\tau}}a_1 & e^{-p_{\tau}}a_2 } \rright] \lleft[\matrix{ \sqrt{2}
\,dW_{\tau}^1
\vspace*{3pt}\cr
\sqrt{2}
\,dW_{\tau}^2 } \rright].
\]
\end{remark}

%
\begin{convention}
Throughout this paper, very often we will have a fixed time $t>0$ so
that the stochastic target problem is defined on $[0,t]$ or the
(normalized) Ricci flow is defined up to time $t$. Since the process
time is always going to be in $[0,t]$, all the stopping times involved
will always be minimized with $t$ so that the stopped process is well defined.

Also, the constants involved in the main estimates may change from line
to line in such a way that they do not depend on time $t$.
\end{convention}

\section{Verification and the connection with Ricci flow}\label{s2}

At this point, we have described a pair of closely related stochastic
target problems, namely the determination of $V(t)$ and $V^n(t)$, which
we associate with Ricci flow and normalized Ricci flow, respectively.
However, we have given no justification for these associations. In the
present section, we prove that, under the assumption that a solution to
the Ricci flow exists, the solution is given by the reachable sets.
This justifies the \hyperref[sec1]{Introduction} of these particular stochastic target
problems in the context of Ricci flow.

Continuing with the notation of the previous section, we suppose that
there is a smooth solution $\bar{p}_t$ to the Ricci flow, that
is, to
equation (\ref{EqnRFp}), with initial condition $\bar{p}_0$ on the
interval $t\in[0,T)$ (where we allow the possibility that $T=\infty
$). At each time $t$, we can associate the solution with a section of
$E$ over $M$ and thus with a sub-manifold of the total space $E$, which
is smooth and intersects each fiber once, transversely. We call the
resulting sub-manifolds $\Gamma(t)$ and note that this extends our
earlier definition of $\Gamma(0)$. Of course, knowing the $\Gamma(t)$
for $t\in[0,T)$ is equivalent to knowing $\bar{p}_t$. Similarly,
suppose there is a smooth solution $\bar{p}^n_t$ to the normalized
Ricci flow, that is, to equation (\ref{EqnNRFp}), with initial
condition $\bar{p}^n_0=\bar{p}_0$ on the interval $t\in[0,T^n)$
[where, for the same manifold $(M,h)$ with the same initial metric
$g_0$, it is not necessarily true that $T$ and $T^n$ are equal]. Then
we have the associated sub-manifolds $\Gamma^n(t)$ of $E$. The
connection between the Ricci flow and normalized Ricci flow (viewed in
this way) and the stochastic target problems introduced above is given
by the following theorem. Note that both this sort of result and the
method of proof mirror that of \cite{SonerTouzi}. The main additional
complication, besides the geometric formalism needed for the general
statement of the target problem, is that the controls are not
restricted to a compact set.

%
\begin{teo}\label{teoVeri}
Let $(M,h)$ be a smooth, compact Riemannian surface with initial metric
$g_0 = e^{2\bar{p}_{0}} h$, as above. Suppose that the Ricci flow has
a smooth solution $\bar{p}_t$ on $t\in[0,T)$. Then $\Gamma(t)=V(t)$
for all $t\in[0,T)$. Similarly, if the normalized Ricci flow has a
smooth solution $\bar{p}^n_t$ on $t\in[0,T^n)$, then $\Gamma
^n(t)=V^n(t)$ for all $t\in[0,T^n)$.
\end{teo}

\begin{pf}
We start with the Ricci flow. We fix some $t\in(0,T)$ and let $\tau$
be the time parameter for the controlled process $Y_{\tau}(A)$, $\tau
\in[0,t]$ (as usual in probabilistic approaches to PDEs, process time
runs ``backward'' compared to PDE time). We consider the square of the
vertical distance between the controlled process $Y_{\tau}$ and
$\Gamma(t-\tau)$. That is, we consider $\eta(x,p,\tau)= (p -
\bar{p}_{t-\tau}(x) )^2$ along the paths of~$Y_{\tau}$, so
that $\eta_{\tau}= (p_{\tau} - \bar{p}_{t-\tau}(x_{\tau})
)^2$.

Actually, we begin by considering a slightly more general quantity. Let
$\xi(x,p,\tau)= p - \bar{p}_{t-\tau}(x)$, and for the moment let
$\phi\dvtx\Re\rightarrow[0,\infty)$ be any smooth function. We wish to consider
$\phi(\xi(x,p,\tau))$; clearly $\eta$ is just the special case
$\phi(z)=z^2$.

We now apply It\^o's formula (\ref{e2}) to $(\phi(\xi))_{\tau}$.
In the following, $\bar{p}$ is always evaluated at time $t-\tau$ and
position $x_{\tau}$, we write $\mathfrak{e}_i$ for $\mathfrak
{e}(x_{\tau})(e_i)$ and we suppress other arguments (such as for the
controls $a_i$) as desired to make things more readable. Then we have
%
\begin{eqnarray}
\label{ep1} d\bigl(\phi(\xi)\bigr)_{\tau} &=& \sqrt{2}
\phi^{\prime} e^{-p_{\tau}} \bigl[ (a_1-\mathfrak{e}_1
\bar{p} ) \,dW_{\tau}^1+ (a_2 -
\mathfrak{e}_2\bar{p} ) \,dW_{\tau}^2 \bigr]\nonumber
\\
&&{} +\sum_{i=1}^2 e^{-2p_{\tau}} \bigl[
\phi^{\prime\prime} (-\mathfrak{e}_i\bar{p} )^2 +
\phi^{\prime} \bigl(-\mathfrak{e}_i^2\bar{p}
\bigr) \bigr] \,d\tau+ \phi^{\prime} \partial_t \bar{p} \,d
\tau
\nonumber\\[-8pt]\\[-8pt]\nonumber
&&{} + e^{-2p_{\tau}} \bigl[\phi^{\prime} K_h +
\phi^{\prime\prime
} \bigl(a_1^2+a_2^2
\bigr) \bigr] \,d\tau
\\
&&{} +2e^{-2p_{\tau}}\phi^{\prime\prime} [-a_1
\mathfrak{e}_1\bar{p}-a_2 \mathfrak{e}_2
\bar{p} ] \,d\tau.\nonumber
\end{eqnarray}
Recall that $\mathfrak{e}_1^2 + \mathfrak{e}_2^2$ is just $\Lap_h$.
Then a little algebra and the fact that $\bar{p}$ satisfies equation
(\ref{EqnRFp}) allows us to simplify this, yielding
%
\begin{eqnarray}
\label{EqnGeneralPhi}
d\bigl(\phi(\xi)\bigr)_{\tau}
&=& \sqrt{2}
\phi^{\prime} e^{-p_{\tau}} \bigl[ (a_1-\mathfrak{e}_1
\bar{p} ) \,dW_{\tau}^1+ (a_2 -
\mathfrak{e}_2\bar{p} ) \,dW_{\tau}^2 \bigr]\nonumber
\\
&&{} + \bigl\{ e^{-2p_{\tau}} \phi^{\prime\prime} \bigl[ (a_1-
\mathfrak{e}_1\bar{p} )^2 + (a_2-
\mathfrak{e}_2\bar{p} )^2 \bigr]
\\
&&\hspace*{18pt}\hspace*{15pt}{} +
\phi^{\prime} \bigl(e^{-2\bar{p}} -e^{-2p_{\tau}} \bigr) (
\Lap_h\bar{p} -K_h ) \bigr\}\,d\tau.\nonumber
\end{eqnarray}
We now return to considering $\eta$. In this case, this equation
specializes to
%
\begin{eqnarray}
\label{EqnEtaSM}
d\eta_{\tau} &=& 2\sqrt{2} (p_{\tau}-\bar{p}
)e^{-p_{\tau
}} \bigl[ (a_1-\mathfrak{e}_1\bar{p}
) \,dW_{\tau}^1 + (a_2-\mathfrak{e}_2
\bar{p} ) \,dW_{\tau}^2 \bigr]\nonumber
\\
&&{} + 2 e^{-2p_{\tau}} \bigl[ (a_1-\mathfrak{e}_1
\bar{p} )^2 + (a_2-\mathfrak{e}_2
\bar{p} )^2 \bigr] \,d\tau
\\
&&\hspace*{15pt}{} +2 (p_{\tau}-\bar{p} )
\bigl(e^{-2\bar{p}}-e^{-2p_{\tau
}} \bigr) (\Lap_h
\bar{p}-K_h ) \,d\tau.\nonumber
\end{eqnarray}

First, we show that any point $(x,\bar{p}_t(x))$ in $\Gamma(t)$
is in
$V(t)$. Obviously, this is true for $t=0$. Now choose $t>0$. We choose
our controls $a_1$ and $a_2$ as follows: for $\tau\in[0,t]$, we let
$a_1$ be $\mathfrak{e}_1 \bar{p}_{t-\tau}(x_{\tau})$ and $a_2$ be
$\mathfrak{e}_2 \bar{p}_{t-\tau}(x_{\tau})$. Thus, our
controls are
Markov with respect to the process' position and the time (and the
``current'' frame, although this is largely just a convention, as
discussed above). Intuitively, all we are doing is trying to cause the
process to be tangent to the evolving solution given by $\bar{p}$. Our
controls are not only Markov in space and time, but they are given by
evaluating smooth functions of space and time (and the lift of
``space'' into the orthonormal frame bundle) along the controlled
process, and thus we know that the system of SDEs for $Y_{\tau}$ has a
unique strong solution. In particular, $Y_{\tau}$ is uniquely
determined by these controls. Using these controls, equation~(\ref
{EqnEtaSM}) simplifies to
\[
d\eta_{\tau} =2 (p_{\tau}-\bar{p} ) \bigl(e^{-2\bar{p}}-e^{-2p_{\tau}}
\bigr) (\Lap_h\bar{p}-K_h ) \,d\tau.
\]

Because $\bar{p}$ is smooth on $M\times[0,T)$ and $M$ is
compact, we
know that both $\bar{p}_{t-\tau}(x)$ and $\Lap_h\bar{p}_{t-\tau
}(x)-K_h$ are bounded on $(x,\tau)\in M\times[0,t]$. Now choose any
$\delta>0$ and let $\theta_{\delta}= \inf\{ \tau\dvtx\eta_{\tau}\geq
\delta\}$ be the first hitting time of $\delta$. Also observe that
both the controlled process $Y_{\tau}=(x_\tau,p_{\tau})$ and $\eta
_{\tau}$ have continuous paths. If we stop our process at $\theta
_{\delta}$, then $p_{\tau}$ is also bounded (this follows from the
fact that $\bar{p}$ is bounded and the definition of $\eta$).
Combining the boundedness of both $\bar{p}$ and $p_{\tau}$ with an
easy estimate for the exponential function, we see that $e^{-2\bar
{p}}-e^{-2p_{\tau}}$ is bounded above and below by a constant multiple
of $\pm(p_{\tau}-\bar{p} )$, respectively. It follows
that (for $\tau\leq\theta_{\delta}$), we have $d\eta_{\tau} \leq
C \eta_{\tau} \,d\tau$, for some positive constant $C$ depending on
$t$, $\delta$, and the bounds mentioned above. Recalling that $\eta
_0=0$, because we start our controlled process on $\Gamma(t)$, and
integrating gives
\[
\eta_{\tau\wedge\theta_{\delta}} \leq C \int_0^{\tau\wedge
\theta_{\delta}}
\eta_s \,ds \qquad\mbox{for }\tau\in[0,t].
\]
Then Gronwall's lemma implies that $\eta_{\tau\wedge\theta_{\delta
}}=0$ for all $\tau\in[0,t]$. Because $\eta_{\tau}$ has continuous
paths, this means that $\theta_{\delta}>t$, and thus we have that
$\eta_{\tau}=0$ for all $\tau\in[0,t]$. In particular, $\eta_t=0$,
and so $Y_t\in\Gamma(0)$. Thus we have shown that $\Gamma(t)\subset V(t)$.

Next, we need to show the opposite inclusion, $V(t)\subset\Gamma(t)$.
Again, this is clear for $t=0$, so we fix some $t\in(0,T)$. We have
some starting point $(\alpha,\beta)\in M\times\Re$, and we assume
that there exists a control $(a_1,a_2)$ such that $Y_{\tau}(a_1,a_2)$
almost surely hits $\Gamma(0)$ at time $\tau=t$.

At this point, we produce a mollified version of $\eta$ by a judicious
choice of $\phi$. In particular, we now let $\phi\dvtx\Re\rightarrow
[0,\infty
)$ be a smooth, symmetric function satisfying the following additional
properties: $\phi$ is nondecreasing on $[0,\infty)$, $\phi(z)=z^2$
in some neighborhood of 0, and $\phi$ is constant on $[A,\infty)$ for
an appropriately chosen constant $A$. It follows that the value of
$\phi$ on $[A,\infty)$ is positive, $\phi$ is 0 only at 0, and all
derivatives of $\phi$ are bounded. If we now let $\hat{\eta
}(x,p,\tau)=\phi(\xi(x,p,\tau))$, then $\hat{\eta}$ is a
mollified version of $\eta$, in the sense that they agree for small
values of $\eta$ but $\hat{\eta}$ is bounded, along with all of its
derivatives.

Let $D(\tau)=\E[\hat{\eta}_{\tau} ]$. Then
equation~(\ref{EqnGeneralPhi}) shows that
\begin{eqnarray}\label{eeqD}
D(\tau) &=& D(0) + \int_0^{\tau} \E
\bigl[ e^{-2p_{s}} \phi^{\prime
\prime} \bigl[ (a_1-
\mathfrak{e}_1\bar{p} )^2 + (a_2-
\mathfrak{e}_2\bar{p} )^2 \bigr]
\nonumber\\[-9pt]\\[-9pt]\nonumber
&&\hspace*{79pt}{}  +
\phi^{\prime} \bigl(e^{-2\bar{p}} -e^{-2p_{s}} \bigr) (\Lap
_h\bar{p} -K_h ) \bigr] \,ds.\nonumber
\end{eqnarray}
Here, of course, the derivatives of $\phi$ are evaluated at $\xi
(x_{s},p_{s},s)$. Note that $\mathfrak{e}_1\bar{p}$, $\mathfrak
{e}_2\bar{p}$ and $\Lap_h\bar{p} -K_h$ are all bounded.
Also, for
small $\xi$ we have that $\phi^{\prime\prime}=2$ and $\phi^{\prime
}=2 (p_{\tau}-\bar{p} )$, and both\vspace*{1pt} of these
derivatives are
bounded for all $\xi$. Moreover, both $e^{-2p_{\tau}} \phi^{\prime
\prime}$ and $\phi^{\prime} (e^{-2\bar{p}} -e^{-2p_{\tau}}
)$ are bounded because the derivatives of $\phi$ are
identically zero for $\xi>A$. In addition, for any two constants
$C_{1},C_{2}\ge0$, there is another constant $C_{3}>0$ such that for
any $\xi\in\Re$,
\[
C_{1}\phi''(\xi)-C_{2}
\phi'(\xi)\xi\ge-C_{3}\phi(\xi).
\]
Notice that as a consequence of (\ref{eeqD}) and the continuity of
the inside functions, $D(\tau)$ is actually differentiable as a
function of $\tau$. In particular, combining this with the above
inequality we deduce that
\[
D'(\tau)\ge-C D(\tau)
\]
for all $\tau\in[0,t]$. This means that $D(\tau)e^{C\tau}$ is
increasing with $\tau$, so
%
\begin{equation}
\label{eeqD2} D(t)e^{Ct}\ge D(0)\ge0.
\end{equation}
By assumption, the controlled process hits $\Gamma(0)$ at time $t$
a.s., and thus $D(t)=0$. This, and the preceding inequality,
immediately lead to $D(0)=0$ which is equivalent to saying that our
initial point $(\alpha,\beta)$ is in $\Gamma(t)$. Thus, we have
proven that $V(t)\subset\Gamma(t)$.

The proof for the normalized Ricci flow is almost identical. With the
appropriate quantities, $\bar{p},p_{\tau},x_{\tau}$ and so on,
equation (\ref{ep1}) becomes
%
\begin{eqnarray}
\label{ep2} d\bigl(\phi(\xi)\bigr)_{\tau}&=& \sqrt{2}
\phi^{\prime} e^{-p_{\tau}} \bigl[ (a_1-\mathfrak{e}_1
\bar{p} ) \,dW_{\tau}^1+ (a_2 -
\mathfrak{e}_2\bar{p} ) \,dW_{\tau}^2 \bigr]\nonumber
\\
&&{} +\sum_{i=1}^2 e^{-2p_{\tau}} \bigl[
\phi^{\prime\prime} (-\mathfrak{e}_i\bar{p} )^2 +
\phi^{\prime} \bigl(-\mathfrak{e}_i^2\bar{p}
\bigr) \bigr] \,d\tau+ \phi^{\prime} \partial_t \bar{p} \,d
\tau
\nonumber\\[-8pt]\\[-8pt]\nonumber
&&{} + e^{-2p_{\tau}} \bigl[\phi^{\prime} K_h-re^{2p_{\tau}}
+ \phi^{\prime\prime} \bigl(a_1^2+a_2^2
\bigr) \bigr] \,d\tau
\\
&&{} +2e^{-2p_{\tau}}\phi^{\prime\prime} [-a_1
\mathfrak{e}_1\bar{p}-a_2 \mathfrak{e}_2
\bar{p} ] \,d\tau\nonumber
\end{eqnarray}
and then from (\ref{EqnNRFp}), we get exactly the same equation from
(\ref{EqnGeneralPhi}), thus the rest of the proof is identical.
\end{pf}

From the point of view of control theory, the above result is a
verification theorem. From the point of view of PDE theory, this can
also be thought of as a uniqueness theorem. In particular, it shows
that smooth solutions to the Ricci flow are unique and we state this in
the following.

%
\begin{Cor}
If there is a (smooth) solution to (normalized) Ricci flow on the time
interval $[0,T)$, then it is unique.
\end{Cor}

It bears repeating that the above relies on already knowing that the
Ricci flow has a smooth solution on some interval; in other words, it
sheds no light on the existence of a solution (to either the Ricci flow
or the control problem). On the other hand, this existence is well
known in the present case. Cao \cite{Cao} and Hamilton \cite{Ham1}
show that, for a smooth, compact initial surface, the Ricci flow always
has a smooth solution on some (nontrivial) interval of time, and the
normalized Ricci flow has a smooth solution for all time. (Of course,
much more can be said, including the relationship between the
normalized and un-normalized flows, but again, this is well known and
can be found in any book on the subject.) For an accessible overview we
refer to \cite{ChKn1}, Chapter~5, which treats the (normalized) Ricci
flow on surfaces.

One additional feature of the successfully controlled process is that
it provides Brownian motion on $M$ under the backward Ricci flow (or
backward normalized Ricci flow, of course), as we now explain. If we
put a smooth family of metrics $g_{\tau}$ on a smooth manifold $M$,
then a process $B_{\tau}$ is a Brownian motion on $(M,g_{\tau})$ if
it solves the martingale problem for the time-inhomogeneous operator
$\Lap_{g_{\tau}}$. Suppose we have a smooth solution to the Ricci
flow, as above, for $t\in[0,T)$, and let $g_t$ be the metric on $M$
corresponding to this solution. Then if we choose a time $t$ [in
$(0,T)$] and point $x_0\in M$, there is a unique point $(x_0,\bar
{p}_{0})$ over $x_0$ (where, of course, we use our standard fiber
coordinate $p$) in $\Gamma(t)=V(t)$. If we now run our successfully
controlled process $Y_{\tau}=(x_{\tau},p_\tau)$ starting from this
point, we know that it is on $\Gamma(t-\tau)$ for all $\tau\in
[0,t)$, or equivalently that $p_{\tau}=\bar{p}_{t-\tau}(x_{\tau})$,
for all $\tau\in[0,t]$ almost surely. Then looking at $x_{\tau}$
(which is just the $M$-marginal) and recalling that $g_t= e^{2
\bar{p}_{t}}h$, a little thought shows that $x_{\tau}$ is a Brownian
motion on $(M,g_{t-\tau})$ for $\tau\in[0,t]$. That ``process time''
runs backward compared to ``PDE'' time, which manifests itself in the
$t-\tau$ parameter (with $t$ fixed and $\tau$ increasing) for the
metric $g$, explains why we get Brownian motion on $M$ under backward
Ricci flow, as opposed to just Ricci flow.

For clarity, let us temporarily denote $x_{\tau}$ under the successful
control as $\hat{x}_{\tau}$. Then recognizing it as Brownian motion
under backward Ricci flow gives a way of representing the solution to
the Ricci flow (or normalized Ricci flow) that looks more like the
usual representations for parabolic (linear) PDEs. In the special case
when $h$ is flat, normalized and un-normalized Ricci flow are the same,
and we see that $p_{\tau}$ is a martingale. Further, we have that
%
\begin{equation}
\label{e5} \bar{p}_t(x_0) = \E^{x_0,t}
\bigl[\bar{p}_0 (\hat{x}_{t} ) \bigr],
\end{equation}
where the expectation is taken with respect to the successfully
controlled process started from $(x_0,\bar{p}_t(x_0))$ and run until
$\tau=t$. This is analogous to solving the heat equation with some
initial condition by running Brownian motion and then using it to
average the initial condition. The difference is that, for the heat
equation, we can construct Brownian motion (or more analytically, the
heat kernel) without already having a solution to the heat equation
with our initial data. This is because Brownian motion (or the heat
kernel) does not depend on the initial data, and so we can use it to
solve the heat equation in the first place. All of this is a
manifestation of the linearity of the heat equation. In the case of
Ricci flow, we need to know $\hat{p}_{\tau}$ in order to determine
$\hat{x}_{\tau}$ (or more accurately, these two are intertwined by
the system of SDEs they solve), so we cannot first determine $\hat
{x}_{\tau}$ and then use it in the above to solve the Ricci flow.

Also, we can now say a bit more about the recent work of \cite{KACP}
and \cite{ACT}. They give a lift of Brownian motion on a manifold with
time-dependent metric to the frame bundle which gives the parallel
transport along the Brownian paths. They then introduce a notion of
damped parallel transport which, under the Ricci flow (but not the
normalized flow), becomes an isometry as well. This damped parallel
transport can be used to produce martingales from solutions to heat
problems under the Ricci flow. In our notation, $x_{\tau}$ is the
Brownian motion with respect to a time-dependent metric (with an
additional factor of $\sqrt{2}$ to get the normalization right, of
course), and $\{ e^{-p_{\tau}}\mathfrak{e}(x_{\tau})(e_1),
e^{-p_{\tau}}\mathfrak{e}(x_{\tau})(e_2) \}$ (which is an
orthonormal frame for the time-varying metric) gives the parallel
transport along the Brownian path $x_{\tau}$.

\section{The blow ups of the Ricci flow for the case of positive or negative Euler characteristic}\label{s3}

This section is dedicated to showing that in the case of the
(unnormalized) Ricci flow, there are blow ups either in finite or
infinite time if the curvature of the reference metric $K_{h}$, is
either always positive or always negative.

Assume now that the Ricci flow has a smooth solution defined on the
time interval $[0,T)$. Then, from Theorem \ref{teoVeri}, we learn
that for any fixed time $t\in[0,T)$, $p_{\tau}=\bar{p}_{t-\tau
}(x_{\tau})$ where $(x_{\tau},p_{\tau})$ is the solution to (\ref
{e0}) with the initial conditions $(x,\bar{p}_{0}(x))$. On the other
hand, taking a smooth function $\phi\dvtx[0,t]\times\Re\to\Re$ in
(\ref{e2}), we obtain that
\begin{eqnarray*}
d\phi(\tau,p_{\tau})&=&e^{-p_{\tau}}\phi'(p_{\tau})
\sum_{i=1}^{2}a_{i}
\sqrt{2}\,dW_{\tau}^{i}
\\
&&{} + \Biggl[ \partial_{\tau}\phi(
\tau,p_{\tau})+e^{-2p_{\tau}}\Biggl(\phi'(
\tau,p_{\tau})K_{h}(x_{\tau
})+\phi''(
\tau,p_{\tau})\sum_{i=1}^{2}a_{i}^{2}
\Biggr) \Biggr]\,d\tau.
\end{eqnarray*}
Since the successful control is given by $a_{i}=\mathfrak
{e}_{i}\bar{p}_{t-\tau}$, we get
\[
\sum_{i=1}^{2}a_{i}^{2}=
\bigl\llvert\nabla\bar{p}_{t-\tau}(x_{\tau
})\bigr\rrvert
^{2},
\]
and this means that
\begin{eqnarray*}
&& \phi(\tau,p_{\tau})-\int_{0}^{\tau} \bigl[
\partial_{\tau}\phi(\sigma,p_{\sigma})
\\
&&\hspace*{67pt}{} +e^{-2p_{\sigma}} \bigl(
\phi'(\sigma,p_{\sigma
})K_{h}(x_{\sigma})+
\phi''(\sigma,p_{\sigma})\bigl\llvert\nabla
\bar{p}_{t-\sigma}(x_{\sigma})\bigr\rrvert^{2} \bigr)
\bigr]\,d\sigma
\end{eqnarray*}
is a martingale. In particular, taking expectation at times $\tau=0$
and $\tau=t$ and using $p_{\tau}=\bar{p}_{t-\tau}(x_{\tau})$, yields
%
\begin{eqnarray}
\label{e3} \phi\bigl(0,\bar{p}_{t}(x)\bigr)&=&\E^{(x,t)}
\bigl[\phi\bigl(t,\bar{p}_{0}(x_{t})\bigr)\bigr]\nonumber
\\
&&{}-\int_{0}^{t}\E^{(x,t)} \bigl[
\partial_{t}\phi(\sigma,p_{\sigma})
\nonumber\\[-8pt]\\[-8pt]\nonumber
&&\hspace*{54pt}{}+e^{-2p_{\sigma}} \bigl(
\phi'(\sigma,p_{\sigma})K_{h}(x_{\sigma
})
\\
&&\hspace*{96pt}{}
+ \phi''(\sigma,p_{\sigma})\bigl\llvert\nabla
\bar{p}_{t-\sigma
}(x_{\sigma
})\bigr\rrvert^{2} \bigr)
\bigr]\,d\sigma.\nonumber
\end{eqnarray}
There are two obvious obstructions stemming from this formula. The
first one is that if $K_{h}(x)>0$ for all $x\in M$, then taking $\phi
(\tau,p)=e^{2p}$, the above formula (\ref{e3}) implies
\begin{eqnarray*}
e^{2\bar{p}_{t}(x)} &=& \E^{(x,t)}\bigl[e^{2\bar{p}_{0}(x_{t})}\bigr
]-2\int
_{0}^{t}\E^{(x,t)} \bigl[
K_{h}(x_{\sigma})+ 2\bigl\llvert\nabla\bar{p}_{t-\sigma}(x_{\sigma
})\bigr\rrvert^{2} \bigr]\,d\sigma
\\
&\le& \E^{(x,t)}\bigl[e^{2\bar{p}_{0}(x_{t})}\bigr]-2\int_{0}^{t}
\E^{(x,t)}\bigl[ K_{h}(x_{\sigma})\bigr]\,d\sigma
\end{eqnarray*}
and thus, upon denoting the uniform norm by $\llvert\cdot\rrvert_{u}$
and taking
$K_{0}=\inf_{x\in M} K_{h}(x)$,
\[
e^{2\bar{p}_{t}(x)}\le e^{2\llvert\bar{p}_{0}\rrvert
_{u}}-2t K_{0}.
\]
As this is true for any $t\in[0,T)$, the extinction time of the Ricci
flow is finite and is certainly at most $e^{2\llvert\bar{p}_0\rrvert
_{u}}/(2K_{0})$. Therefore, in the case of positive
curvature the
flow develops singularities in finite time.

On the other hand, if the curvature is negative ($K_{h}<0$ on $M$),
then there are some constants $C_{1},C_{2}>0$ such that
\[
\bar{p}_{t}(x)\ge\log(C_{1}t+1)-C_{2}\qquad
\mbox{for all }x\in M\mbox{ and }t\ge0.
\]
To see this, take $K_{0}=\inf_{x\in M} -K_{h}(x)>0$, thus $K_{h}(x)\le
-K_{0}<0$ and then consider $\phi(\tau,p)=p$ in (\ref{e3}) to
deduce that
\[
\bar{p}_{t}(x)=\E^{(x,t)}\bigl[\bar{p}_0(x_{t})
\bigr]-\int_{0}^{t}\E^{(x,t)}
\bigl[e^{-2p_{\sigma}}K_{h}(x_{\sigma})\bigr]\,d\sigma\ge\inf
_{x\in
M}\bar{p}_0
\]
which means that $\bar{p}_{t}(x)$ is bounded below uniformly in $t\ge
0$ and $x\in M$. Now consider the test function $\phi(\tau,p)=\exp
(\alpha(t-\tau-\frac{1}{2K_{0}}e^{2p}) )$. Since $\bar{p}_{t}(x)$ is
bounded below, this implies that for large enough
$\alpha$, $\phi''(\sigma, p_{\sigma})\ge0$. On the other hand,
$\partial_{\tau}\phi(\sigma,p)-K_{0}e^{-2p}\phi'(\sigma, p)=0$,
and this combined with the preceding and the fact that $\phi'$ is
negative leads to
\[
\phi\bigl(0,\bar{p}_{t}(x)\bigr)\le\E^{t,x}\bigl[\phi
\bigl(t,\bar{p}_0(x_{t})\bigr)\bigr]\le1,
\]
which means that $\bar{p}_{t}(x)\ge\frac{1}{2}\log(2K_{0}t)$ for
any $t>0$ for which $\bar{p}_{t}$ exists. In particular, this shows
that either the flow ceases to exist after a finite time, or, if it
does exist for all times, $\bar{p}_{t}(x)$ goes to infinity uniformly
over $x\in M$. The moral is that we cannot expect the Ricci flow to
converge as the time approaches either the extinction time or infinity.

For the flat case, since the curvature is $0$, the normalized and the
unnormalized Ricci flows are the same, and thus we will treat this case
as the normalized Ricci flow.

%
\begin{remark} The blow up in the negative case does not take place in
finite time but this requires more arguments which we do not provide here.
\end{remark}

\section{Time-dependent a priori bounds for Ricci flow}\label{s4}

We now turn our attention to using the stochastic target representation
for the normalized Ricci flow to derive (more accurately, of course, to
re-derive) geometric facts about the flow. We will always work with the
case where the reference metric $h$ has constant curvature. By
uniformization, this is no loss of generality, and it simplifies the
analysis considerably. After a preliminary rescaling, we can assume
that this constant curvature is either $1$, $0$, or $-1$. Further, we
can rescale the initial metric $g_0$ so that it has the same area as
$h$. Thus, without loss of generality, we are in one of three cases (by
the Gauss--Bonnet theorem). First, if the Euler characteristic of $M$
is positive, we have that $K_h \equiv r\equiv1$. If the Euler
characteristic of $M$ is zero, we have that $K_h\equiv r\equiv0$.
Finally, when the Euler characteristic of $M$ is negative we have that
$K_h\equiv r\equiv-1$. The bounds we have in mind are similar in all
three cases, although the differences in sign of $K_h$ result in
important differences.

We call these bounds ``a priori'' because they do not depend on the
structure of the reachable set. We elaborate on this after Theorem \ref
{teoAprioriBounds}.

We have one more comment about notations before we begin. Because we
will be concerned with the normalized Ricci flow for the rest of the
paper, we drop the ``n'' superscripts. Thus, for instance, we let $\bar
{p}_t$ denote a solution to the normalized Ricci flow, unless otherwise
indicated.

The interesting feature of choosing $h$ to be a metric of constant
curvature is that the drift of the SDE satisfied by $p_{\tau}$ does
not depend on $x_{\tau}$ (although the target always does, except in
trivial cases). In particular, we have the following three cases:
%
\begin{eqnarray}
\label{Eqn3cases} r&=&1\dvt \qquad dp_{\tau} = e^{-p_{\tau}} \Biggl
[\sum
_{i=1}^2 a_i \sqrt{2}
\,dW_{\tau}^i \Biggr] + \bigl(e^{-2p_{\tau}}-1 \bigr) \,d\tau,\nonumber
\\
r&=&0\dvt \qquad dp_{\tau} = e^{-p_{\tau}} \Biggl[\sum
_{i=1}^2 a_i \sqrt{2}
\,dW_{\tau}^i \Biggr],
\\
r&=&-1\dvt \qquad dp_{\tau} = e^{-p_{\tau}} \Biggl[\sum
_{i=1}^2 a_i \sqrt{2}
\,dW_{\tau}^i \Biggr] + \bigl(1-e^{-2p_{\tau}} \bigr) \,d\tau.\nonumber
\end{eqnarray}

In general, the stochastic target problem for the normalized Ricci flow
(and also the Ricci flow itself) gives an equation of the form
%
\begin{equation}
\label{eap3} dp_{\tau}=e^{-p_{\tau}} \Biggl[\sum
_{i=1}^{2}a_{i}\,dW^{i}_{\tau
}
\Biggr]+U_{\tau}(p_{\tau}) \,d\tau,
\end{equation}
where the controls $a_{i}$, $i=1,2$ are bounded and chosen such that
$p_{t}$ is almost surely on $M_{0}$, the section corresponding to $\bar
{p}_{0}$ in the bundle $M\times\Re$. In the case at hand, we assume
that $U_{\tau}(p)$ is a function $U\dvtx[0,t]\times\Re\to\Re$ which is
uniformly locally Lipschitz in the second variable, that is, for any
$L>0$ there is a constant $C_{L}$ with $\llvert U_{\tau
}(p)-U_{\tau}(q)\rrvert\le
C_{L}\llvert p-q\rrvert$ for all $\tau\in[0,t]$ and $p,q\in[-L,L]$.

The basic point is that there are natural barriers for $p_{\tau}$
given in terms of equation~(\ref{eap3}) where the martingale part is
set to be equal to $0$. To be precise, we define a barrier as a
solution $q_{\tau}$ to the ODE
%
\begin{equation}
\label{eap4} dq_{\tau}=U_{\tau}(q_{\tau}) \,d\tau.
\end{equation}
In this framework, we have a general result as follows.

%
\begin{Lemma}
Assume that $p_{\tau}$ and $q_{\tau}$ are solutions to (\ref
{eap3}) and (\ref{eap4}), respectively, for $\tau\in[0,t]$ with
$U$ a uniformly locally Lipschitz function in the second variable on
$[0,t]\times\Re$.

If at any time $\tau_{1}\in[0,t)$, $p_{\tau_{1}}< q_{\tau_{1}}$
with positive probability, then at any later time $\tau_{2}\in(\tau
_{1},t]$, $p_{\tau_{2}}< q_{\tau_{2}}$ with positive probability.

Similarly, if at any time $\tau_{1}\in[0,t)$, $p_{\tau_{1}}> q_{\tau
_{1}}$ with positive probability, then at any later time $\tau_{2}\in
(\tau_{1},t]$, $p_{\tau_{2}}> q_{\tau_{2}}$ with positive probability.
\end{Lemma}

\begin{pf} The proof is a basic application of stopping time and
Gronwall-type argument. We will prove only the first part, the second
one being similar.

So, assume that $q_{\tau_{1}}>p_{\tau_{1}}$ with positive probability
and, therefore, that we can choose a constant $L>0$ such that $L\ge
q_{\tau_{1}}-p_{\tau_{1}}>1/L$ with positive probability. We further
take $L$ large enough so that $\llvert q_{\tau}\rrvert\le L$
for all $\tau\in[0,t]$.

Now, for any smooth function $\eta\dvtx\Re\to\Re$, we have
%
\begin{eqnarray}\label{eap5}
\eta(q_{\tau}-p_{\tau}) &=& \eta(q_{\tau_{1}}-p_{\tau
_{1}})+M_{\tau
}\nonumber
\\
&&{} +
\int_{\tau_{1}}^{\tau} \bigl(e^{-2p_{s}}\eta
''(q_{s}-p_{s})
\bigl(a_{1}^{2}(s)+a_{2}^{2}(s)\bigr)
\\
&&\hspace*{33pt}{} +
\eta'(q_{s}-p_{s}) \bigl(U_{s}(q_{s})-U_{s}(p_{s})
\bigr) \bigr)\,ds,\nonumber
\end{eqnarray}
where $M_{\tau}$ is a martingale with $M(\tau_{1})=0$. Further, we
choose the function $\eta(\xi)$ such that it is nondecreasing, equal
to $0$ for $\xi\le0$, equal to $1$ for $\xi\ge2L$ and $\eta(\xi
)=\xi^{2}$ for small $\xi\ge0$.

Next, we define the stopping time $\sigma=\inf\{u\ge\tau_{1}\dvtx
p_{u}\ge q_{u} \}\wedge t$. With this setup, we will denote for
simplicity $\eta_{\tau}=\eta(q_{\tau}-p_{\tau})$, $\eta'_{\tau
}=\eta'(q_{\tau}-p_{\tau})$ and $\eta''_{\tau}=\eta''(q_{\tau
}-p_{\tau})$. Furthermore, from (\ref{eap5}),
%
\begin{eqnarray}\label{eeqEta}
\E[\eta_{\tau\wedge\sigma}]&=&\E[\eta_{\tau_{1}}]
\nonumber
\\
&&{} +\int
_{0}^{\tau
}\E\bigl[\mathbh{1}_{[\tau_{1},\sigma]}(s)
\bigl(e^{-2p_{s}} \eta''_{s}
\bigl(a_{1}^{2}(s)+a_{2}^{2}(s)\bigr)
\\
&&\hspace*{85pt}{} +
\eta'_{s}\bigl(U_{s}(q_{s})-U_{s}(p_{s})
\bigr) \bigr)\bigr] \,ds.\nonumber
\end{eqnarray}
Since $q_{s}$ remains bounded on $[\tau_{1},\tau_{2}]$ and $\eta'$
has compact support, combined with the property that $U_{\tau}$ is
uniformly Lipschitz in the second variable on compact intervals, we can
find a constant $C>0$, such that
\[
\eta'_{s}\bigl(U_{s}(q_{s})-U_{s}(p_{s})
\bigr)\ge-C\eta'_{s}(q_{s}-p_{s}).
\]
This, the choice of our function $\eta$, the fact that the controls
$a_{i}$, $i=1,2$ are bounded, and that $q_{s}$ is bounded, yield, in
the first place, that $e^{-2p_{s}}\eta_{s}''$ is bounded, and also
that for some constant $C>0$,
%
\begin{equation}
\label{eeqEta2} \bigl(e^{-2p_{s}}\eta''_{s}
\bigl(a_{1}^{2}(s)+a_{2}^{2}(s)\bigr)+
\eta'_{s}\bigl(U_{s}(q_{s})-U_{s}(p_{s})
\bigr) \bigr)\ge-C\eta_{s}.
\end{equation}

To check this, one can reason as follows. For $q_{s}\le p_{s}$, both
sides are $0$. For $\epsilon> q_{s}-p_{s}>0$ with small $\epsilon$,
the first term is nonnegative and the second one is bounded below by
$-C(q_{s}-p_{s})^{2}$ which is again a constant times $\eta_{s}$. For
$q_{s}-p_{s}>\epsilon$, the inequality follows easily as the left-hand
side is bounded below by some negative constant and $\eta_{s}$ is
certainly bounded below by $\epsilon^{2}$.

The next step is similar to the passage from (\ref{eeqD}) to (\ref
{eeqD2}). To wit, notice that, from (\ref{eeqEta}), $u(\tau)=\E
[\eta_{\tau\wedge\sigma}]$ is a continuous and differentiable
function of $\tau$ for $\tau\in[\tau_1,\tau_2]$. Combining this
with (\ref{eeqEta2}) leads to
\[
u'(\tau)\ge-C\E\bigl[\mathbh{1}_{[\tau_{1},\sigma]}(\tau)
\eta_{\tau}\bigr].
\]

Since $\sigma$ is the first time $p_{s}= q_{s}$, it follows that,
$\mathbh{1}_{[\tau_{1},\sigma]}(\tau)\eta_{\tau}=\eta_{\tau
\wedge\sigma}$, consequently,
\[
u'(\tau)\ge-Cu(\tau),
\]
which results with
\[
u(\tau)e^{C(\tau-\tau_1)}\ge u(\tau_1)>0
\]
or equivalently,
\[
\E[\eta_{\tau\wedge\sigma}] e^{C(\tau-\tau_{1})}\ge\E[\eta_{\tau_{1}}]>0.
\]
The hypothesis $q_{\tau_{1}}>p_{\tau_{1}}$ with positive probability
is translated into positivity of $\E[\eta_{\tau_{1}}]$.
For $\tau=\tau_{2}$ we obtain $\E[\eta_{\tau_{2}\wedge\sigma
}]=\E[\eta_{\tau_{2}},\sigma>\tau_{2}]>0$ and, therefore, we
conclude that $\{\sigma>\tau_{2} \}$ has positive probability; stated
otherwise, the probability that $q_{\tau_{2}}>p_{\tau_{2}}$ is positive.

One technical word is in place here. Namely, the definition from (\ref
{eap3}) is in the sense of local martingales, but during the proof we
look at $\eta(q_{\tau}-p_{\tau})$ and this is actually a
semi-martingale in the sense that is a sum of martingale and a bounded
variation process, not merely a sum of a local martingale and a locally
bounded variation. This is indeed due to the boundedness and continuity
of the quantities involved, namely $e^{-p_{s}}\eta_{s}'$,
$e^{-2p_{s}}\eta_{s}''$ and the controls $a_{i}$, $i=1,2$.
\end{pf}

Next, we solve equation (\ref{eap4}) for each of the three cases
described in equation~(\ref{Eqn3cases}) (this is straightforward, as
the resulting ODEs are separable). For ease of reference, we will label
the resulting equations as $B^K_c(\tau)$ with super- and sub-scripts
indicating relevant parameters. In the case $r=1$, we have that
\[
B^1_c(\tau) = \tfrac{1}{2}\log
\bigl(1-ce^{-2\tau} \bigr)\qquad\mbox{for some constant }c\in(-\infty, 1).
\]
The choice of $c$ allows any initial condition. Note that $c=0$ gives
the constant solution $B^1_0(\tau)\equiv0$. For any $c$, as $\tau
\rightarrow
\infty$, we see that $B^1_c(\tau)\rightarrow0$. The case $r=0$ gives
\[
B^0_c(\tau) = c \qquad\mbox{for some constant }c\in
\Re.
\]
Obviously, the choice of $c$ allows any initial condition. (This is
perhaps a bit pedantic, but we include it for the sake of
completeness.) Finally, $r=-1$ gives
\[
B^{-1}_c(\tau)=\tfrac{1}{2}\log
\bigl(1-ce^{2\tau} \bigr)\qquad\mbox{for some constant }c\in(-\infty, 1).
\]
Again, the choice of $c$ allows any initial condition, and $c=0$ gives
the constant solution $B^{-1}_c(\tau)\equiv0$. This time, though, if
$c\neq0$, then the solution heads to $\pm\infty$ as $\tau$
increases (in finite time for negative initial condition, and as $\tau
\rightarrow\infty$ for positive initial condition).

Continuing, we want to use the previous lemma and a judicious choice of
the parameter $c$ to bound the reachable set at time $t$. Recall that
$\bar{p}_0$ gives the initial metric $g_0$ and serves as the
target in
the target problem [and which as a section we write as $\Gamma(0)$].
The assumption that $g_0$ and $h$ have the same area implies that $\max
_{x\in M} \bar{p}_0(x) =\alpha\geq0$ and that $\min_{x\in M}
\bar{p}_0(x) =\beta\leq0$. Further, if either $\alpha$ or
$\beta$ is
zero then both are, meaning that $\bar{p}_0\equiv0$ and $g_0$ is
just $h$.

The logic of the proof of the following theorem explains why solutions
$q_{\tau}$ of equation (\ref{eap4}) are called barriers, in this context.

%
\begin{teo}\label{teoAprioriBounds}
Consider the target problem (for the normalized Ricci flow) where $h$
corresponds to one of the three constant curvature cases as discussed
above (and with $\alpha$ and $\beta$ as just described). For any
$t\geq0$, we have that
\[
\sup_{(x,p)\in V(t)} p \leq\cases{ \frac{1}{2}\log
\bigl(1-e^{-2t} \bigl(1-e^{2\alpha} \bigr) \bigr), &\quad if $r=-1$,
\vspace*{3pt}\cr
\alpha, &\quad if $r=0$,
\vspace*{3pt}\cr
\frac{1}{2}\log\bigl(1-e^{2t}
\bigl(1-e^{2\alpha} \bigr) \bigr), &\quad if $r=1$,}
\]
and
\[
\inf_{(x,p)\in V(t)} p \geq\cases{ \frac{1}{2}\log
\bigl(1-e^{-2t} \bigl(1-e^{2\beta} \bigr) \bigr), &\quad if $r=-1$,
\vspace*{3pt}\cr
\beta, &\quad if $r=0$,
\vspace*{3pt}\cr
\frac{1}{2}\log\bigl(1-e^{2t}
\bigl(1-e^{2\beta} \bigr) \bigr), &\quad if $r=1$ and $t<-\frac{1}{2}
\log\bigl(1-e^{2\beta} \bigr)$.}
\]
[If $\beta=0$, we set $-\frac{1}{2}\log(1-e^{2\beta}
)=\infty$.]
\end{teo}

\begin{pf} We start with the upper bound in the $r=-1$ case. We
consider some fixed but arbitrary $t\geq0$. Let $c^{\prime
}=e^{-2t} (1-e^{2\alpha} )$. Then
\[
B^{-1}_{c^{\prime}}(t) = \alpha\quad\mbox{and}\quad
B^{-1}_{c^{\prime}}(0) = \tfrac{1}{2}\log\bigl(1-e^{-2t}
\bigl(1-e^{2\alpha} \bigr) \bigr).
\]
Thus, by the previous lemma, if we start from a point $(x_0,\bar
{p}_{0})$ with $\bar{p}_{0}>B^{-1}_{c^{\prime}}(0)$, we have that
$p_{t}>B^{-1}_{c^{\prime}}(t) = \alpha$ with positive probability
(for any controls). By the definition of $\alpha$, this means that
$p_t$ is not in the target with positive probability. Since this holds
for any controls, it follows that $(x_0,\bar{p}_{0})$ is not in the
reachable set at time $t$, which we recall we denote $V(t)$. This
implies the upper bound on $\sup_{(x,p)\in V(t)} p$ given in the theorem.

For the lower bound in the $r=-1$ case, consider $c^{\prime}=
e^{-2t} (1-e^{2\beta} )$. Then
\[
B^{-1}_{c^{\prime}}(t) = \beta\quad\mbox{and}\quad
B^{-1}_{c^{\prime}}(0) = \tfrac{1}{2}\log\bigl(1-e^{-2t}
\bigl(1-e^{2\beta} \bigr) \bigr).
\]
Analogously to the argument for the upper bound, the previous lemma
implies that no point $(x_0,\bar{p}_{0})$ with $\bar
{p}_{0}<B^{-1}_{c^{\prime}}(0)$ can be in $V(t)$. This implies the
desired lower bound.

For the $r=0$ case, analogous arguments apply, using $c^{\prime}=
\alpha$ for the upper bound and $c^{\prime}= \beta$ for the lower bound.

Finally, we consider the $r=1$ case. The upper bound is proven just as
in the $K=-1$ case, using $c^{\prime}=e^{2t} (1-e^{2\alpha
} )$. The proof of the lower bound is similar, except that if
$t\geq-\frac{1}{2}\log(1-e^{2\beta} )$, we have that
$B_c^{1}(t)>\beta$ for any choice of $c\in(-\infty,1)$. Thus, these
arguments do not produce any lower bound for $\inf_{(x,p)\in V(t)} p$
in this case. On the other hand, if $t<-\frac{1}{2}\log
(1-e^{2\beta} )$, we can let $c^{\prime}=e^{2t}
(1-e^{2\beta} )$ and argue just as before.
\end{pf}

In light of the verification theorem, these conclusions can be restated
in terms of $\bar{p}_t$. Namely, we can replace $\sup_{(x,p)\in V(t)}
p$ in the above theorem with $\max_{x\in M} \bar{p}_t(x)$ and
$\inf_{(x,p)\in V(t)} p$ with $\min_{x\in M} \bar{p}_t(x)$. Nonetheless,
there is a reason to state the theorem as above. Suppose we consider
the same target problem (or problems, since there are three cases),
except that now we allow the target to be any (nonempty) closed set
$\Gamma$ such that $\max_{\Gamma} p=\alpha\geq0$ and $\max_{\Gamma}
p=\beta\leq0$, rather than just a smooth section
corresponding to a metric $g_0$ on $M$. Then we can still ask about the
reachable set at time $t\geq0$. Assuming that it is nonempty, the
bounds in the above theorem still hold (with the same proofs). This
shows that these bounds do not depend on the verification theorem and
the resulting connection with PDEs, or on the structure of the
reachable set, such as its smoothness or whether it is a section.
(Moreover, similar methods could be employed even if $\alpha$ and
$\beta$ were not assumed to be nonnegative and nonpositive,
resp.) It is this sense in which we refer to them as ``a priori
bounds.'' Of course, it is likely that these bounds are only \emph
{interesting} in light of their connection to the Ricci flow, as given
by the verification theorem.

We close this section with some easy observations about this theorem.
First of all, if $\alpha=0$, then $\sup_{(x,p)\in V(t)} p=0$ for all
$t\geq0$, and this holds in all three cases. Similarly, if $\beta=0$,
then $\inf_{(x,p)\in V(t)} p=0$ for all $t\geq0$, in all three cases.
Since one of $\alpha$ or $\beta$ being zero implies that both are, we
conclude that if either $\alpha$ or $\beta$ is zero, the reachable
set only contains points with $p=0$. On the other hand, every point
with $p=0$ will clearly be in the reachable set (just let the controls
be identically zero). Thus, we will have $V(t)=\{p\equiv0\}$ for all
$t\geq0$. This corresponds to the basic fact that if $g_0$ is already
a metric of constant curvature, then it is stationary under the
normalized Ricci flow.

In the case when $\alpha$ and $\beta$ are not zero, we see much
different behavior for the cases of the three different curvatures. For
$r=-1$, the bounds improve as $t$ increases, which we will see makes
this the easiest case to deal with. For $r=0$, the bounds are constant.
Finally, for $r=1$, the bounds get worse as $t$ increases, and the
lower bound even ceases to exist in finite time. This corresponds to
the well-known observation that the case of the sphere (or projective
space) is the hardest case to handle for Ricci flow on compact surfaces.

%
\begin{remark} It is worth pointing out that the above argument from
Theorem~\ref{teoAprioriBounds} is overkill in the $r=0$ case, since
then the result follows directly from the fact that $p_{\tau}$ is a
martingale and martingales have constant expectation.
\end{remark}

We finish this discussion with the following useful corollary which
plays an important role later on.

\begin{Cor}\label{c22}
For the case of $r=-1$, or equivalently, the case $\chi(M)<0$, the
solution $\bar{p}_{t}$ of the normalized Ricci flow converges to $0$
uniformly in the \mbox{$C^{0}$-}norm exponentially fast as $t\to\infty$.
\end{Cor}

The same arguments work in the case of unnormalized Ricci flow. We
record this here as follows.

%
\begin{teo}
For the unnormalized Ricci flow, as long as the stochastic target is
well defined up to time $t$,
\[
\sup_{(x,p)\in V(t)} p \leq\cases{ \frac{1}{2}\log
\bigl(e^{2\alpha} +t \bigr), &\quad if $r=-1$,
\vspace*{3pt}\cr
\frac{1}{2}\log
\bigl(e^{2\alpha}-t \bigr), &\quad if $r=1$ and $t<e^{2\alpha}$}
\]
and
\[
\inf_{(x,p)\in V(t)} p \ge\cases{ \frac{1}{2}\log
\bigl(e^{2\beta} + t \bigr), &\quad if $r=-1$,
\vspace*{3pt}\cr
\frac{1}{2}\log
\bigl(e^{2\beta}-t \bigr), &\quad if $r=1$ and $t<e^{2\beta}$.}
\]
\end{teo}

The only thing we should point out here is that there is a blow-up in
finite time for the case of $r=1$ and there is also a blow up in finite
or infinite time for the case of $r=-1$. This recovers the blow-up
results in the previous section, only this time we used uniformization.

%
\begin{remark} This theorem shows that for the unnormalized Ricci flow,
in the negative curvature case, the flow does not blow up in finite
time, at least in the $C^{0}$ topology. This is already a good
indication that the solution is defined for all times and corroborated
with the above theorem shows that the flow blows up at infinity. Thus,
this result is probably a better result (in the case of negative
constant curvature case) as the one obtained in Section~\ref{s3}.
\end{remark}

\section{Mirror coupling}\label{s5}

For the remainder of the paper, we assume that we have a smooth initial
metric and a smooth solution to the normalized Ricci flow for all time
(which we do since the initial conditions are smooth on a compact
surface). We are interested in studying the convergence to the constant
curvature limit according to the stochastic framework we have been developing.

We consider the cases of zero Euler characteristic and of negative
Euler characteristic, and we work relative to the underlying metric of
constant curvature, as in the previous section. The positive Euler
characteristic case (the sphere or projective plane) is well known to
be more difficult. This is largely due to the fact that there are many
constant curvature metrics in any given conformal class, so that it is
not clear in advance which one will be the limiting metric under
normalized Ricci flow (this is related to the issue of solitons). As a
result, we do not pursue this case.

We are assuming that we have a smooth solution to the normalized Ricci
flow for all time. This means that the reachable set is always a smooth
hypersurface transverse to the vertical fibers. From now on, we are
only interested in the successfully controlled process, so for
notational simplicity we will\vspace*{1pt} let $(x_{\tau},p_{\tau})$ always denote
that process [i.e., what we previously denoted $\hat{Y}_{\tau
}=Y_{\tau}(\hat{A})$]. Moreover, if $\bar{p}$ is the smooth
solution, we see that $p_{\tau}=\bar{p}_{t-\tau}(x_\tau)$. One
consequence of this is that we can generally restrict our attention to
the $x_{\tau}$ process. In particular, if we wish to couple two copies
of the successfully controlled process (so that they meet as quickly as
possible), it is enough to couple the $x_{\tau}$ marginals, since if
the processes meet on the manifold, then they also meet on the fiber.
In this sense, what we are doing is equivalent to just considering
Brownian motion on the underlying time-varying manifold, and so we see
again that running a Brownian motion along the solution flow (and
employing the stochastic techniques that apply in that situation) is
subsumed by the more general construction of the stochastic target problem.

A significant part of our results on the long-time convergence of the
normalized Ricci flow is based on coupling two copies of the marginal
process on $M$, which we denote by $x_{\tau}$ and $y_{\tau}$. Recall
that $x_{\tau}$ will be time-changed Brownian motion on $(M,h)$, with
the time change given by integrating $a=2e^{-2\bar{p}}$ along the
paths, and analogously for $y_{\tau}$, where we let $b$ denote the
instantaneous time-dilation (this is one significant advantage to
working relative to this fixed metric). Note that we have incorporated
the $\sqrt{2}$ normalization factor into the time-change, so that we
really do have Brownian motion with respect to $h$ as the underlying
object. This makes the stochastic analysis look a bit more standard.

We wish to implement the mirror coupling for $x_{\tau}$ and $y_{\tau
}$, where the mirror map is with respect to the fixed $h$ metric.
Viewed in this way, this is a fairly straightforward variant of the
mirror coupling for two Brownian motions on a smooth (nonvarying)
Riemannian manifold. We simply generalize to allow our processes to be
Brownian motions up to a random but smooth (in terms of the particle's
position in space--time) time-change. References for the standard
(nontime changed) construction are \cite{Elton} and \cite
{CranstonJFA}, and we proceed by modifying this as necessary and by not
belaboring the aspects which carry over without modification.

Note that, since we are working only in the cases of nonpositive Euler
characteristic, $a$ (and thus also $b$) is bounded above and below by
positive constants (depending only on the initial metric) for all time,
by the results of the previous section.

First, let $C_M$ be the subset of $M\times M$ consisting of points
$(x,y)$ such that $y\in\Cut(x)$ [which is equivalent to $x\in\Cut
(y)$], and let $D_M$ be the diagonal subset of $M\times M$. Then let
$E_M$ be $M\times M$ minus $C_M$ and $D_M$. Note that the distance
function $\dist(x,y)$ is smooth on $E_M$, and that the direction of
the (unique) minimal geodesic from $x$ to $y$ is smooth on $E_M$. Let
$(x,y)\in E_M$; then the mirror map is the isometry from $T_xM$ to
$T_yM$ given by reflection along the minimal geodesic connecting $x$
and $y$. We see that the mirror map is smooth (on $E_M$, which is where
it is defined). As a result, there is no problem in running the mirror
coupling as long as the joint process\vspace*{1pt} is in $E_M$. That is, for
one-dimensional independent Brownian motions $W^1_{\tau}$ and
$W^2_{\tau}$, consider the system of SDEs
\begin{eqnarray*}
dx_{\tau} &=& a_{\tau} \Biggl[ \Biggl[\sum
_{i=1}^2 \mathfrak{e}_{i}(x_{\tau})
\circ dW_{\tau}^i \Biggr] \Biggr],
\\
dy_{\tau} &=& b_{\tau} \Biggl[ \Biggl[\sum
_{i=1}^2 \Psi_{\tau} \bigl[\mathfrak
{e}_{i}(y_{\tau}) \bigr] \circ dW_{\tau}^i
\Biggr] \Biggr],
\end{eqnarray*}
where $\Psi_{\tau}=\Psi(x_{\tau},y_{\tau}) =m_{x_{\tau},y_{\tau
}}\mathfrak{e}(x_{\tau})\mathfrak{e}(y_{\tau})^{-1}$ with $m_{x,y}$
being the mirror map, namely parallel transport followed by reflection
with respect to the perpendicular to the geodesic from $x$ to $y$. Then
the coefficients are smooth in both space and time, so the system
admits a unique strong solution, up until the first time the process
leaves $E_M$.

The point of the coupling is to get the particles to meet, so we turn
our attention to this issue next. First, note that the marginals
$x_{\tau}$ and $y_{\tau}$ are time-changed Brownian motions as
desired, so we are coupling the right processes. The natural object of
study is the distance between the particles, with respect to the fixed
metric $h$. We denote this distance by $\rho_{\tau}$. It is a
(continuous, nonnegative) semi-martingale, so we derive the SDE that
it satisfies by It\^o's formula. This is the standard computation with
the factors of $a$ and $b$ included, so we will be brief. For more on
this, see \cite{Elton}, Section~6.5.

The martingale part is easily seen to be $(a+b) \,d\hat{W}_{\tau}$
for some Brownian motion~$\hat{W}_{\tau}$, whether we are in the $r=
0$ or $r= -1$ case. (In what follows, we use $\hat{W}_{\tau}$ to
denote some Brownian motion, which may change from appearance to
appearance, in order to more conveniently describe the SDE satisfied by
a given process.) As for the drift, the only contribution comes from
the second derivative of the distance with respect to the diffusions
perpendicular to the geodesic from $x$ to $y$, which is computed in
terms of the index of the appropriate Jacobi field along the geodesic
from $x$ to $y$. We now summarize the computation.

Let $\gamma$ be the unique minimal geodesic from $x$ to $y$
(parametrized by arc length), and let $E$ be a unit vector field along
$\gamma$, perpendicular to $\gamma$ (this determines $v$ uniquely up
to sign, and either of choice of sign is fine). Then we want the Jacobi
field $w(s)E(\gamma(s))$ where $w\dvtx[0,\rho]\rightarrow\bR$ satisfies
\[
\ddot{w}+rw=0, \qquad w(0)=a, \qquad w(\rho)=b.
\]
When $r\equiv0$, the solution space to this differential equation is
spanned by $1$ and $s$. Taking the boundary conditions into account, we
see that the solution is
\[
w(s) = a + \frac{b-a}{\rho} s.
\]
Similarly, when $r\equiv-1$, the solution space is spanned by $\cosh
s$ and $\sinh s$, and the boundary conditions give
\[
w(s) = \frac{a\sinh(\rho-s) + b\sinh s}{\sinh\rho}.
\]

The index of each of these Jacobi fields is given by
\[
\int_{\gamma} \bigl(\dot{w}^2 -rw^2
\bigr) \,ds = w(\rho)\dot{w}(\rho) -w(0)\dot{w}(0),
\]
where the right-hand side is obtained from the left via integration by
parts and the differential equation satisfied by $w$. Thus, for
$r\equiv0$, the index is
\[
b \biggl(\frac{b-a}{\rho} \biggr)- a \biggl(\frac{b-a}{\rho} \biggr)=
\frac{(a-b)^2}{\rho},
\]
and for $r\equiv-1$, the index is
\begin{eqnarray*}
&& b \biggl[a\sinh\rho+ (b-a\cosh\rho)\frac{\cosh\rho
}{\sinh\rho} \biggr]-a \biggl[
\frac{b-a\cosh\rho}{\sinh\rho
} \biggr]
\\
&&\qquad = \bigl(a^2+b^2 \bigr)\coth
\rho-2ab\frac{1}{\sinh\rho}
\\
&&\qquad = (a-b )^2 \coth\rho+2ab\tanh\frac{\rho}{2}.
\end{eqnarray*}

Putting this together, we see that
\[
d\rho_{\tau} = \cases{\displaystyle (a+b) \,d\hat{W}_{\tau} + \frac{1}{2}
\biggl[\frac{ (a-b
)^2}{\rho} \biggr] \,d\tau, &\quad for $r= 0$,
\vspace*{3pt}\cr
\displaystyle (a+b) \,d
\hat{W}_{\tau} + \frac{1}{2} \biggl[ (a-b )^2 \coth
\rho+2ab\tanh\frac{\rho}{2} \biggr] \,d\tau, &\quad for $r= -1$.}
\]
As mentioned, this holds until the first exit time from $E_M$.
Following the reasoning in \cite{Elton}, Section~6.6, one can show
that $\hat{W}_{\tau}=-\sum_{i=1}^{2}\langle\mathfrak
{e}_{i}(x_{\tau}),\dot{\gamma}_{\tau}(0) \rangle \,dW^{i}_{\tau}$
where $\gamma_{\tau}$ is the minimal geodesic joining $x_{\tau}$ and
$y_{\tau}$ starting at $x_{\tau}$ and running at unit speed.

When the particles meet, we have achieved our goal, and we can either
stop the process, or allow it to continue to run as $x_{\tau}=y_{\tau
}$. Either way, there is no problem caused by the process hitting the
diagonal. On the other hand, we do need to find a way to continue the
process past the first hitting time of the cut locus. Showing that this
is possible constitutes the content of the remaining of this section.

%
\begin{teo}\label{thmc}
Let $M=(M,h)$ be a compact surface of constant curvature $0$ or $-1$,
and let $a=a(x,\tau)$ and $b=b(y,\tau)$ be as above. Then there
exists a process $(x_{\tau},y_{\tau})$ on $M\times M$, started from
any $(x_0,y_0)\notin D_M$ and run until the first time of hitting
$D_M$, such that:
\begin{longlist}[(1)]
\item[(1)] The marginals $x_{\tau}$ and $y_{\tau}$ are time-changed
Brownian motions, with times changes given by $a$ and $b$, respectively.
\item[(2)] The distance (relative to $h$) between $x_{\tau}$ and $y_{\tau
}$, denoted $\rho_{\tau}$, satisfies the SDE
%
\begin{equation}
\label{ec1} d\rho_{\tau} = \cases{ \displaystyle (a+b) \,d\hat{W}_{\tau} +
\frac{1}{2} \biggl[\frac{ (a-b
)^2}{\rho} \biggr] \,d\tau-L_{\tau},
\cr
\qquad\mbox{for $r\equiv0$,}
\vspace*{3pt}\cr
\displaystyle (a+b) \,d\hat{W}_{\tau} + \frac{1}{2} \biggl[ (a-b
)^2 \coth\rho+2ab\tanh\frac{\rho}{2} \biggr] \,d
\tau-L_{\tau},
\cr
\qquad\mbox{for $r\equiv-1$,}}
\end{equation}
where $L_{\tau}$ is a nondecreasing process which increases only when
$(x_{\tau},y_{\tau})\in C_M$ [and the set of $\tau$ for which
$(x_{\tau},y_{\tau})\in C_M$ has measure zero almost surely].
\end{longlist}
\end{teo}

\begin{pf}
As mentioned, the only issue is extending the construction mentioned
above past the first hitting time of $C_M$. As usual, we proceed by
approximation.

Choose small, positive $\delta$. Until $y_{\tau}$ is within distance
$\delta$ of $\Cut(x_{\tau})$, we run the mirror coupling as above.
When $y_{\tau}$ hits distance $\delta$ from $\Cut(x_{\tau})$, at
time $\tau_1$, we start to run $x_{\tau}$ and $y_{\tau}$ as
independent (time-changed) Brownian motions. This continues until
$y_{\tau}$ is distance $2\delta$ from $\Cut(x_{\tau})$, at time
$\tau_2$, when we again run them under the mirror coupling. We
continue this procedure, so that we have a joint process $(x^{\delta
}_{\tau},y^{\delta}_{\tau})$ which evolves under the mirror coupling
on intervals of time $[\tau^{\delta}_{2n},\tau^{\delta}_{2n+1})$
and as independent processes on intervals of time $[\tau^{\delta
}_{2n-1},\tau^{\delta}_{2n})$, for nonnegative integers $n$, where
the $\tau_{m}$ are the alternating hitting times of the $\delta$ and
$2\delta$ level sets of the distance from $y_{\tau}$ to $\Cut
(x_{\tau})$. [This is less symmetric than switching when the joint
process is distance $\delta$ or $2\delta$ from $C_M$, in the product
metric on $M\times M$, but it is more convenient to compute with and
works in essentially the same way. In particular, the condition $\dist
(y_{\tau},\Cut(x_{\tau}))<\delta$ determines an open neighborhood
of $C_M$ in $M\times M$, and these neighborhoods converge to $C_M$ as
$\delta\rightarrow0$.]

It is clear that $x^{\delta}_{\tau}$ and $y^{\delta}_{\tau}$ are
time-changed Brownian motions as desired, and that the $\rho^{\delta
}_{\tau}$ satisfies the desired SDE when $(x^{\delta}_{\tau
},y^{\delta}_{\tau})$ is distance more than $2\delta$ from~$C_M$. It
is also clear that when $x^{\delta}_{\tau}$ and $y^{\delta}_{\tau}$
are being run independently, $\rho^{\delta}$ satisfies an SDE of the form
\[
d\rho^{\delta}_{\tau} = u \,d\hat{W}_{\tau} + v \,d\tau-\hat
{L}_{\tau},
\]
where $u$ and $v$ are bounded (with bound depending only on $M$ and the
bounds on $a$ and $b$) and $\hat{L}_{\tau}$ is a nondecreasing
process which increases only when $(x^{\delta}_{\tau},y^{\delta
}_{\tau})\in C_M$ (again, see the references mentioned above).

Suppose we show that, for any $t>0$ and any $\varepsilon>0$, the
expected amount of time on the interval $[0,t]$ that $y^{\delta}_{\tau
}$ spends within distance $\varepsilon$ of $\Cut(x^{\delta}_{\tau
})$ goes to zero with $\varepsilon$ at a rate independent of $\delta
$. Then the amount of time on $[0,t]$ that $y^{\delta}_{\tau}$ spends
within distance $2\delta$ of $\Cut(x^{\delta}_{\tau})$ goes to zero
with $\delta$ (just let $\varepsilon=2\delta$), and thus the amount
of time the particles spend being run independently goes to zero almost
surely as $\delta\searrow0$. (The point is that the total amount of
time spent in the union of all intervals of the form $[\tau^{\delta
}_{2n-1},\tau^{\delta}_{2n})\cap[0,t]$ goes to zero uniformly, even
though the number of such intervals that are nonempty might increase
without bound as $\delta$ goes to zero.) So letting $\delta$ go to
zero, we know there is at least one subsequence along which the process
$(x^{\delta}_{\tau},y^{\delta}_{\tau})$ converges to a limiting
process $(x_{\tau},y_{\tau})$ (by compactness). That this limiting
process satisfies the first property in the theorem is immediate, since
$x^{\delta}_{\tau}$ and $y^{\delta}_{\tau}$ do for all $\delta>0$.
For the second property, note that the contributions from the $u\,
d\tilde{W}_{\tau}$ term and the $v \,d\tau$ term go to zero by the
boundedness of $u$ and $v$ and the fact that the expected length of
time over which these terms are integrated goes to zero. It follows
that the martingale part and the ``regular'' part of the drift come
entirely from the SDE for $\rho$ induced by the (mirror) coupling, and
that the time spent at $C_M$ [equivalently, the time spent with
$y_{\tau}\in\Cut(x_{\tau})$] has measure zero. Finally, the $\hat
{L}_{\tau}$ contribution converges to a term $L_{\tau}$ as indicated.

Thus, to complete the proof, we need only show that the expected amount
of time on the interval $[0,t]$ that $y^{\delta}_{\tau}$ spends
within distance $\varepsilon$ of $\Cut(x^{\delta}_{\tau})$ goes to
zero with $\varepsilon$ at a rate independent of $\delta$. Here, we
will take advantage of the specific geometry with which we are dealing
much more so than in the general approximation procedure just
described. Because the argument is somewhat lengthy, we divide it into
four steps. Moreover, at the end of the first step, we highlight as a
``key fact'' the most important aspect of the geometry for our purposes.\vspace*{6pt}

\textit{Step} 1. Here, we describe the structure of the cut locus, which
is also summarized in Figure~\ref{Fig1} below.

%
\begin{figure}

\includegraphics{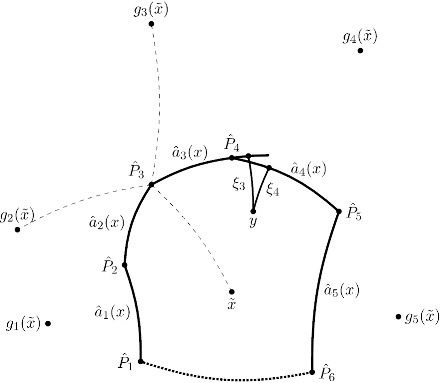}

\caption{The fundamental domain for a negatively curved surface as a
subset of the universal cover, which is the hyperbolic plane.
The pre-image of $x$ on the universal cover is $\{\tilde{x},
g_{1}(\tilde{x}),g_{2}(\tilde{x}),\dots\}$, with $g_{1},g_{2},\dots$
in the Deck group. The\vspace*{1pt} arc $\hat{a}_{i}(x)$ is equidistant from
$\tilde{x}$ and $g_{i}(\tilde{x})$. For instance, the vertex
$\hat{P}_{3}$ is equidistant to $\tilde{x}$, $g_{2}(\tilde{x})$ and
$g_{3}(\tilde{x})$. Here, $\xi_{i}$ is the distance
from $y$ to $a_{i}$, placed by the geodesic that realizes this
distance. Notice that the closest point to $y$ on
$a_3(x)$ falls outside the arc $\hat{a}_{3}(x)$, which is one of the
reasons for introducing a tubular neighborhood around each arc later in
the proof.}\label{Fig1}
\end{figure}

In particular, note that, because we deal with surfaces of nonpositive
curvature, there are no conjugate geodesics, and a point $z$ is in
$\Cut(x)$ exactly when there is more than one minimal geodesic from
$x$ to $z$.
In this case, there are necessarily only finitely many such geodesics,
and the exponential map at $x$ is a local diffeomorphism near (the
tangent vector corresponding to) each of these geodesics.

More concretely, if we let $\tilde{M}$ be the universal cover of $M$
(with the metric induced by $M$) and we let $\tilde{x}$ denote a
distinguished lift of $x$ to $\tilde{M}$, then all other lifts of $x$
can be written as $g(\tilde{x})$ for $g\in G$, the group of Deck
transformations. For a complete treatment of the Deck transformation in
a more general framework, see \cite{Hatcher}. Then one can construct
an open fundamental polygon $P$ (also called a Voronoi region or
Dirichlet region) around $\tilde{x}$ by taking all points of $\tilde
{M}$ that are closer to $\tilde{x}$ than to any other lift of $x$.
Note that $P$ is convex. The boundary of this fundamental polygon
$\partial P$ has each side given by (a portion of) the curve of points
equidistant from $\tilde{x}$ and $g(\tilde{x})$ for some $g$.
Moreover, let $q\dvtx\tilde{M}\rightarrow M$ be the covering map (and
local isometry) given by quotienting by the action of $G$. Then if $z$
is a point on a side (but not a corner) of $\partial P$, $\gamma_0$ is
the minimal geodesic from $\tilde{x}$ to $z$, and $\gamma_1$ is the
minimal geodesic from (the appropriate) $g(\tilde{x})$ to $z$, we see
that $q(z)\in\Cut(x)$ and that $q(\gamma_0)$ and $q(\gamma_1)$ are
the two minimal geodesics from $x$ to $q(z)$ (in $M$). Each corner of
$\partial P$ corresponds to a point $z$ where there are at least two
(but only finitely many) lifts of $x$, say $g_1(\tilde{x}),\ldots,
g_k(\tilde{x})$, such that $z$ is equidistant from $\tilde{x}$ and
each of these other lifts, with corresponding minimal geodesics $\gamma
_1,\ldots,\gamma_k$, and we obtain the minimal geodesics from $x$ to
$q(z)\in\Cut(x)$ as $q(\gamma_0),q(\gamma_1),\ldots,q(\gamma_k)$.
More globally, $q(P)=M\setminus\Cut(x)$ and $q(\partial P)= \Cut(x)$.

The purpose of the above is that it gives us a way to understand how
$\Cut(x_{\tau})$ evolves as $x_{\tau}$ evolves. Indeed, in our
situation, it would be possible to give a fairly\vspace*{1pt} precise description,
since we deal with surfaces of constant curvature. If $r=0$, $\tilde
{M}$ is $\bR^2$ with the Euclidean metric, and the group of Deck
transformations consists of translations by a lattice, if $M$ is
orientable, and thus a torus, or is generated by such translations plus
a reflection, if $M$ is nonorientable, and hence a Klein bottle.
Similarly, if $r=-1$, $\tilde{M}$ is the hyperbolic space $\mathbb
{H}^2$, and the group of Deck transformations consists of a Fuchsian
group, if $M$ is orientable, or is generated by such a group plus a
reflection, if $M$ is nonorientable, and these can be realized fairly
concretely using the upper half-space model of the hyperbolic plane.
Nonetheless, such an argument by cases is tedious and provides more
than we need here. Instead, we give a more general argument.

We choose some $\varepsilon_0>0$, with the intent of studying the
distance to the cut locus in, roughly, an $\varepsilon_0$-neighborhood
of the cut locus, and we will assume $\varepsilon_0$ is small enough
to satisfy various conditions as we go. Recall that the fundamental
polygon $P(x)$ (where we now allow the possibility of making the
dependence on the point $x$ from above explicit) has a finite number of
smooth sides (which we think of as closed segments by including the
corners), which means that $\Cut(x)$ is given by the union of a finite
number of smooth (closed) arcs, which vary smoothly with $x$ [this
smooth dependence follows from the fact that $\tilde{x}$ and all of
the $g(\tilde{x})$ in $\tilde{M}$ vary smoothly with $x$, and thus so
do the curves of points equidistant between them]; denote these arcs by
$\hat{a}_1(x),\ldots,\hat{a}_K(x)$, for some positive integer $K$ as
are shown in Figure~\ref{Fig1} below. (For clarity in the figures, we
label the vertices of $P$ by $\hat{P}_i$, with $\hat{a}_i$ being the
side between $\hat{P}_i$ and $\hat{P}_{i+1}$, with indices understood
modulo $K$.)

Further, we let $a_i(x)$ be a (closed) arc which smoothly extends $\hat
{a}_i(x)$ some small amount (independent of $x$) in each direction. We
can accomplish this by extending each side of the fundamental polygon a
small amount past the two adjacent corners; indeed, in the constant
curvature case, the $\hat{a}_i(x)$ are geodesics segments, and thus we
can extend them to slightly longer geodesic segments. (This is the
general case when $a_i$ is a segment with two endpoints. It is possible
for $a_i$ to be a closed geodesic loop, in which case $a_i$ is just
$\hat{a}_i$.) Next, consider a tubular (open) neighborhood around
$a_i(x)$ consisting of all points that lie on a geodesic perpendicular
to the interior of $a_i(x)$ at distance less than $\varepsilon_0$, and
denote this neighborhood by $Q_i(x)$. Note that $Q_i(x)$ also varies
smoothly with $x$. We now assume that $\varepsilon_0$ is small enough
so that there is always only one such minimal geodesic from $a_i(x)$ to
any point in $Q_i(x)$. Because $Q_i(x)$ varies smoothly, $M$ is
compact, and there are only finitely many sets $Q_i(x)$, it is indeed
possible to choose such $\varepsilon_0>0$ for all $x\in M$ and $i\in\{
1,2,\ldots, K\}$; see Figure~\ref{Fig2}. (Again for clarity in the
figures, we let $P_i$ and $P_{i+1}$ be the endpoints of the extended
arc $a_i$.)

Now let $\xi_i(y)$ be the distance of $y$ from $a_i(x)$. Of course
$\xi_i=\xi_i(y)$ also depends on $x$, through its dependence on
$a_i(x)$. We see that $\xi_i$ is Lipschitz on all of \mbox{$M\times M$} (in
fact, $\xi_i$ is locally given by the minimum or maximum of a finite
number of smooth functions), smooth in both $x$ and $y$ on
$Q_i(x)\setminus a_i(x)$, and convex at $a_i(x)$ [indeed, the signed
distance is smooth in a neighborhood of any point in the interior of
$a_i(x)$, and $\xi_i$ is just the absolute value of this signed
distance]. We also let $S_i(x; \varepsilon)$ be the (closed) set
consisting of all points that lie on a geodesic perpendicular to $\hat
{a}_i(x)$ at a distance no more than $\varepsilon$, for any
$0<\varepsilon<\varepsilon_0/2$. See picture Figure~\ref{Fig2}
below for an illustration of the relevant elements.

%
\begin{figure}

\includegraphics{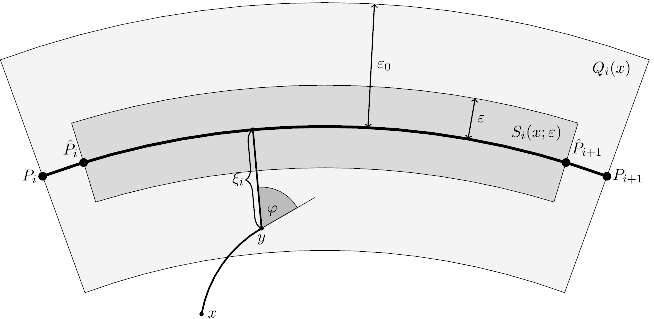}

\caption{The arc $\hat{a}_{i}$ from Figure~\protect\ref{Fig1} is
the arc $\hat{P}_{i}\hat{P}_{i+1}$ and the extension $a_{i}$
described above is given by $P_{i}P_{i+1}$. The light gray area is
$Q_{i}(x)$ while the darker gray area is $S_{i}(x;\varepsilon)$. In
this picture, $\phi$ is the angle between the minimal geodesic joining
$x$ to $y$ and the minimal geodesic from $y$ to $a_i$. We will
frequently think of this tubular neighborhood as lifted to the
universal cover.}
\label{Fig2}
\end{figure}

Consider a point $y$ such that $ \dist(y,\Cut(x)) <\varepsilon_0/2$.
If $y$ is not in $\Cut(x)$, then the closest point (or points) to $y$
in $\Cut(x)$ is in the interior of an $a_i(x)$. This follows from the
fact that the fundamental polygon $P(x)$ is convex, and thus the
closest boundary point to any interior point is in the interior of an
edge (i.e., the closest point is not a corner). It follows that, for
any $0<\varepsilon<\varepsilon_0/2$ and $x\in M$,
\[
\bigl\{ y\dvtx \dist\bigl(y,\Cut(x)\bigr) \leq\varepsilon\bigr\}
\subset\bigcup
_{i=1}^{K} S_i(x; \varepsilon).
\]
So, in order to control the expected amount of time on the interval
$[0,t]$ that $y^{\delta}_{\tau}$ spends within distance $\varepsilon
$ of $\Cut(x^{\delta}_{\tau})$, it is enough to control the expected
amount of time on the interval $[0,t]$ that $y^{\delta}_{\tau}$
spends in $S_i(x^{\delta}_{\tau}; \varepsilon)$, for each $i$.

Before we move on to the next step, we make an important point, which
will be in fact the backbone of the argument, and comes from the fact
that on nonpositively curved manifolds there are no conjugate points.
Let $\phi$ be the angle between the minimal geodesic joining $x$ to
$y$ and the minimal geodesic from $y$ to $a_i$, as shown in
Figures~\ref{Fig2} and \ref{Fig3}. Then we claim that $\llvert
\phi\rrvert$ is
bounded away from $\pi/2$ on $Q_i$, with the bound depending only on
$M$ and $\varepsilon_0$. To see this, first note that, any geodesic
from $x$ to $a_i$ cannot be tangent to $a_i$. Indeed, this is so
because in the Euclidean plane and the hyperbolic half space, the
curves which are equidistant to two points are geodesics and on any
manifold a geodesic curve is uniquely defined by a point and the
tangent at the point. Thus, the if the arcs from $x$ to $z$ would be
tangent, this would mean that $x$ is on the arc $a_i$ which is a
contradiction. It follows thus that any geodesic from $x$ to a point on
$a_i$ intersects $a_i$ transversally. Thus, if $z$ is a point on any of
the curves $a_i$, the angle between the geodesic arcs $xz$ and $a_i$ is
always positive, and it varies continuously as $z$ moves along $a_i$.
In addition, since the fundamental polygon $P_i$ changes continuously
with $x$ (and $M$ is a compact manifold) we see that there is a value
$\omega>0$ which depends on the manifold $M$ and the length of the
extended arcs $a_i$, such that for any $z$ on any of the $a_i$ arcs,
the angle between the geodesic arc $xz$ and $a_i$ belongs to $[\omega
,\pi/2]$.

%
\begin{figure}

\includegraphics{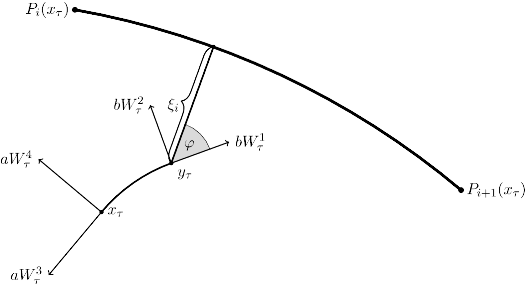}

%
%
%
%
%
%
%
%
%
%
%
\caption{This is the picture of the arc $a_i(x_{\tau})$ together with
the distances from $y_\tau$ to $a_i(x_{\tau})$ and to~$x_\tau$. When
$x_{\tau}$ and $y_{\tau}$ move independently, the 2-dimensional
driving Brownian motions $(W^{3},W^{4})$ and $(W^{1},W^{2})$ are
independent while for the mirror coupling case they are the same (i.e.,
$W^{1}=W^{3}$ and $W^{2}=W^{4}$).}
\label{Fig3}
\end{figure}

Next, suppose the point $y$ approaches a point $z\in a_i$ smoothly,
from $Q_i\setminus a_i$ (visually, we think of letting $\xi_i$ go to
zero in Figures~\ref{Fig2} or \ref{Fig3}). Then the limit of $\llvert
\phi
\rrvert$ is the angle between $xz$ and the (``outward pointing'') normal
vector to $a_i$ at $z$ (this follows from writing everything up to
first order at $z$), which in turn is $\pi/2$ minus the angle between
$xz$ and $a_i$. Thus, the limit of $\llvert\phi\rrvert$ as
$y$ approaches $z$ is
bounded from above by $\pi/2-\omega$. (As this argument makes clear,
this is just a simple consequence of the transversality of geodesics
that do not coincide.) Again by continuity and compactness, this
implies that there is some neighborhood of $a_i$ where the absolute
value of $\phi$ is bounded from above by some constant less than~$\pi/2$.

Thus, if we take $\varepsilon_0$ small enough and the point $y$ moves
in any of the sets $Q_i$, the absolute value of the angle $\phi$
introduced above is bounded away from $\pi/2$ with the bound only
depending on $M$, the length of the extended arcs $a_i$, and
$\varepsilon_0$. Equivalently, $\cos\phi$ is bounded from below by a
positive constant under the same conditions. Because this is one of the
key geometric facts underlying our argument, we highlight it separately here.

\begin{fact*}\label{r1} There is a constant $\phi_{0}<\pi/2$ such
that for small enough $\varepsilon_{0}$, and any $y\in Q_{i}(x)$,
%
\begin{equation}
\label{ecphi} \llvert\phi\rrvert\le\phi_{0} <\pi/2.
\end{equation}
\end{fact*}

From now on, we assume that $\varepsilon_0$ is small enough so that
this holds.

(We note that if we consider a high-dimensional compact manifold of
nonpositive sectional curvature, the analogous fact holds relative to
the hypersurface components of the cut locus. For this and related
reasons, the present argument extends naturally to higher dimensions.
However, if we allow positive curvature, the structure of the cut locus
can change significantly, and new ideas would be required to extend
this method of proving the existence of the mirror coupling.)\vspace*{6pt}

\textit{Step} 2. Here, we study the evolution of $\xi_i$ under the
process by controlling the SDE it satisfies, both when the particles
are running independently and when they are running under the mirror
coupling. We also (and much more briefly) derive an SDE which governs
how quickly $y^{\delta}_{\tau}$ can move from the complement of
$Q_i(x^{\delta}_{\tau})$ to $S_i(x^{\delta}_{\tau}; \varepsilon_0/2)$.

Let $\xi_{i,\tau}$ be, as usual, the process $\xi_i(y^{\delta
}_{\tau})=\dist(y^{\delta}_{\tau},a_i(x^{\delta}_{\tau
}) )$. From the convexity properties of $\xi_i$ and the It\^
o--Tanaka formula, we see that $\xi_{i,\tau}$ is a semi-martingale.
Next, suppose that $y^{\delta}_{\tau}\in Q_i(x^{\delta}_{\tau})$.
There are two cases to consider, the one when the particles are running
independently, and the one when they are running under the mirror
coupling. Since we will be assuming either one or the other of these
cases in what follows, we will drop the superscript $\delta$'s in the
notation, making it less cumbersome.

We\vspace*{1pt} begin with some observations that apply in either case. Referring to
Figure~\ref{Fig3}, we run the processes $x_{\tau}$ and $y_{\tau}$
as indicated, without yet assuming that $(W^{1},W^{2})$ and
$(W^{3},W^{4})$ are either independent or identical. Then, since the
distance function $\xi_{i}$ is smooth away from 0 and convex at $0$,
we can use It\^o--Tanaka formula to get that
%
\begin{equation}
\label{ecxi}\qquad d\xi_{i,\tau}=-b\cos\phi \,dW_{\tau}^{1}-b
\sin\phi \,dW_{\tau
}^{2}+r_{i} \,dW_{\tau}^{3}+s_{i}
\,dW_{\tau}^{4}+v_{i} \,d\tau+dL_{i},
\end{equation}
where $\llvert r_{i}\rrvert, \llvert s_{i}\rrvert,\llvert
v_{i}\rrvert$ are bounded by some constants
depending only on $M$, $\varepsilon_{0}$ and the bounds on $a$ and
$b$, and where $L_{i}$ is a nondecreasing process increasing only when
$\xi_{i}$ is $0$. Notice the minus sign in the first term on the
left-hand side above equation, which is due to the fact that $\xi_{i}$
decreases as $y$ approaches $a_{i}$ because the gradient of $\xi_i$
points opposite of the minimal geodesic from $y_{\tau}$ to $a_i$.
Notice also that the first two terms of the martingale part are
obtained by fixing the point $x$ and taking the derivative with respect
to $y$, while the last two terms of the martingale part come from
fixing $y$ and taking the derivative with respect to $x$ (in this case
the arc $a_{i}$ changes with $x$).

The martingale part of equation (\ref{ecxi}) is controlled by its
quadratic variation process; equivalently, the martingale part is a
time-changed (one-dimensional) Brownian motion, and thus controlled by
the time change. Clearly, the precise behavior of the quadratic
variation is different in our two cases (the independent case and the
mirror-coupled case). But in either case, our goal now is to show that
the martingale part is of the form $u_i \,d\hat{W}_{\tau}$ where
$\hat{W}_{\tau}$ is a Brownian motion and $u_i$ a process such that
$0<\alpha\le u_i \le\beta$ with $\alpha$ and $\beta$ two constants
independent of $\varepsilon$ and $\delta$. The purpose is that, if
this is true, standard methods of stochastic analysis will allow is to
estimate the amount of time that $y_{\tau}$ spends near $a_i(x_{\tau
})$, which is our overall task. In particular, the upper bound by some
$\beta$ already follows from equation (\ref{ecxi}), in both cases.
Thus, the real work is in obtaining the lower bound, and for this we
treat the two cases separately.

If the particles $x_{\tau}$ and $y_{\tau}$ evolve independently, then
$W^{1}$, $W^{2}$, $W^{3}$, and $W^{4}$ in equation (\ref{ecxi}) are
independent, and thus the martingale part can be written as $\sqrt
{b^{2}\cos^{2}\phi+b^{2}\sin^{2}\phi+r^{2}_{\tau}+s^{2}_{\tau}}
\,d\hat{W}_{\tau}=\sqrt{b^{2}+r^{2}_{\tau}+s^{2}_{\tau}} \,d\hat
{W}_{\tau}$. Because $b$ is bounded from below by a positive constant,
this proves that, in the case the particles run independently, for
$y_{\tau}\in Q_i(x_{\tau})$, $\xi_{i,\tau}$ satisfies the SDE
\[
d\xi_{i,\tau} = u_{I,i} \,d\hat{W}_{\tau} +
v_{I,i} \,d\tau+dL_{I,i},
\]
where $u_{I,i}$ and $\llvert v_{I,i}\rrvert$ are bounded and
$u_{I,i}$ is bounded
from below by a positive constant, with all of these bounds depending
only on $M$, $\varepsilon_0$, and the bounds on $a$ and $b$, and where
$L_{I,i}$ is a nondecreasing process that increases only when $\xi
_{i,\tau} =0$. (Here, the subscript $I$ is meant to denote that these
are the coefficients for the SDE induced by running the particles
independently.)

Now we wish to perform a similar analysis when the particles are being~run under the mirror coupling. The issue now is that, in this case, the
Brownian motions in equation (\ref{ecxi}) are correlated by
$W_{3}=W^{1}$ and $W_{4}=W_{2}$.  Therefo\-re, the martingale part is of
the form
$\sqrt{(-b\cos\phi+r_{\tau})^{2}+(-b\sin^{2}\phi+s_{\tau
})^{2}} \,d\hat{W}_{\tau}$. To show that the coefficient is bounded
from below by a positive constant, it is enough to show that at least
one of the squares under the square root stays bounded from below.
Recall now that $\llvert\phi\rrvert$ is bounded away from
$\pi/2$, which we noted
as our ``key fact'' earlier and, therefore, $b\cos\phi$ is bounded
away from $0$. Our strategy in what follows is to show that the term
$r_{\tau}$ does not spoil this property (i.e., we want to make sure
that the contribution to the quadratic variation coming from moving
$y_{\tau}$ by $dW^1_{\tau}$ is not cancelled by the movement of
$x_{\tau}$ by $dW^3_{\tau}=dW^1_{\tau}$). More precisely, we are
going to show that $r_{\tau}$ is actually negative and this proves
that $-b\cos\phi+r_{\tau}$ stays away from 0, which is enough to get
the desired conclusion.

In this case, we write the evolution for $\xi_{i,\tau}$ in the form
%
\begin{equation}
\label{EqnLocalXiCoupled} d\xi_{i,\tau} = u_{C,i} \,d\hat{W}_{\tau} +
v_{C,i} \,d\tau+dL_{C,i},
\end{equation}
where $u_{C,i}=\sqrt{(-b\cos\phi+r_{\tau})^{2}+(-b\sin^{2}\phi
+s_{\tau})^{2}}$. (Here, the subscript $C$ denotes that the
coefficients for the SDE are induced by running the particles mirror
coupled.) Also, recall that the particles never run under the mirror
coupling when $y_{\tau}$ hits $\Cut(x_{\tau})$ (for any $\delta$).
Thus, when considering the present case, we have that the geodesic
between $x_{\tau}$ and $y_{\tau}$ along which we perform the mirror
coupling evolves continuously. That is, essentially, Figure~\ref
{Fig3} evolves continuously, and in particular, the vectors along
which the diffusions $W^{1}$, $W^{2}$, $W^{3}$ and $W^{4}$ occur and
the angle $\phi$ evolves continuously.


Since the martingale part of $d\xi_{i,\tau}$ depends only on the
first-order structure at a point, we see that we can consider the
contribution of $x_{\tau}$ with $y$ fixed and the contribution of
$y_{\tau}$ with $x$ fixed separately (the ``complete'' martingale part
is just given by the sum of these two contributions). We have already
seen that when $x$ is fixed, the arc $a_i$ is also fixed, and the
contribution coming from the evolution of $y_{\tau}$ is $-b\cos\phi
\,dW^1_{\tau} -b\sin\phi \,dW^2_{\tau} $.

%
\begin{figure}[t]

\includegraphics{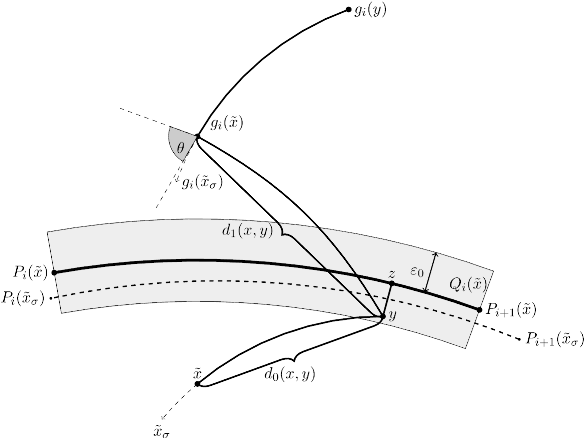}

\caption{Here, $\tilde{x}$ moves away from $y$ at unit speed along
the dashed line, motion which we parametrize as $\tilde{x}_{\sigma}$.
The distance $d_{1}(\tilde{x},y)$ is obtained as the distance between
$g_{i}(\tilde{x})$ and $y$ (on the universal cover), and as $\tilde
{x}$ moves, $g_{i}(\tilde{x})$ moves along the dashed line as
$g_{i}(\tilde{x}_{\sigma})$. The derivative of $d_{1}(x,y)$ is thus
given by $\cos(\theta)$, where $\theta$ is the angle between the
arcs $g_{i}(\tilde{x})g_{i}(y)$ and $g_{i}(\tilde{x})y$. Notice that
for small enough $\varepsilon_{0}$ and $y\in Q_{i}(\tilde{x})$ the
angle $\theta$ is positive (and we assume that $\varepsilon_{0}$
satisfies this condition). Indeed, $g_{i}(y)$ cannot be on the geodesic
$g_{i}(\tilde{x})y$ because $y$ and $g_{i}(y)$ must be some distance
apart (given by the shortest noncontractible loop on the manifold M).
The dashed line $P_{i}(\tilde{x}_{\sigma})P_{i+1}(\tilde{x}_{\sigma
})$ is the arc $a_{i}(\tilde{x}_{\sigma})$, corresponding to moving
$\tilde{x}$ to $\tilde{x}_{\sigma}$.}
\label{Fig4}
\end{figure}

The other contribution to the martingale part of $d\xi_{i,\tau}$
comes from letting $x_{\tau}$ evolve while keeping $y$ fixed [namely
the $r_{i}$ and $s_i$ terms in equation (\ref{ecxi})]. To provide a
good picture for what follows, we put all the relevant elements in
Figure~\ref{Fig4} below. The point is that when $x_{\tau}$ moves,
$a_i(x_{\tau})$ moves as well. In order to make the exposition
clearer, we will assume for the moment that $\xi_{i,\tau}\neq0$ [and
thus $\dist(y,a_i(x_{\tau}))>0$], so that $\xi_i$ is smooth in a
neighborhood of the present point. We now use $x$ to denote the
starting point of $x_{\tau}$, before we let it move to first order
(stochastically). Thus, the closest point to $y$ on $a_i(x)$, which we
denote $z$, and which we now also fix, is in the interior of $a_i$, by
the definition of $Q_i$. Let $d_0(\cdot,\cdot)$ denote the distance
between points in a neighborhood of $x$ to points in a neighborhood of
$y$ along geodesics which are close (in the exponential map) to the
minimal geodesics from $x$ to $\hat{a}_i(x)$ that lie on the same side
of $a_i(x)$ as $y$. Let $d_1(\cdot,\cdot)$ denote the similar
distance from points in a neighborhood of $x$ to a points in a
neighborhood of $y$ along minimal geodesics from $x$ to $\hat{a}_i(x)$
that lie on the opposite side of $a_i(x)$ as $y$. In other words, if we
think about the universal cover, $d_0$ corresponds to minimal geodesics
(in $\tilde{M}$) starting from a points in a neighborhood of $\tilde
{x}$, and $d_1$ corresponds to minimal geodesics starting from points
in a neighborhood of $g(\tilde{x})$, where $g$ is such that $g(\tilde
{x})$ is the point ``on the other side'' of the lift of $a_i(x)$. Both
$d_0$ and $d_1$ are smooth in both arguments. We can assume $z$ is in
both neighborhoods of $y$, so that $d_0(x,z)=d_1(x,z)$, and moreover,
$a_i(x)$ is given by the equation $d_0(x,\cdot)=d_1(x,\cdot)$ near
$z$. Also $d_0(x,y)<d_1(x,y)$, and it is this inequality which shows
``which side'' of $a_i(x)$ $y$ is on.

Now suppose $x_{\sigma}$ moves away (it moves away because of the
mirror coupling) from $y$ along the minimal geodesic connecting them,
at unit speed. (We imagine $x_{\sigma}$ moves smoothly in order to
estimate the relevant gradients, and then we use It\^o's rule to
determine the stochastic analogue.) Referring to Figure~\ref{Fig4}
and (\ref{ecxi}), our next goal is to show that $r=\frac{\partial
}{\partial\sigma}\xi_i(x_\sigma,y)\mid_{\sigma=0}\le0$. On one hand,
we have $\frac{d}{d\sigma} \,d_0(x_{\sigma},y)=1$.

Now we will invoke a similar argument to the one involved in
establishing the key fact. Namely, for a point $z$ on $a_i$, the angle
between the arc from $z$ to $g_i(\tilde{x})$ and the arc from $g_i(z)$
to $g_i(\tilde{x})$ is not zero (this is the angle $\theta$ in
Figure~\ref{Fig4}, when $y$ allowed to go to $z$). To see this, we
argue otherwise. If the angle were to be 0, then since both arcs are
geodesic, they would overlap (said differently, one arc would be a
sub-arc of the other). Further, since $d(\tilde{x},z)=d(g_i(\tilde
{x}),g_i(z))$ we would obtain that $g_i(z)=z$ (i.e., the arcs would
be identical), which is impossible since $g_i$ is an element of the
group of Deck transformations other than the identity, and thus $g_i$
does not fix any point of the universal cover (see \cite{Hatcher},
page~70). (Note that we do not rule out the possibility that
$\theta=\pi$, which can happen, but causes no trouble for the present
proof.) Since $\theta$ is not zero in the limit as $y$ approaches $z$,
the same continuity and compactness arguments as before show that there
is some neighborhood of $a_i$ on which $\theta$ is bounded below by a
positive constant.

In particular, according to the discussion above and referring to
Figure~\ref{Fig4}, for $\varepsilon_0$ small enough, the angle
$\theta$ is not $0$, thus $\frac{d}{d\sigma} \,d_1(x_{\sigma},y) =
\cos\theta\le1-\lambda$, for some small, positive $\lambda$.
Hence, $\frac{d}{d\sigma} (d_1-d_0)(x_{\sigma},y)<-\lambda<0$.
Because $d_1-d_0$ is smooth, if $y$ is close enough to $z$, we must
have that $\frac{d}{d\sigma} (d_1-d_0)(x_{\sigma},z)<0$. Further, by
compactness and continuity, we can make $\varepsilon_{0}$ small enough
so that this holds whenever $y\in Q_i\setminus a_i$. The point is that
as $x_{\sigma}$ moves away from $y$ in this way, $d_1(x_{\sigma},z)$
immediately becomes smaller than $d_0(x_{\sigma},z)$, putting $z$ on
the ``opposite side'' of $a_i(x_{\sigma})$ from $y$. Since
$a_i(x_{\sigma})$ moves smoothly, this means that it immediately
intersects the minimal geodesic from $y$ to $z$ between $y$ and $z$, or
in other words, that $\xi_i(x_{\sigma},y)$ decreases to first order,
and thus $r$ is negative. In fact, an even softer argument gives that
the distance between $y$ and $a_{i}(x_{\sigma})$ is smaller than $\xi
_{i}$, as is obvious from Figure~\ref{Fig4}, which implies $r\le0$.
As pointed out earlier, this is enough to conclude that $u_{C,i}$ in
equation (\ref{EqnLocalXiCoupled}) is bounded from below by $b\cos
\phi$, and thus is bounded from below by a positive constant depending
only on $M$, $\varepsilon_0$, and the bounds on $a$ and $b$.

%
\begin{figure}[b]

\includegraphics{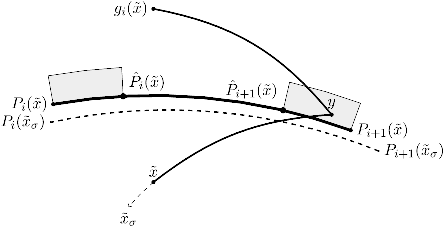}

\caption{The case when particle $y$ moves in the region $ \{
\tilde{\xi}_i<0 \}$ [corresponding to $y$ being on the opposite side
of $a_i(x)$ from $x$], which is denoted by the gray boxes. Here, we see
$a_{i}(\tilde{x}_{\sigma})$ moving away from $y$, but this still
corresponds to $\tilde{\xi}_i$ decreasing, because of our choice of sign.}
\label{Fig5}
\end{figure}

%
\begin{figure}[b]

\includegraphics{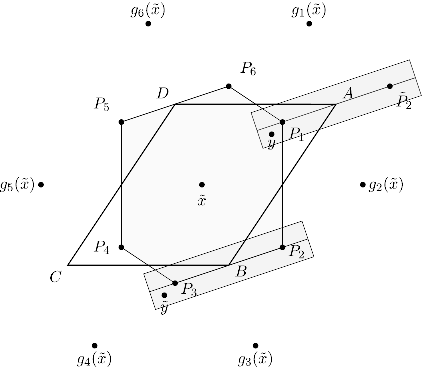}

\caption{This is the universal cover of a flat torus with the group of
Deck transformation generated by two translations (one by the vector
$DA$ and the other one by the vector $AB$). The torus is obtained by
gluing the edges of the parallelogram $ABCD$, however the fundamental
polygon is $P_1P_2P_3P_4P_5P_6$. One of the arcs we compute the
distance to is the arc $P_2P_3$ and its corresponding extension. The
point is that because of the identification of the points down on the
surface $M$, the two rectangles around $P_2P_3$ and $P_1\tilde{P}_2$
are the same. Thus, even though it seems that the point $y$ is far away
from the arc $P_2P_3$, in fact it is not due to this identification
with the point $\tilde{y}$. The point $\tilde{y}$ viewed from the
point of view of $\tilde{x}$ is ``on the other side'' of the arc
$P_2P_3$. This explains why the case in Figure~\protect\ref{Fig5}
has to be considered.}
\label{Fig6}
\end{figure}

To extend this to the case when we allow $\xi_{i,\tau} =0$, which
means $y_{\tau}\in a_i(x_{\tau})\setminus\hat{a}_i(x_{\tau})$
(because the process never runs under the mirror coupling on the cut
locus itself), let $\tilde{\xi}_i$ be the signed distance from
$y_{\tau}$ to $a_i$ in some neighborhood of $y_{\tau}$, as shown in
Figure~\ref{Fig5} [here we take our sign so that $\tilde{\xi
}_i<0$ when $y$ is on the opposite side of $a_i(x)$ from $x$]. The
reader may ask why do we have to consider this case at all. The answer
is provided in the caption of Figure~\ref{Fig6} below and it comes
from the fact that essentially the picture on the universal cover does
not reflect exactly what happens in the projection. To resume, then
$\tilde{\xi}_i$ is smooth on this neighborhood and, as noted above,
the minimal geodesic from $x_{\tau}$ to $y_{\tau}$ is evolving
continuously (so there is no problem with the definition of the mirror
coupling). Then the above arguments apply to $\tilde{\xi}_i$ as well,
by continuity. To be more precise, we can think about the analogue of
equation (\ref{ecxi}) for the signed distance~$\tilde{\xi}$. On
the region where $\tilde{\xi}_i\geq0$, the same arguments as above
apply (since here $\tilde{\xi}_i=\xi_i$, and now we can include
points with $\tilde{\xi}_i=0$ because they are now smooth points). On
the region where $\tilde{\xi}_i< 0$, the gradient of $\tilde{\xi
}_i$ is minus the gradient of $\xi_i$. Thus, the first term in
equation (\ref{ecxi}) is still $-b\cos\phi \,dW^1_{\tau}$. Now
the arc $a_{i}(\tilde{x}_{\sigma})$ is moving away from $y$, however,
because the gradient has the opposite sign, the above reasoning again
shows that $r\leq0$. This is illustrated in Figure~\ref{Fig5}, which
should make the underlying geometry clear. Thus, in taking the
quadratic variation, we still have that $(-b\cos\phi+r_{\tau})^{2}$
is bounded from below by a positive constant, which is what we wanted.
Because $\xi_i=\llvert\tilde{\xi}_i\rrvert$, we use the
It\^o--Tanaka
formula to see that, for $y_{\tau}\in Q_i(x_{\tau})$, we have
%
\begin{equation}
\label{EqnXiUnderInd} d\xi_{\tau} = u_{C} \,d\hat{W}_{\tau} +
v_{C} \,d\tau+dL_C,
\end{equation}
where $u_{C}$ and $\llvert v_{C}\rrvert$ are bounded and
$u_{C}$ is bounded from
below by a positive constant, with all of these bounds depending only
on $M$, $\varepsilon_0$, and the bounds on $a$ and $b$, and where
$L_C$ is a nondecreasing process that increases only when $\xi
_{i,\tau} =0$ (assuming the process is being run under the mirror
coupling, of course).

Now we see that the SDE satisfied by $\xi_{i,\tau}$ switches between
these two possibilities, running under independence or running under
the mirror coupling, at the stopping times $\tau^{\delta}_i$. In particular,
\begin{eqnarray*}
d\xi_{i,\tau} &=& u_i \,dW_{\tau} + v_i d
\tau+dL_i \qquad\mbox{for }y^{\delta}_{\tau}\in
Q_i\bigl(x^{\delta}_{\tau}\bigr),
\end{eqnarray*}
where
\[
u =\cases{ u_{I,i}, &\quad for $\tau\in\bigl[
\tau^{\delta}_{2n-1},\tau^{\delta
}_{2n}\bigr)$,
\vspace*{3pt}\cr
u_{C,i}, &\quad for $\tau\in\bigl[\tau^{\delta}_{2n},
\tau^{\delta}_{2n+1}\bigr)$}
\]
and
\[
v=\cases{
v_{I,i}, &\quad for $\tau\in\bigl[\tau^{\delta}_{2n-1},
\tau^{\delta
}_{2n}\bigr)$,
\vspace*{3pt}\cr
v_{C,i}, &\quad for $\tau\in\bigl[
\tau^{\delta}_{2n},\tau^{\delta}_{2n+1}\bigr)$,}
\]
and where $L$ is a nondecreasing process that increases only when $\xi
_{i,\tau}=0$.
The previously discussed bounds on $u_{I,i}$, $u_{C,i}$, $v_{I,i}$ and
$v_{C,i}$ imply that there exist positive constants $\alpha$, $\beta$
and $\gamma$, depending only on $M$, $\varepsilon_0$, and the bounds
on $a$ and~$b$, such that $\alpha\leq u_i \leq\beta$ and $\llvert
v_i\rrvert\leq
\gamma$, for any $\delta>0$ and any $i\in\{1,\ldots,K\}$. (I.e.,
these bounds hold for both $u_{I,i}$ and $u_{C,i}$ and both $v_{I,i}$
and $v_{C,i}$, and thus they hold for $u_i$ and $v_i$ regardless of
whether the process is being run under independence or under the mirror
coupling, and thus they hold independent of $\delta$. Also, because
there are only finitely many $i$, these bounds can be made independent
of $i$.)

Our final task, in this step, is to introduce a semi-martingale that
will allow us to control how the joint process transitions from having
$y_{\tau}\in S_i(x_{\tau};\varepsilon_0/2)$ to having $y_{\tau
}\notin Q_i(x_{\tau})$. Indeed, this control is the other reason for
introducing the neighborhood $Q_i(x)$. Note that $S_i(x;\varepsilon
_0/2)$ and the complement of $Q_i(x)$ are a positive distance apart,
for any $x$, so we can take $\eta_i$ to be a smooth function taking
values in $[0,1]$, such that $\eta_i$ is identically equal to 0 on
$S_i(x;\varepsilon_0/2)$ and identically equal to 1 on the complement
of $Q_i(x)$. Further, we can let $\eta_i$ vary smoothly in $x$. As
usual, we let $\eta_{i,\tau}$ be the semi-martingale arising from
composing $\eta_i$ with the process $(x_{\tau},y_{\tau})$, where the
particles can be running independently or under the mirror coupling
(and thus switching at the $\tau^{\delta}_{n}$ for any $\delta>0$).
Then, by smoothness and compactness, we see that $\eta_{i,\tau}$
satisfies the SDE
\[
d\eta_{i,\tau} = U_i \,dW_{\tau} + V_i\, d
\tau,
\]
everywhere on $M\times M$, where $U_i$ and $\llvert V_i\rrvert
$ are bounded, with
bounds depending only on $M$, $\varepsilon_0$, and the bounds on $a$
and $b$. More precisely, there are positive constants $\tilde{\beta}$
and $\tilde{\gamma}$, depending only on $M$, $\varepsilon_0$, and
the bounds on $a$ and $b$, such that $0\leq U_i \leq\tilde{\beta}$
and $\llvert V_i\rrvert\leq\tilde{\gamma}$, for any
$\delta>0$ and any $i\in\{
1,\ldots,K\}$. (Because we are dealing only with coarse bounds, it
seems unnecessary to consider the cases of independence and mirror
coupling separately, as we did for $\xi_i$.)\vspace*{6pt}

\textit{Step} 3. Here, we give the basic estimate on the amount of time
spent near each piece of the cut locus; that is, the amount of time
$y^{\delta}_{\tau}$ spends in $S_i(x^{\delta}_{\tau}; \varepsilon
)$. The argument is essentially an exercise in stochastic calculus,
which uses only the bounds on the SDEs satisfied by $\xi_{i,\tau}$
and $\eta_{i,\tau}$ that we just derived.

For $0<\varepsilon<\varepsilon_0/2$, consider the function
\[
f(x)=\cases{ x^2, &\quad for $0\leq x\leq\varepsilon$,
\vspace*{3pt}\cr
2
\varepsilon x-\varepsilon^2, &\quad for $x>\varepsilon$.}
\]
Then $f$ is $C^1$ with $\llvert f^{\prime}(x)\rrvert\leq
2\varepsilon$, and
$f^{\prime\prime}$ exists in the weak sense.

For now, we fix some $i$, and just write $\xi$ for $\xi_i$, $Q$ for
$Q_i$, etc.

We first suppose that $y_{0}\in Q(x_{0})$. Then the It\^o--Tanaka
formula shows that, at least until the first time $y_{\tau}$ exits
$Q(x_{\tau})$, $f(\xi_{\tau})$ satisfies the SDE
%
\begin{equation}
\label{ec131} \qquad df(\xi_{\tau}) = f'(\xi_{\tau})
u_\tau \,dW_{\tau} + f'(\xi_{\tau
})v_{\tau}
\,d\tau+ u_{\tau}^2 \mathbf{1}_{(-\varepsilon,\varepsilon)}(
\xi_{\tau}) \,d\tau+f'(\xi_{\tau})\,dL_{\tau}.
\end{equation}
Notice here that $\int_{0}^{\tau}f'(x_{u})\,dL_{u}$ is a nondecreasing
process due to the conditions on $L$.

Next, consider the sequence of stopping times-defined inductively as follows:
\[
\sigma_{0}=0\quad\mbox{and}\quad\zeta_{0}=\inf\bigl\{ s
\ge0\dvtx y_{s} \notin Q(x_{s}) \bigr\}
\]
and for $n\ge1$
\[
\sigma_{n}=\inf\bigl\{s\ge\zeta_{n-1}\dvtx y_s\in
S(x_s;\varepsilon_0/2) \bigr\} \quad\mbox{and}\quad
\zeta_{n}=\inf\bigl\{s\ge\sigma_{n}\dvtx y_{s}
\notin Q(x_{s}) \bigr\}.
\]
It is clear now, from the geometry of these sets, that
\[
\int_0^t \mathbf{1}_{S(x_{\tau};\varepsilon)}
(y_{\tau} )\,d\tau\leq\sum_{n\ge0}\int
_{\sigma_{n}\wedge t}^{\zeta
_{n}\wedge t}\mathbf{1}_{(-\varepsilon,\varepsilon)} (\xi
_{\tau} )\,d\tau
\]
and thus
\[
\E\biggl[\int_0^t \mathbf{1}_{S(x_{\tau};\varepsilon)}
(y_{\tau} )\,d\tau\biggr]\leq\sum_{n\ge0}\E
\biggl[\int_{\sigma_{n}\wedge t}^{\zeta_{n}\wedge t}\mathbf
{1}_{(-\varepsilon,\varepsilon)} (
\xi_{\tau} )\,d\tau\biggr].
\]

On each time interval $[\sigma_{n}\wedge t,\zeta_{n}\wedge t]$, we
use (\ref{ec131}) combined with the fact that
\[
\E\biggl[\int_{\sigma_{n}\wedge t}^{\zeta_{n}\wedge t} f'(\xi
_{\tau})u_{\tau}\,dW_{\tau} \biggr]=0
\]
and $\llvert f'(x)\rrvert\le2\varepsilon$ to first justify
that (recall that $\int
f^{\prime} \,dL_{u}$ is nondecreasing)
\begin{eqnarray*}
&& \E\biggl[\int_{\sigma_{n}\wedge t}^{\zeta_{n}\wedge t}u_{\tau
}^{2}
\mathbf{1}_{(-\varepsilon,\varepsilon)} (\xi_{\tau} )\,d\tau\biggr] + \E
\biggl[\int
_{t\wedge\sigma_{n}}^{t\wedge
\zeta_{n}} f^{\prime} (\xi_{\tau}
)v_{\tau} \,d\tau\biggr]
\\
&&\qquad \le\E\bigl[f(\xi_{t\wedge\zeta_{n}})-f(
\xi_{t\wedge\sigma_{n}})\bigr]\le2\varepsilon\varepsilon_0 \Prob(
\sigma_n<t).
\end{eqnarray*}
Complement this with the fact that $v_{\tau}f'(\xi_{\tau})\ge
-2\varepsilon\gamma$ and $u_{\tau}\ge\alpha$ to arrive at
\[
\E\biggl[\int_{\sigma_{n}\wedge t}^{\zeta_{n}\wedge t}\mathbf
{1}_{(-\varepsilon,\varepsilon)} (\xi_{\tau} )\,d\tau\biggr]\le\frac
{2\epsilon}{\alpha^{2}}\bigl(
\varepsilon_0 \Prob(\sigma_n<t)+\gamma E[t\wedge
\zeta_{n}-t\wedge\sigma_{n}]\bigr).
\]
Consequently, since $\sum_{n\ge0}(t\wedge\zeta_{n}-t\wedge\sigma
_{n})\le t$ this results in the main estimate
\[
\E\biggl[\int_0^t \mathbf{1}_{S(x_{\tau};\varepsilon)}
(y_{\tau} )\,d\tau\biggr]\le\frac{2\epsilon}{\alpha
^{2}}\bigl(\varepsilon_0
\E[D_{t}] +\gamma t\bigr),
\]
where $D_{t}$ is the number of ``downcrossings'' of $y_{\tau}$ from
the complement of $Q(x_{\tau})$ to $S(x_{\tau};\varepsilon_0/2)$,
inside the interval $[0,t]$. That is, $D_t$ is supremum of $n$ such
that $\sigma_n \leq t$.

This basic estimate leaves us with the task of getting an upper bound
on the number of downcrossings, as just described. First, note that
$\eta_{\zeta_{n}} = 1$ and $\eta_{\sigma_{n}} = 0$,
assuming these stopping times are less than or equal to $t$. Also, we
have that
\begin{eqnarray*}
\E[\eta_{t\wedge\zeta_n} - \eta_{t\wedge\sigma_n} ]&=& \E\biggl[\int
_{t\wedge\sigma_n}^{t\wedge\zeta_n} U_i \,dW_{u}
\biggr]+ \E\biggl[\int_{t\wedge\sigma_n}^{t\wedge\zeta_n} V_i
\,du \biggr]
\\
&=& \E\biggl[\int_{t\wedge\sigma_n}^{t\wedge\zeta_n} V_i \,du
\biggr]
\\
&\leq&\tilde{\gamma} \E[t\wedge\zeta_n- t\wedge\sigma_n
],
\end{eqnarray*}
where we used the boundedness of $U_i$ to see that the martingale part
is actually integrable. For any $N\geq1$, we have that
\[
\sum_{n=0}^{N} \E[\eta_{t\wedge\zeta_n} -
\eta_{t\wedge
\sigma_n} ] \leq\tilde{\gamma} \sum_{n=0}^{N}
\E[t\wedge\zeta_n - t\wedge\sigma_n ]\leq\tilde{\gamma}
t.
\]
Since we always have $0\leq\eta_{\tau}\leq1$, we let $N\rightarrow
\infty$
to see that
\[
-1+\E[D_{t} ]\leq\sum_{n=0}^{\infty}
\E[\eta_{t\wedge\zeta_n} - \eta_{t\wedge\sigma_n} ] \leq\tilde{\gamma
} t.
\]
This, in turn, implies that
\[
\E\biggl[\int_0^t \mathbf{1}_{S(x_{\tau};\varepsilon)}
(y_{\tau} )\,d\tau\biggr]\le\frac{2\varepsilon}{\alpha^{2}} \bigl
(\varepsilon_0(1+
\tilde{\gamma}t)+\gamma t\bigr).
\]
In the case, we start with $y_{0}\notin Q(x_{0})$, we run the process
until it hits $S(x_{\tau};\varepsilon_0/2)$, and once this happens
use the same argument as above.\vspace*{6pt}

\textit{Step} 4. From here, the proof is easy to complete. We just put
everything together.

For any $0<\varepsilon<\varepsilon_0/2$, the expected amount of time
on the interval $[0,t]$ that $y^{\delta}_{\tau}$ spends within
distance $\varepsilon$ of $\Cut(x^{\delta}_{\tau})$ satisfies
\begin{eqnarray*}
\E\biggl[\int_0^t \mathbf{1}_{\{\dist(\cdot,\Cut(x^{\delta
}_{\tau}))\leq\varepsilon\}}
\bigl(y^{\delta}_{\tau} \bigr)\,d\tau\biggr]&\leq&\sum
_{i=1}^K \E\biggl[\int_0^t
\mathbf{1}_{S_i(x^{\delta}_{\tau
};\varepsilon)} \bigl(y^{\delta}_{\tau} \bigr)\,d\tau
\biggr]
\\
&\leq& K \frac{2\varepsilon}{\alpha^{2}} \bigl(\varepsilon
_0(1+\tilde{\gamma
}t)+\gamma t\bigr) =C\varepsilon,
\end{eqnarray*}
where $C>0$ (defined by the above equality) is a constant depending
only on $t$, $M$, $\varepsilon_0$, and the bounds on $a$ and $b$ (in
particular, $C$ does not depend on $\delta$). As noted just before
step~1, this is exactly the estimate we need to complete the proof.
\end{pf}

\section{Convergence of first order to constant curvature in the case \texorpdfstring{$\chi(M)=0$}{chi(M)=0}}\label{s6}

Now that we have our uniqueness/verification theorem and the general
coupling procedure, we begin exploring some of the consequences. As
usual, for simplicity, we assume that we have a smooth solution
$\bar{p}_t$ for all time $t\ge0$ on the manifold $M$. We take
here a flat
metric $h$, which is possible under the assumption that \mbox{$\chi(M)=0$}.

The main result of this section is the following.

%
\begin{teo}\label{tt1}
For $M$, $h$, and $\bar{p}_{0}$ as above, suppose that we have a
smooth solution $\bar{p}_t$ to equation (\ref{EqnNRFp}) for all
$t\in
[0,\infty)$. Then there exist constants, $c,C>0$ which depend only on
the metrics $g_{0}$ and $h$ such that
%
\begin{equation}
\label{etc} \sup_{x\in M}\bigl\llvert\bar{p}_{t}(x)
\bigr\rrvert\le ce^{-Ct}.
\end{equation}
\end{teo}

\begin{pf}
Fix a time $t>0$, a time $s\in[0,t)$ and a point $x\in M$ so that the
Ricci flow has a solution on $[0,t]$. The first thing to notice is that
$p_{\tau}=\bar{p}_{t-\tau}(x_{\tau})$ is a martingale. Thus, we
have the following stochastic representation:
%
\begin{equation}
\label{EqnPAverage} \bar{p}_t(x)= \E\bigl[\bar{p}_{t-\sigma}(x_{\sigma})
\bigr]
\end{equation}
valid for any stopping time $\sigma$ with $0\le\sigma\le t$. In
particular, setting $\sigma=t$ shows that $\bar{p}_t(x)$ is a weighted
average of the values of $\bar{p}_{0}$. Thus,
%
\begin{equation}
\label{EqnInequalities} \min_M \bar{p}_{0} \leq\min
_M \bar{p}_t \leq\max
_M \bar{p}_t \leq\max
_M \bar{p}_{0}
\end{equation}
for any $t$.
The main idea for getting (\ref{etc}) is to prove that for some $c,C>0$,
%
\begin{equation}
\label{etc2} \osc\bar{p}_{t}\le c e^{-Ct}.
\end{equation}
Indeed, if this is true, then combining this with the fact that the
integral of $e^{2\bar{p}_{t}}$ with respect to the volume induced by
$h$ is $1$, we deduce that there is at least one point $\tilde{x}$ for
which $\bar{p}_{t}(\tilde{x})=0$ and from here it is clear that we
get (\ref{etc}).

We now choose any two starting points $x$ and $y$ for the processes
$x_{\tau}$ and $y_{\tau}$. Over each of these points, there is
exactly one point [$\bar{p}_{t}(x)$ and $\bar{p}_{t}(y)$]
in the fiber
which is in the reachable set $\Gamma_{t}$. We wish to run the
controlled process starting from both $(x,\bar{p}_t(x))$ and
$(y,\bar{p}_t(y))$, and couple them so that they meet as quickly
as possible.
Our reachable sets have the semi-group property, that is, the process
$(x_{\tau},p_{\tau})$ at time $\tau\in[0,t]$ is on $\Gamma_{t-\tau
}$, and since we know that we have a solution until time $t$, we know
that after running the controlled processes for time $\tau\leq t$ they
will be on the solution section corresponding to the Ricci flow at time
$t-\tau$. This means that if the particles couple on $M$, they couple
in the total space as well, that is, $x_{\tau}=y_{\tau}$ implies that
$\bar{p}_{t-\tau}(x_{\tau})=\bar{p}_{t-\tau}(y_{\tau})$
as well.

In light of this, if $\sigma$ is the coupling time of $x_{\sigma}$
and $y_{\sigma}$, the martingale property gives that
%
\begin{eqnarray}
\label{etc3} \bar{p}_{t}(x)-\bar{p}_{t}(y)&=&\E\bigl[
\bar{p}_{t-\sigma\wedge
s}(x_{\sigma\wedge s})\bigr]-\E\bigl[\bar{p}_{t-\sigma\wedge
s}(y_{\sigma
\wedge s})
\bigr]\nonumber
\\
&=& \E\bigl[\bar{p}_{t-\sigma}(x_{\sigma})-\bar{p}_{t-\sigma
}(y_{\sigma}),
\sigma\le s\bigr]
\nonumber\\[-8pt]\\[-8pt]\nonumber
&&{} +\E\bigl[\bar{p}_{t-s}(x_{s})-
\bar{p}_{t-
s}(y_{s}),s<\sigma\bigr]
\\
&=&\E\bigl[\bar{p}_{t-s}(x_{s})-\bar{p}_{t-s}(y_{s}),s<
\sigma\bigr].\nonumber
\end{eqnarray}
The outcome of this is that
%
\begin{equation}
\label{etc4} \osc\bar{p}_{t}\le\Prob(s<\sigma)\osc
\bar{p}_{t-s}.
\end{equation}
What remains to be controlled here is $\Prob(s<\sigma)$. While the
above is true for any coupling of $x_{\tau}$ and $y_{\tau}$, we wish
to use the mirror coupling, as was introduced in the previous section.
The main property of this coupling, for us, is contained in (\ref{ec1}) which gives the equation satisfied by the distance function
$\rho_{\tau}=d(x_{\tau},y_{\tau})$, namely
%
\begin{equation}
\label{etc5} \,d\rho_{\tau} = (a+b) \,d\hat{W}_{\tau} +
\frac{1}{2\rho_{\tau
}}(a-b)^2 \,d\tau-L_{\tau}
\end{equation}
with $a_{\tau}=e^{-\bar{p}_{t-\tau}(x_{\tau})}$, $b_{\tau
}=e^{-\bar{p}_{t-\tau}(y_{\tau})}$ and $\hat{W}$ being a
one-dimensional Brownian motion on the time interval $[0,t]$. Obviously,
the time $\tau$ runs up to $\sigma$ (the hitting time of 0) or $t$,
whichever comes first and the term $L_{\tau}$ is nonnegative. We are
interested in estimating the probability this hitting time $\sigma$
occurs after time $s$. To this end, the first thing which will be used
here is the fact that from (\ref{EqnInequalities}) we know that $a$
and $b$ are all bounded from above as well from below. So we have two
constants $A,B>0$ which are depending only on $\bar{p}_{0}$, or
otherwise the starting metric $g_{0}$, with the property that
%
\begin{equation}
\label{etc6} A \le a,\qquad b\le B.
\end{equation}

To move on, we let
\[
\lambda(u)=\int_{0}^{u}\frac{1}{(a_{v}+b_{v})^{2}}\,dv
\]
be the time-change making the martingale part of $\rho_{\tau}$ from
(\ref{etc5}) into a Brownian motion. Then with the notation $\tilde
{\rho}_{u}=\rho_{\lambda(u)}$,
%
\begin{equation}
\label{etc7} d\tilde{\rho}_{u} = d\tilde{W}_{u} +
\frac{1}{2\tilde{\rho
}_{u}}\frac{(a-b)^2}{(a+b)^2} \,d u -d\tilde{L}_{u},
\end{equation}
where $a$ and $b$ are evaluated at time $\lambda(u)$ and the above
equation is valid for $u\in[0,t\wedge\lambda^{-1}(t))$, where
$\lambda^{-1}(t)$ is the first value of $u$ corresponding to $\lambda
(u)=t$. Obviously, $c u\le\lambda(u)\le Cu$ for some constants
$c,C>0$ and also because of (\ref{etc6}),
\[
\biggl\llvert\frac{a-b}{a+b}\biggr\rrvert\le\frac
{B-A}{B+A}=1-
\epsilon<1.
\]
Ignoring the $L$ term in (\ref{etc7}) and then using standard
comparison for ordinary stochastic differential equations, we learn
that the process $\tilde{\rho}$ is bounded above by a Bessel process
of dimension $\delta<2$ and starting at some value $\tilde{\rho
}_{0}$ bounded by the diameter (with respect to the metric $h$) of the
manifold $M$. Thus, invoking \cite{JY}, equation (15), which gives the
distribution of the hitting time $\tilde{\sigma}$ of $0$ for a Bessel
process of dimension $\delta<2$ starting at $\tilde{\rho}_{0}$, we obtain
\[
\Prob(s<\tilde{\sigma})=\frac{1}{\Gamma(1-\delta/2)}\int_{0}^{\tilde
{\rho}^{2}_{0}/(2s)}y^{-\delta/2}e^{-y}\,dy.
\]
Finally, since $cu\le\lambda(u)\le Cu$ and the diameter of the
manifold $M$ is finite, we arrive at
\[
\Prob(s<\sigma)\le\frac{1}{\Gamma(1-\delta/2)}\int_{0}^{D/s}y^{-\delta
/2}e^{-y}\,dy=:
\Lambda(s),
\]
where $D$ is a constant which depends only on the initial metric
$g_{0}$ and some geometry of the underlying metric $h$ (more precisely
the diameter of $M$ with respect to~$h$). Hence, it turns out that the
function $\Lambda$ is determined by the metrics $h$ and~$g_{0}$.

To summarize, from (\ref{etc4}) and the preceding we now have that
\[
\osc\bar{p}_{t}\le\Lambda(s) \osc\bar{p}_{t-s}.
\]
Using this, it is easy to get (\ref{etc2}) as follows. For $t\in
[0,1]$, we know from (\ref{EqnInequalities}), that $\osc\bar{p}_{t}\le
\osc\bar{p}_{0}$. Now for each $t\in[n,n+1]$, $n\ge1$,
using repeatedly the above inequality, we arrive at
\[
\osc\bar{p}_{t}\le\Lambda(1)^{n}\osc\bar{p}_{t-n}
\le\Lambda(1)^{t}\osc\bar{p}_{0}/\Lambda(1)
\]
which is exactly the exponential decay of (\ref{etc2}) since
$0<\Lambda(1)<1$.
\end{pf}

%
\begin{remark}
It is interesting to point out that we can prove the same exponential
decay as in Theorem~\ref{tt1} for the case of $\chi(M)<0$ using the
coupling argument. This decay is, however, already taken care of by the
a priori estimates of Corollary~\ref{c22}. Nonetheless, this coupling
argument is the one we will employ for the gradient estimates in the
following section.
\end{remark}

\section{Estimates on the gradient decay of the normalized Ricci flow in the case \texorpdfstring{$\chi(M)\le0$}{$chi(M)<=0$}}\label{s7}

We continue under the same assumptions that $M$ is a compact surface
with reference metric $h$ of constant curvature $0$ or $-$1 (so $M$ has
nonpositive Euler characteristic by the Gauss--Bonnet theorem) and
$g_0$ is a smooth initial metric in the same conformal class and with
the same area as $h$, so that the normalized Ricci flow has a smooth
solution for all time which is given by $\bar{p}_t$. Now, $\bar{p}_t$
converges in the $C^0$-norm exponentially fast to 0 as shown in
Corollary~\ref{c22} for the case $\chi(M)<0$ and Theorem~\ref{tt1}
for the case \mbox{$\chi(M)=0$}. So we have that for some constants~\mbox{$c,C>0$,}
%
\begin{equation}
\sup_{x\in M}\bigl\llvert\bar{p}_{t}(x)\bigr\rrvert
\le ce^{-Ct}.
\end{equation}

Let
\[
G(t)=\sup_{x\in M}\bigl\llvert\nabla\bar{p}_{t}(x)
\bigr\rrvert.
\]
The idea is to start with
\[
\bigl\langle\nabla\bar{p}_{t}(x),\xi\bigr\rangle=\lim
_{h\to0}\frac{\bar{p}_{t}(\gamma_{h}(x))-\bar{p}_{t}(x)}{h},
\]
where $\xi$ is a unit vector in the tangent space at $x$ and $\gamma
_{t}(x)$ is any curve started at $x$ with initial speed $\xi$. Then we
use the coupling to estimate $\bar{p}_{t}(x)-\bar{p}_{t}(y)$ for $x$
and $y$ close to one another. Due to the nonlinearity of the flow, the
estimates coming from the above will still contain the gradient bounds,
but in the end, letting $x$ and $y$ come close to one another leads to
a functional inequality on $G(t)$, from which we are able to derive the
desired estimate.

%
\begin{teo}\label{t7}
If $\chi(M)\le0$ then $G(t)$ goes to 0 exponentially fast. As a
consequence, $\bar{p}_t$ converges to 0 exponentially fast in $C^1$.
\end{teo}

\begin{pf}
Pick two sufficiently close points $x,y\in M$ and some $t>0$, and let
$\rho_{\OldSigma}=d(x_{\OldSigma},y_{\OldSigma})$ for $0\le
\OldSigma\le t$ be the distance (measured with respect to the time
independent metric $h$) between the processes $x_{\OldSigma}$ and
$y_{\OldSigma}$ started at $x$ and $y$, respectively. We are going to
use mirror coupling for the processes $x_{\cdot}$ and $y_{\cdot}$.
Recall that the coupling equations satisfied by $(x_{\OldSigma
},p_{\OldSigma})$ and $(y_{\OldSigma},q_{\OldSigma})$ are given by
%
\begin{eqnarray}
\label{ge6} \,dx_{\OldSigma} &=& e^{-p_{\OldSigma}} \Biggl[ \Biggl[\sum
_{i=1}^2 \mathfrak{e}_{i}(x_{\OldSigma})
\sqrt{2} \circ dW_{\OldSigma}^i \Biggr] \Biggr],\nonumber
\\
dy_{\OldSigma} &=& e^{-q_{\OldSigma}} \Biggl[ \Biggl[\sum
_{i=1}^2 \mathfrak{e}_{i}(y_{\OldSigma})
\sqrt{2} \circ d\tilde{W}_{\OldSigma
}^i \Biggr] \Biggr],
\nonumber\\[-8pt]\\[-8pt]\nonumber
dp_{\OldSigma} &=& e^{-p_{\OldSigma}} \Biggl[\sum_{i=1}^2
a_i \sqrt{2} \,dW_{\OldSigma}^i \Biggr] + r
\bigl(e^{-2p_{\OldSigma}}-1\bigr) \,d\OldSigma,
\\
dq_{\OldSigma} &=& e^{-q_{\OldSigma}} \Biggl[\sum_{i=1}^2
a'_i \sqrt{2} \,d\tilde{W}_{\OldSigma}^i
\Biggr] + r\bigl(e^{-2q_{\OldSigma}}-1\bigr) \,d\OldSigma,\nonumber
\end{eqnarray}
where $r=0$ or $-1$ and $\tilde{W}$ is the Brownian motion given by
the mirror coupling.

We consider $\OldTau$, the coupling time of $x_{\cdot}$ and $y_{\cdot
}$. From the fact that $p_{\OldSigma}+r\int_{0}^{\OldSigma
}(1-e^{-2p_{u}})\,du$ is a martingale and $p_{\OldSigma}=\bar
{p}_{t-\OldSigma}(x_{\OldSigma})$, we write
%
\begin{equation}
\label{ge1} \bar{p}_{t}(x)-\bar{p}_{t}(y)=
\E[p_{t\wedge\OldSigma}-q_{t\wedge
\OldSigma}]-r\E\biggl[\int_{0}^{t\wedge\OldSigma
}
\bigl(e^{-2p_u}-e^{-2q_{u}}\bigr)\,du \biggr]
\end{equation}
for any stopping time $\OldSigma$. The useful estimates we are
interested in are estimates from above of $\bar{p}_{t}(x)-\bar
{p}_{t}(y)$, and this is good if we assume that $\bar{p}_{t}(x)-\bar
{p}_{t}(y)>0$. This is always possible unless $\bar{p}_{t}$ is
constant in which case the gradient is $0$, so there is nothing to
prove then. Thus, assume that $\bar{p}_{t}(x)-\bar{p}_{t}(y)>0$ for
some points $x$ and $y$ (which is the same as $p_{0}>q_{0}$) and take
$\alpha$ to be the first time $u$ for which $p_{u}=q_{u}$. With this
choice of the stopping time, for any $u\in[0,\alpha]$ we know that
$p_{u}\ge q_{u}$, which thus means $e^{-2p_{u}}-e^{-2q_{u}}\le0$. This
combined with the fact that $r\le0$ and the exponential decay of $\bar
{p}_{t}$, implies that for any $s\in[0,t\wedge1]$,
\[
\bar{p}_{t}(x)-\bar{p}_{t}(y)\le\E[p_{ \alpha}-q_{\alpha},
\alpha\le s]+\E[p_{s}-q_{s},s<\alpha]\le
ce^{-Ct}\Prob(s<\alpha).
\]

The point is that if $\OldTau$ is the first coupling time, of the
processes $x_{\cdot}$ and $y_{\cdot}$, it is obvious that $\alpha\le
\OldTau$, and thus
\[
\Prob(s<\alpha)\le\Prob(s<\OldTau)\qquad\mbox{for any }s\in[0,t\wedge1],
\]
which in turn yields
%
\begin{equation}
\label{ge9} \bar{p}_{t}(x)-\bar{p}_{t}(y)\le
ce^{-Ct}\Prob(s<\OldTau)\qquad\mbox{for any }s\in[0,t\wedge1].
\end{equation}
With this equation our next task becomes the estimate of $\Prob
(s<\OldTau)$.

From Theorem~\ref{thmc}, we learn that the distance process $\rho
_{\OldSigma}$ satisfies
%
\begin{equation}
\label{ge7} \,d\rho_{\OldSigma}\le\bigl(e^{-p_{\OldSigma
}}+e^{-q_{\OldSigma
}}
\bigr)\,dB_{t}+\frac{(e^{-p_{\OldSigma}}-e^{-q_{\OldSigma}})^{2}}{2\rho
_{\OldSigma}}\,d\OldSigma
\end{equation}
in the case $r=0$ and
%
\begin{eqnarray}
\label{ge8}
d\rho_{\OldSigma}&\le&\bigl(e^{-p_{\OldSigma
}}+e^{-q_{\OldSigma
}}
\bigr)\,dB_{\OldSigma}
\nonumber\\[-8pt]\\[-8pt]\nonumber
&&{} + \frac{1}{2} \biggl[ \bigl(e^{-p_{\OldSigma
}}-e^{-q_{\OldSigma}}
\bigr)^2 \coth\rho_{\OldSigma} +2e^{-p_{\OldSigma}-q_{\OldSigma}}\tanh
\frac{\rho_{\OldSigma}}{2} \biggr]\,d\OldSigma
\end{eqnarray}
in the case $r=-1$. Here, $B_{t}$ is a one-dimensional Brownian motion
run in the time interval $[0,t]$.

So far, we have used this strategy of coupling in the proof of
Theorem~\ref{tt1}, in which, due to the singularity in the drift of
the equations (\ref{ge7}) and (\ref{ge8}), we compared the distance
function $\rho_{\tau}$ with a Bessel process. For the gradient
estimates, we are going to remove the singularity based on the
observation that
\[
p_{\OldSigma}=\bar{p}_{t-\OldSigma}(x_{\OldSigma})\quad\mbox{and
similarly}\quad q_{\OldSigma}=\bar{p}_{t-\OldSigma}(y_{\OldSigma}).
\]
The upshot of this is that the term $e^{-p_{\tau}}-e^{-q_{\tau}}$ is
in fact of order $\rho_{\tau}$. More precisely, due to the
boundedness of $\bar{p}$,
\begin{eqnarray*}
\bigl\llvert e^{-p_{\OldSigma}}-e^{-q_{\OldSigma}}\bigr\rrvert&=&\bigl
\llvert
e^{-\bar{p}_{t-\OldSigma
}(x_{\OldSigma})}-e^{-\bar{p}_{t-\OldSigma}(y_{\OldSigma})}\bigr\rrvert
\le C\,d(x_{\OldSigma},y_{\OldSigma})
\sup_{x\in M}\bigl\llvert\nabla\bar{p}_{t-\OldSigma}(x)\bigr
\rrvert
\\
&=&C G(t-\OldSigma)\rho_{\OldSigma}.
\end{eqnarray*}

Since $\rho_{\OldSigma}\le D$, where $D$ is the diameter of $M$, it
is straightforward to show that either (\ref{ge7}) or (\ref{ge8}) implies
\[
d\rho_{\OldSigma}\le\bigl(e^{-p_{\OldSigma}}+e^{-q_{\OldSigma
}}
\bigr)\,dB_{\OldSigma}+C\bigl(1+G^{2}(t-\OldSigma)\bigr)
\rho_{\OldSigma}\,d\OldSigma.
\]
To go further from here, consider $\tilde{\rho}_{\OldSigma}$ the
solution to
\[
d\tilde{\rho}_{\OldSigma}= \bigl(e^{-p_{\OldSigma}}+e^{-q_{\OldSigma
}}
\bigr)\,dB_{\OldSigma}+C\bigl(1+G^{2}(t-\OldSigma)\bigr)\tilde{
\rho}_{\OldSigma
}\,d\OldSigma,
\]
with the same initial condition $\rho_{0}=d(x,y)$ as $\rho_{\OldSigma
}$. Standard arguments (in fact a simple application of Gronwall's
lemma) give that
\[
\rho_{\OldSigma}\le\tilde{\rho}_{\OldSigma}
\]
which results in the fact that the first hitting time of $0$ for $\rho
$ is less then or equal to the first hitting time of $0$ for $\tilde
{\rho}$. Now if $\tilde{\OldTau}$ denotes the hitting time of $0$
for the process $\tilde{\rho}_{t}$
%
\begin{equation}
\label{ge10} \Prob(s<\OldTau)\le\Prob(s<\tilde{\OldTau})\qquad\mbox{for
all }s
\in[0,t\wedge1].
\end{equation}

Therefore, the task now is to estimate the latter, and to do this we
solve for $\tilde{\rho}$ as
\[
\tilde{\rho}_{\OldSigma}= \biggl(\rho_{0}+\int
_{0}^{\OldSigma} \bigl(e^{-p_{v}}+e^{-q_{v}}
\bigr)e^{-\int_{0}^{v}f(z)\,dz}\,dB_{v} \biggr)e^{\int
_{0}^{\OldSigma}f(z)\,dz}
\]
with the notation $f(\OldSigma)=C(1+G^{2}(t-\OldSigma))$, for $0\le
\OldSigma\le t$. Consequently, the first hitting time of $0$ for
$\tilde{\rho}$ is the first hitting time of $-\rho_{0}$ for the
time-changed Brownian motion $\int_{0}^{\OldSigma}
(e^{-p_{v}}+e^{-q_{v}})e^{-\int_{0}^{v}f(z)\,dz}\,dB_{v}$. In law, this is
the same as the first hitting time of $-\rho_{0}$ of $B_{c(\OldSigma
)}$, with the time change
\[
c(\OldSigma)=\int_{0}^{\OldSigma
}\bigl(e^{-p_{v}}+e^{-q_{v}}
\bigr)^{2}e^{-2\int_{0}^{v}f(z)\,dz}\,dv.
\]
Once again using the boundedness of $\bar{p}$, we can find a constant
$C>0$ such that
\[
c(\OldSigma)\ge\tilde{c}(\OldSigma):=C\int_{0}^{\OldSigma
}e^{-2\int_{0}^{v}f(z)\,dz}\,dv
\qquad\mbox{for }\tau\in[0,t].
\]

Now, if $\OldTau_{-\rho_{0}}$ is the first hitting time of $-\rho
_{0}$ for the Brownian motion, then the hitting time of $-\rho_{0}$
for $B_{c(\OldSigma)}$ is given by $c^{-1}(\OldTau_{-\rho_{0}}\wedge
c(t))$. This combined with (\ref{ge10}) yields that
%
\begin{eqnarray}\label{ge30}
\Prob(s<t\wedge\tilde{\OldTau}) &=&\Prob\bigl(s< c^{-1}
\bigl(\OldTau_{-\rho
_{0}}\wedge c(t)\bigr)\bigr)\le\Prob\bigl(c(s)\le
\OldTau_{-\rho_{0}}\bigr)
\nonumber\\[-8pt]\\[-8pt]\nonumber
&\le&\Prob\bigl(\tilde{c}(s)<\OldTau_{-\rho_{0}}
\bigr).
\end{eqnarray}

The distribution of $\OldTau_{-\rho_{0}}$ is actually well understood
(see, e.g., the remark after \cite{RevYor}, Proposition 3.7 of
Chapter~II), and its density is given by $\frac{\rho_{0}}{\sqrt
{2\pi x^{3}}}e^{-\rho_{0}^{2}/(2x)}$ on the positive axis, which
results with
\[
\Prob\bigl(\tilde{c}(s)<\OldTau_{-\rho_{0}}\bigr)=\int_{\tilde
{c}(s)}^{\infty}
\frac{\rho_{0}}{\sqrt{2\pi x^{3}}}e^{-\rho
_{0}^{2}/(2x)}\,dx=\frac{2}{\sqrt{2\pi}}\int_{0}^{\sfrac{\rho
_{0}}{\sqrt{\tilde{c}(s)}}}e^{-\OldSigma^{2}/2}\,d
\OldSigma.
\]

Going back to (\ref{ge9}) and using the preceding, we conclude that
for $s\in[0,t]$,
\[
\bar{p}_{t}(x)-\bar{p}_{t}(y)\le ce^{-C(t-s)}\int
_{0}^{\sfrac{\rho
_{0}}{\sqrt{\tilde{c}(s)}}}e^{-\OldSigma^{2}/2}\,d\OldSigma,
\]
from which, using the fact that $d(x,y)=\rho_{0}$ and letting $\rho
_{0}$ go to $0$, we fairly easily deduce that
\[
G(t)\le c\frac{e^{-Ct}}{\sqrt{\tilde{c}(s)}},
\]
which we rearrange as
\begin{eqnarray}
A(t)\int_{0}^{s}e^{-\int_{0}^{\OldSigma}A(t-u)\,du}\,d\OldSigma\le
ce^{-Ct}
\nonumber
\\
\eqntext{\mbox{for all }s\in[0,t\wedge1]\mbox{ with } A(
\OldSigma)=CG^{2}(\OldSigma).}
\end{eqnarray}

From here, the exponential decay of $A(t)$ is taken care of by the
following lemma.
\end{pf}

%
\begin{Lemma}\label{gel1} Suppose $A\dvtx[0,\infty)\to[0,\infty)$ is a
continuous function with the property that for some constants $c,C>0$,
%
\begin{equation}
\label{ge12} A(t)\int_{0}^{s}e^{-\int_{0}^{\OldSigma}A(t-u)\,du}\,d
\OldSigma\le ce^{-Ct}\qquad\mbox{for all }s\in[0,t\wedge1].
\end{equation}
Then there are constants $k,K>0$ such that
\[
A(t)\le Ke^{-kt}\qquad\mbox{for all }t>0.
\]
\end{Lemma}

\begin{pf}
For each $n\ge1$, let
\[
m_{n}=\sup_{t\in[n,n+1]}A(t)\quad\mbox{and}\quad M_{n}=\sup _{t\in[n-1,n+1]}A(t).
\]
Notice that the exponential decay we are
looking for is actually equivalent to $m_{n}\le K e^{-k n}$ for large
enough $n$.

Now, for $t\in[n,n+1]$ and $s\in[0,1]$, we have $t-s\in[n-1,n+1]$
and, therefore, $-A(t-u)\ge-M_{n}$, which combined with (\ref{ge12})
yields, for $t$ near the supremum of $A(t)$ on $[n,n+1]$, and
eventually another constant $c>0$
\[
m_{n}\int_{0}^{1}e^{-\OldSigma M_{n}}\,d
\OldSigma= m_{n}\frac
{1-e^{-M_{n}}}{M_{n}}\le e^{-cn}\qquad\mbox{for
all large }n,
\]
which in turn gives
{\renewcommand{\theequation}{**}
\begin{equation}\label{eq**}
m_{n}\le\frac{M_{n}}{1-e^{-M_{n}}}e^{-cn}.
\end{equation}}\setcounter{equation}{51}%
Now, for each particular $n$, we have one of the following two alternatives:
\begin{longlist}[(2)]
\item[(1)]$M_{n}\le e^{-cn/2}$, in which case it is clear that
{\renewcommand{\theequation}{$\#$}
\begin{equation}\label{eqt}
m_{n}\le e^{-cn/2}.
\end{equation}}\setcounter{equation}{51}%
\item[(2)]$M_{n}>e^{-cn/2}$, and\vspace*{1pt} in this case $1-e^{-M_{n}}>
1-e^{-e^{-cn/2}}>\frac{1}{2}e^{-nc/2}$ for large enough $n$, say $n\ge
n_{0}$. From (\ref{eq**}), it follows that $m_{n}\le2M_{n}e^{-cn/2}\le M_{n}
e^{-c/2} $ for all $n$ large enough, say $n\ge n_{1}$. This inequality
implies that
{\renewcommand{\theequation}{$\#\#$}
\begin{equation}\label{eqtt}
m_{n}\le m_{n-1}e^{-c/2}
\qquad\mbox{for all }n\ge n_{1}.
\end{equation}}\setcounter{equation}{51}%
Indeed if the supremum of $A(t)$ on the interval $[n-1,n+1]$ is the
same as the supremum on $[n,n+1]$, then $M_{n}=m_{n}$ and this in turn
implies $M_{n}=0$, in particular we trivially have (\ref{eqtt}). If the
supremum of $A(t)$ on $[n-1,n+1]$ is the same as the supremum on
$[n-1,n]$, this gives $M_{n}=m_{n-1}$ and then (\ref{eq**}) gives~(\ref{eqtt}).
\end{longlist}

Using these two alternatives we argue as follows. Assume that there is
a large enough $n_{2}$ such that $m_{n_{2}}\le e^{-cn_{2}/2}$. Then an
easy induction using the two alternatives above give that $m_{n}\le
e^{-cn/2}$ for all $n\ge n_{2}$. If there is no such $n_{2}$, this
means that for all $n\ge n_{1}$ we clearly have the second alternative
and in this case $m_{n}\le m_{n_{1}}e^{-(n-n_{1})c/2}$. In both cases,
we obtain the exponential decay we were looking for.

An alternative proof can be given as follows. Take a sufficiently large
constant $K>0$, which will be chosen later. Now we look at
$B(t)=A(t)e^{kt}$. Assume there is a time $t\ge K$ such that $B(t)=\max
_{\tau\in[0,t]}B(\tau)$. We then have $A(\tau)\le A(t)e^{-k(\tau
-t)}$ for $\tau\in[0,t]$ and from (\ref{ge12}) with $s=1$,
\[
A(t)\int_{0}^{1}e^{-\tau A(t)e^{k} \,du}\,d\tau\le
ce^{-C t},
\]
and from this
\[
1-e^{-A(t)e^{k}}\le ce^{k}e^{-Ct},
\]
which gives that
\[
A(t)\le-e^{-k}\log\bigl( 1-ce^{k}e^{-Ct}\bigr).
\]
If we choose the constant $K$ large enough and $k$ small enough, so
that $1/2<1-ce^{k}e^{-CK}$, then we arrive at
\[
A(t)\le ce^{-Ct}\le Ce^{-kt},
\]
where we again have to take $K$ large enough to ensure this. In
particular, this means that $A(t)e^{kt}\le C$. As this $B(t)$ is the
maximum of $B(\tau)$ over $\tau\in[0,t]$, we get that $A(\tau)\le Ce^{-kt}$.

The other alternative which remains is that there is no $t\ge K$ for
which $B(t)$ attains a maximum on $[0,t]$ for $t\ge K$. In this case,
we deduce that $\sup_{t\ge0}B(t)=\sup_{t\in[0,K]}B(t)$ and the
exponential decay follows again.
\end{pf}

Before we close this section, let us point out that the exponential
decay of the gradient has the following consequence that we will use
later on for the estimates of the higher order derivatives.

%
\begin{Cor} Under the same assumptions as in Theorem~\ref{t7},
%
\begin{equation}
\label{ge20} \Prob(s<\OldTau)\le C\frac{\rho_{0}}{\sqrt{s}}\qquad\mbox
{for }s\in(0,t].
\end{equation}
\end{Cor}

\begin{pf} This follows by combining (\ref{ge10}), (\ref{ge30}) and
the fact that $\tilde{c}(s)/s$ is bounded (due to the gradient estimate).
\end{pf}

\section{Triple coupling}\label{s8}
\subsection{Basic idea}

We have just used coupling to prove the exponential convergence of
$\bar{p}$ to 0 in the $C^{1}$-topology. The next step in our analysis
is\vadjust{\goodbreak} the estimate of the decay of the Hessian of $\bar{p}$, which, from
the Ricci flow equation, implies the convergence of the curvature to a
constant. The basic idea starts with writing
\[
\bigl\langle\operatorname{Hess} \bar{p}_{t}(z)\xi,\xi\bigr\rangle=\lim
_{\rho
_{0}\to0}\frac{\bar{p}_{t}(\gamma(-\rho_{0}))-2\bar{p}_{t}(z)+\bar
{p}_{t}(\gamma(\rho_{0}))}{\rho_{0}^{2}},
\]
where $\xi$ is a unit vector at $z$, and $\gamma$ is a geodesic
running at unit speed started (at $t=0$) at $z$ with velocity $\xi$.
Now we are concerned with three points, $x=\gamma(-\rho_{0})$,
$y=\gamma(\rho_{0})$, and the middle point $z$. As in the gradient
estimate case, we want to write $\bar{p}(x)$, $\bar{p}(y)$ and $\bar
{p}(z)$ as integrals of some functions of the associated Brownian
motions and then use probabilistic estimates to find bounds for $\bar
{p}_{t}(\gamma(-\rho_{0}))-2\bar{p}_{t}(z)+\bar{p}_{t}(\gamma(\rho
_{0}))$ in terms of $\rho_{0}$.

There is very little literature on this idea, though it certainly seems
that this probabilistic tool is quite useful for estimating
second-order derivatives for evolution equations. The only reference to
this approach we are aware of is in \cite{CranstonTripleCouple}, where
it is essentially used to estimate the Hessian of harmonic functions on
Euclidean domains.

To make this idea more precise, we will develop a mechanism of triple
coupling (i.e., a coupling of three particles, as opposed to just
two). We will use mirror coupling for the processes corresponding to
the particles $x$ and $y$, taking them as time changed Brownian
motions, as in the previous section. Now we wish to include a third
particle, namely $z$, which we want to couple together with $x$ and
$y$. It is natural to want to have this ``middle particle'' remains on
the geodesic joining the other two as it is pictured in Figure~\ref
{Figtriple}.

%
\begin{figure}[b]

\includegraphics{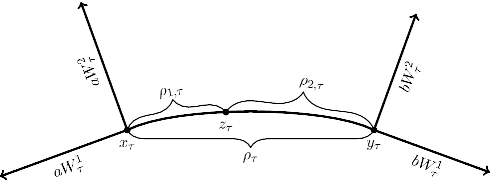}

\caption{The configuration of the three particles $x_\tau,y_\tau$
being mirror coupled and $z_\tau$ on the geodesic between them.}
\label{Figtriple}
\end{figure}
We will see that this is possible (at least in the cases we are
considering) if we allow it to evolve as time-changed Brownian motion,
possibly with drift along the direction of the geodesic.

Instead of starting with a time-changed Brownian motion with a drift,
$z_{\tau}$ and then trying to figure out the time change and drift
necessary so that it stays on the geodesic, we do it the other way
around. Namely, since we want the particle $z_{\tau}$ to move on the
geodesic, we determine the conditions on the distance to one of the
other points so that the corresponding point on the geodesic is a
time-changed Brownian motion with a drift along the geodesic. For the
purpose of the Hessian estimates, and in light of the gradient decay,
this will be sufficient.

\subsection{Rigorous approach}

Assume we start with an arbitrary Riemannian surface $M$ and that
$x_{\tau}$, $y_{\tau}$ run as time-changed Brownian motions with the
time changes $a$ and $b$, as above in Figure~\ref{Figtriple}. The
idea is that the middle point $z_{\tau}$ on the geodesic joining
$x_{\tau}$ and $y_{\tau}$ is completely described by specifying the
distance $\rho_{1,\tau}$ from $z_{\tau}$ to one of the ends, say
$x_{\tau}$. We use a mirror coupling of the particles $x_{\tau}$ and
$y_{\tau}$ and $\rho_{1,\tau}$ will be described in terms of a
real-valued SDE. In addition to~$\rho_{1}$, we will also consider
$\rho_{2}$, which in intuitive terms is just the distance from the
middle particle $z_{\tau}$ to $y_{\tau}$. We are seeking several key
symmetry properties which will play an important role in the economy of
the Hessian estimates to follow.

In what follows, as always, fix a time horizon $t>0$, and assume that
$a=a(\tau,x,y,\rho_{1},\rho_{2})$ and $b=b(\tau,x,y,\rho_{1},\rho
_{2})$ are two positive functions defined on $[0,t]\times M\times
M\times[0,\infty)\times[0,\infty)$, which will be time changes for
the processes $x_{\tau}$ and $y_{\tau}$. To describe this, again
denote by $m_{x,y}\dvtx T_{x}M\to T_{y}M$ the mirror map, that is, the
parallel transport along the minimal unit speed geodesic $\gamma
_{x,y}$ joining $x$ and~$y$ (assuming that $x$, $y$ are not at each
other's cut locus) followed by the reflection about the orthogonal
direction to the geodesic at $y$.

The system we start with is the following:
%
\begin{equation}
\label{tc1} \qquad\cases{
\displaystyle dx_{\tau} = a(\tau) \Biggl[ \Biggl[\sum
_{i=1}^2 \mathfrak{e}_{i}(x_{\tau})
\circ dW_{\tau}^i \Biggr] \Biggr],
\vspace*{3pt}\cr
\displaystyle dy_{\tau} = b(
\tau) \Biggl[ \Biggl[\sum_{i=1}^2
\Psi_{\tau
}\bigl[\mathfrak{e}_{i}(y_{\tau})\bigr]
\circ d W_{\tau}^i \Biggr] \Biggr],
\vspace*{3pt}\cr
\displaystyle d\rho_{1,\tau}=-a(
\tau)\sum_{i=1}^{2}\bigl\langle\mathfrak
{e}_{i}(x_{\tau}),\dot{\gamma}_{\tau}(0) \bigr\rangle
\,d W^{i}_{\tau
}+\alpha(\tau)\,dW^{3}_{\tau}+
\beta(\tau) \,d\tau,
\vspace*{3pt}\cr
\displaystyle d\rho_{2,\tau}=b(\tau)\sum
_{i=1}^{2}\bigl\langle\Psi_{\tau}\bigl[
\mathfrak{e}_{i}(y_{\tau})\bigr],\dot{\gamma}_{\tau}
\bigl(l(\tau)\bigr) \bigr\rangle \,d W^{i}_{\tau}+\tilde{\alpha}(
\tau)\,dW^{3}_{\tau}+\tilde{\beta}(\tau) \,d\tau,}
\end{equation}
where $\Psi_{\tau}=m_{x_{\tau},y_{\tau}}\mathfrak{e}(x_{\tau
})\mathfrak{e}(y_{\tau})^{-1}$ is the reflection map acting on
$T_{y_{\tau}}M$, $\gamma_{\tau}$ is the minimal geodesic running at
unit speed from $x_{\tau}$ to $y_{\tau}$, and $W^{3}$ is a
one-dimensional Brownian motion independent of $(W^{1},W^{2})$. As a
notation, let $l(\tau)$ be the length of the geodesic $\gamma_{\tau
}$. Here, we do not specify what the functions $\alpha$, $\tilde
{\alpha}$, $\beta$, $\tilde{\beta}$ are as we will do this along
the way, depending on the properties we want to reveal. They are
defined, like $a$ and $b$, on $[0,t]\times M\times M\times[0,\infty
)\times[0,\infty)$. The equations for $\rho_{1}$ and $\rho_{2}$ can
be thought of as the equations of the distances from the middle point
$z_{\tau}$ to $x_{\tau}$ and $y_{\tau}$, as indicated in the
previous section, and also as discussed for the coupling in \cite
{Elton}, Section~6.6. Notice here an important point, namely, since
\[
\bigl\langle\Psi_{\tau}\mathfrak{e}_{i}(y_{\tau}),
\dot{\gamma}_{\tau
}\bigl(l(\tau)\bigr) \bigr\rangle=-\bigl\langle
\mathfrak{e}_{i}(x_{\tau}),\dot{\gamma}_{\tau}(0)
\bigr\rangle,
\]
the last equation of (\ref{tc1}) can be rewritten as
%
\begin{equation}
\label{tc1b} d\rho_{2,\tau}=-b(\tau)\sum_{i=1}^{2}
\bigl\langle\mathfrak{e}_{i}(x_{\tau}),\dot{
\gamma}_{\tau}(0) \bigr\rangle \,d W^{i}_{\tau
}+\tilde{
\alpha}(\tau)\,dW^{3}_{\tau}+\tilde{\beta}(\tau) \,d\tau.
\end{equation}

We should also point out that to be in tune with the system (\ref
{e1b}) we should take $\sqrt{2}a$ instead of $a$ and $\sqrt{2}b$
instead of $b$. Since this is not important for this section and to
avoid carrying around an extra $\sqrt{2}$ factor, we will work with
the system in the form (\ref{tc1}).

There is no problem with the existence of a solution for the system
(\ref{tc1}) (as long as the entries $a,b,\alpha,\beta,\tilde
{\alpha}$ and $\tilde{\beta}$ are smooth) up to the stopping time
$\mathcal{T}$, which is the first time $\tau$ when $\rho_{1,\tau
}\rho_{2,\tau}$ hits $0$ or when $d(x_{\tau},y_{\tau})$ hits a
(small) $r_{0}$ smaller than the injectivity radius (with respect to
the background metric $h$). This way we have a well-defined system and
do not have to worry about the extension beyond the cut locus, as we
did in the previous (two particle) coupling case. From now on, during
this section we will assume that the time in the system (\ref{tc1})
is run until $\mathcal{T}$.

The object of interest to us is the process $(x,y,\rho_{1},\rho
_{2})$. It is clear that this is a diffusion, and it is a relatively
straightforward task to determine that the generator of $(x,y,\rho
_{1},\rho_{2})$ is
\begin{eqnarray*}
&&\frac{a^{2}}{2}\Delta_{x}+\frac{b^{2}}{2}
\Delta_{y} +\frac
{a^{2}+\alpha^{2}}{2}\partial_{\rho_{1}}^{2}+
\frac{b^{2}+\tilde
{\alpha}^{2}}{2}\partial_{\rho_{2}}^{2}+ab\langle
m_{xy}X_{1,i},Y_{2,j} \rangle X_{1,i}Y_{2,j}
\\
&&\qquad{}-a^{2} \bigl\langle X_{1,i},\dot{\gamma}_{x,y}(0)
\bigr\rangle X_{1,i}\partial_{\rho_{1}}-ab \bigl\langle
X_{1,i},\dot{\gamma}_{x,y}(0) \bigr\rangle X_{1,i}
\partial_{\rho_{2}}
\\
&&\qquad{} -a b \bigl\langle X_{1,i},\dot{\gamma}_{x,y}(0) \bigr
\rangle m_{x,y}X_{1,i}\partial_{\rho_{1}} -
b^{2} \bigl\langle X_{1,i},\dot{\gamma}_{x,y}(0)
\bigr\rangle m_{x,y}X_{1,i}\partial_{\rho_{2}}
\\
&&\qquad{} +\Biggl(\alpha\tilde{\alpha}-ab\sum_{i=1}^{2}
\bigl\langle X_{1,i},\dot{\gamma}_{x,y}(0) \bigr
\rangle^{2}\Biggr)\partial_{\rho_{1}}\partial_{\rho
_{2}}+
\beta\partial_{\rho_{1}}+\tilde{\beta}\partial_{\rho_{2}},
\end{eqnarray*}
with $X_{1,i}$, $i=1,2$ being an orthonormal basis of $T_{x}M$ and
$Y_{2,j}$, $j=1,2$ an orthonormal basis of $T_{y}M$. In fact, we can
choose $X_{1,1}=\dot{\gamma}_{x,y}(0)$ and $X_{1,2}=\xi_{1}\in
T_{x}M$, which is perpendicular to $\dot{\gamma}_{x,y}(0)$.
Similarly, choose $Y_{2,1}=\dot{\gamma}_{y,x}(0)$ and $Y_{2,2}=\xi
_{2}=m_{x,y}\xi_{1}$, or, in simpler terms, the parallel transport of
$\xi_{1}$ along the geodesic $\gamma_{x,y}$. With these choices, the
generator simplifies to
%
\begin{eqnarray}
\label{tc2} \mathcal{L}&=&\frac{a^{2}}{2}\Delta_{x}+
\frac{b^{2}}{2}\Delta_{y} +\frac{a^{2}+\alpha^{2}}{2}\partial_{\rho_{1}}^{2}+
\frac{b^{2}+\tilde{\alpha}^{2}}{2}\partial_{\rho_{2}}^{2}\nonumber
\\
&&{}  +ab\bigl( \dot{
\gamma}_{x,y}(0)\dot{\gamma}_{y,x}(0)+\xi_{1}
\xi_{2}\bigr)
-a^{2} \dot{\gamma}_{x,y}(0)\partial_{\rho_{1}}-a b
\dot{\gamma}_{y,x}(0)\partial_{\rho_{1}}
\\
&&{}-ab \dot{\gamma}_{x,y}(0)\partial_{\rho_{2}}-b^{2}
\dot{\gamma}_{y,x}(0)\partial_{\rho_{2}}\nonumber
+(\alpha\tilde{\alpha}-ab)\partial_{\rho_{1}}\partial_{\rho
_{2}}+
\beta\partial_{\rho_{1}}+\tilde{\beta}\partial_{\rho_{2}}.\nonumber
\end{eqnarray}

The first property we want to see is that $\rho_{1}+\rho_{2}=\rho$.
This property is nothing but the geometric picture that $\rho_{1}$ is
the distance from $z_{\tau}$ to $x_{\tau}$ while $\rho_{2}$ is the
distance between $z_{\tau}$ to $y_{\tau}$.

To do this, we recall that the distance $\rho_{\tau}$ between the
mirror-coupled processes $x_{\tau}$ and $y_{\tau}$ is given by
%
\begin{equation}
\label{tc26} d\rho_{\tau}=-\bigl(a(\tau)+b(\tau)\bigr)\sum
_{i=1}^{2}\bigl\langle\mathfrak{e}_{i}(x_{\tau}),
\dot{\gamma}_{\tau}(0) \bigr\rangle \,d W^{i}_{\tau
}+
\frac{1}{2}\mathcal{I}(\tau) \,d\tau,
\end{equation}
where $\mathcal{I}$ is the index form of the Jacobi field $J(\tau)$
along the geodesic $\gamma_{\tau}$ which, at the endpoints, has
values $aE$ and $bE$. We use the notation $E$ for the parallel
translation of $\xi_{1}\in T_{x}M$ along the geodesic joining $x$ and
$y$. The index form is computed as
\[
\mathcal{I}(J,J)=\int_{0}^{l(\gamma)}\bigl\llvert
\dot{J}(u)\bigr\rrvert^{2}+\bigl\langle R\bigl(\dot{\gamma}(u),J(u)
\bigr)\dot{\gamma}(u),J(u) \bigr\rangle \,du,
\]
with $l(\gamma)$ being the length of the geodesic $\gamma$. Here, the
curvature tensor is the standard tensor curvature given as in \cite{Chee}
\[
R(X,Y)=\nabla_{X}\nabla_{Y}-\nabla_{Y}
\nabla_{X}-\nabla_{[X,Y]}.
\]
Furthermore, a simple integration by part gives that
%
\begin{equation}
\label{tc101} \mathcal{I}(J,J)=\bigl\langle\dot{J}\bigl(l(\gamma)\bigr),J
\bigl(l(\gamma)\bigr) \bigr\rangle-\bigl\langle\dot{J}(0),J(0)\bigr
\rangle.
\end{equation}
On the other hand, from (\ref{tc1}),
\begin{eqnarray*}
&& d(\rho_{1,\tau}+\rho_{2,\tau})
\\
&&\qquad =-\bigl(a(\tau)+b(\tau)\bigr)\sum
_{i=1}^{2}\bigl\langle\mathfrak{e}_{i}(x_{\tau}),
\dot{\gamma}_{\tau
}(0) \bigr\rangle \,d W^{i}_{\tau}+
\bigl(\alpha(\tau)+\tilde{\alpha}(\tau)\bigr)\,dW^{3}_{\tau}
\\
&&\quad\qquad{} +
\bigl(\beta(\tau)+\tilde{\beta}(\tau)\bigr)\,d\tau.
\end{eqnarray*}
We clearly see here that $\rho_{\tau}$ and $\rho_{1,\tau}+\rho
_{2,\tau}$ have the same martingale part if $\tilde{\alpha}=-\alpha
$. The choice for $\beta$ and $\tilde{\beta}$ is provided by the
following result.

%
\begin{teo}\label{tct1}
Assume that
%
\begin{equation}
\label{tc31} \cases{ \displaystyle\tilde{\alpha}=-\alpha,
\vspace*{3pt}\cr
\displaystyle\beta(\tau,x,y,
\rho_{1},\rho_{2})=\frac{1}{2}\int
_{0}^{\rho
_{1}} \bigl(\bigl\llvert\dot{J}(u)\bigr
\rrvert^{2}+\bigl\langle R\bigl(\dot{\gamma}(u),J(u)\bigr)\dot{
\gamma}(u),J(u) \bigr\rangle\bigr)\,du,
\vspace*{3pt}\cr
\displaystyle\tilde{\beta}(\tau,x,y,
\rho_{1},\rho_{2})=\frac{1}{2}\int
_{l(\gamma)-\rho_{2}}^{l(\gamma)} \bigl(\bigl\llvert\dot{J}(u)\bigr
\rrvert^{2}+\bigl\langle R\bigl(\dot{\gamma}(u),J(u)\bigr)\dot{
\gamma}(u),J(u) \bigr\rangle\bigr)\,du,}\hspace*{-28pt}
\end{equation}
where $J$ is the Jacobi field along the geodesic $\gamma$ from $x$ to
$y$ and having values $aE$ at $0$ and $bE$ at $l(\gamma)$.

If in addition, $\rho_{1,0}=\rho_{2,0}=\rho_{0}/2$, then almost
surely $\rho_{\tau}=\rho_{1,\tau}+\rho_{2,\tau}$.
\end{teo}

\begin{pf}
Take $\tilde{\rho}_{1,\tau}=\rho_{\tau}-\rho_{2,\tau}$. It is
clear now that we have
\[
d(\tilde{\rho}_{1,\tau}-\rho_{1,\tau})=\int_{0}^{\tilde{\rho
}_{1,\tau}}
A(u)\,du-\int_{0}^{\rho_{1,\tau}} A(u)\,du
\]
with
\[
A(u)=\tfrac{1}{2} \bigl[\bigl\llvert\dot{J}(u)\bigr\rrvert
^{2}+\bigl\langle R\bigl(J(u),\dot{\gamma}(u)\bigr)\dot{
\gamma}(u),J(u) \bigr\rangle \,du \bigr].
\]
From here, the fact that $\tilde{\rho}_{1,0}=\rho_{1,0}$ (or $\tilde
{\rho}_{1,0}-\rho_{1,0}=0$) and standard application of Gronwall's
inequality leads to $\tilde{\rho}_{1,\tau}=\rho_{1,\tau}$, which
is what we want.
\end{pf}

We return now to the case where the curvature is constant and start
with \cite{DoC}, Lemma 3.4, which says that
%
\begin{equation}
\label{tc5} R(X,Y)Z=-r\bigl(\langle X,Z \rangle Y -\langle Y,Z \rangle X
\bigr).
\end{equation}
We should point out that do Carmo \cite{DoC} takes the curvature to be
given by the negative of the curvature we consider here, or for that
matter other people as, for instance, \cite{Chee}. Then the Jacobi field
equation becomes
\[
\ddot{J}-R(\dot{\gamma},J)\dot{\gamma}=0
\]
or equivalently,
%
\begin{equation}
\label{tc17} \ddot{J}+rJ-r\langle\dot{\gamma},J \rangle\dot{\gamma}=0.
\end{equation}
Since this Jacobi field is perpendicular to the geodesic, it follows that
\[
\cases{ \ddot{J}+rJ=0,
\vspace*{3pt}\cr
J(0)=aE,
\vspace*{3pt}\cr
J\bigl(l(\gamma)\bigr)=bE.}
\]
The solution is
%
\begin{equation}
\label{tc30} J(s)=\bigl(aw_{1}(s)+bw_{2}(s)\bigr)E(s)
\qquad\mbox{for }s\in\bigl[0,l(\gamma)\bigr],
\end{equation}
where $w_{1},w_{2}$ are defined on the interval $[0,l(\gamma)]$ by the
following ODEs:
%
\begin{equation}
\label{tc21} \cases{ \ddot{w}_{1}+rw_{1}=0,
\vspace*{3pt}\cr
w_{1}(0)=1,
\vspace*{3pt}\cr
w_{1}\bigl(l(\gamma)\bigr)=0,} \quad
\mbox{and}\quad\cases{ \ddot{w}_{2}+rw_{2}=0,
\vspace*{3pt}\cr
w_{2}(0)=0,
\vspace*{3pt}\cr
w_{2}\bigl(l(\gamma)\bigr)=1.}
\end{equation}

Combining now (\ref{tc101}) and the Jacobi field just considered
reveals that
\begin{eqnarray*}
&& \int_{0}^{s}\bigl\llvert\dot{J}(u)\bigr
\rrvert^{2}-\bigl\langle R\bigl(J(u),\dot{\gamma}(u)\bigr)\dot{
\gamma}(u),J(u) \bigr\rangle \,du
\\
&&\qquad =\int_{0}^{s}\bigl
\llvert\dot{J}(u)\bigr\rrvert^{2}-r\bigl\llvert J(u)\bigr\rrvert
^{2}\,du
\\
&&\qquad =\bigl\langle\dot{J}(s),J(s)\bigr\rangle-\bigl\langle\dot
{J}(0),J(0)\bigr
\rangle
\\
&&\qquad =\bigl(aw_{1}(s)+bw_{2}(s)\bigr) \bigl(a
\dot{w}_{1}(s)+b\dot{w}_{2}(s)\bigr)-b\bigl(a\dot
{w}_{1}(0)+b\dot{w}_{2}(0)\bigr)
\end{eqnarray*}
and
\begin{eqnarray*}
&& \int_{s}^{l(\gamma)}\bigl\llvert\dot{J}(u)\bigr
\rrvert^{2}-\bigl\langle R\bigl(J(u),\dot{\gamma}(u)\bigr)\dot{
\gamma}(u),J(u) \bigr\rangle \,du
\\
&&\qquad =\bigl\langle\dot{J}\bigl(l(\gamma)\bigr),J
\bigl(l(\gamma)\bigr)\bigr\rangle-\bigl\langle\dot{J}(s),J(s)\bigr
\rangle
\\
&&\qquad =b\bigl(a\dot{w}_{1}\bigl(l(\gamma)\bigr)+b\dot{w}_{2}
\bigl(l(\gamma)\bigr)\bigr)-\bigl(aw_{1}(s)+bw_{2}(s)\bigr)
\bigl(a\dot{w}_{1}(s)+b\dot{w}_{2}(s)\bigr).
\end{eqnarray*}
A direct consequence of these formulae and the fact that
$w_{2}(s)=w_{1}(l(\gamma)-s)$, plus a few elementary manipulations,
results in
\begin{eqnarray*}
&& \int_{l(\gamma)-s}^{l(\gamma)}\bigl\llvert\dot{J}(u)\bigr
\rrvert^{2}-\bigl\langle R\bigl(J(u),\dot{\gamma}(u)\bigr)\dot{
\gamma}(u),J(u) \bigr\rangle \,du
\\
&&\qquad =\bigl(bw_{1}(s)+aw_{2}(s)
\bigr) \bigl(b\dot{w}_1(s)+a\dot{w}_2(s)\bigr)-b\bigl(b
\dot{w}_1(0)+a\dot{w}_2(0)\bigr).
\end{eqnarray*}
Summarizing, the choices of $\beta$ and $\tilde{\beta}$ from (\ref
{tc31}) in the case of constant curvature become more explicit as
%
\begin{equation}
\label{tc33} \cases{ \beta=\frac{1}{2} \bigl(\bigl(aw_{1}(
\rho_{1})+bw_{2}(\rho_{1})\bigr)
\cr
\hspace*{28pt}{}\times \bigl(a\dot
{w}_1(\rho_{1})+b\dot{w}_2(
\rho_{1})\bigr)-a\bigl(a\dot{w}_1(0)+b\dot
{w}_2(0)\bigr) \bigr),
\vspace*{3pt}\cr
\tilde{\beta}=\frac{1}{2} \bigl(
\bigl(bw_{1}(\rho_{2})+aw_{2}(\rho
_{2})\bigr)
\cr
\hspace*{28pt}{}\times \bigl(b\dot{w}_1(\rho_{2})+a
\dot{w}_2(\rho_{2})\bigr)-b\bigl(b\dot{w}_1(0)+a
\dot{w}_2(0)\bigr)\bigr).}
\end{equation}
It goes without saying that here $a$ and $b$ are evaluated at $(\tau
,x,y,\rho_{1},\rho_{2})$.

We say that a function $f(\tau,x,y,\rho_{1},\rho_{2})$ is symmetric
in $\rho_{1}$ and $\rho_{2}$ if
$f(\tau,x,y,\rho_{1},\rho_{2})=f(\tau,x,y,\rho_{2},\rho_{1})$.

Before we move on to another property of the diffusion $(x,y,\rho
_{1},\rho_{2})$, we close the discussion so far with the following
property of the choices of $\beta$ and $\tilde{\beta}$ from~(\ref{tc33}):
\[
\begin{tabular}{p{260pt}}
If $a$ and $b$ are equal and symmetric in $\rho_{1}$ and $\rho
_{2}$, then $\beta(\tau,x,y,\rho_{1},\rho_{2})=
\tilde{\beta}(\tau,x,y,\rho_{2},\rho_{1})$.
\end{tabular}
\]

A symmetry which plays a crucial role in the Hessian estimates is the following.

%
\begin{teo}\label{tct2}
If, in equation (\ref{tc1}), we take
\[
\cases{ a\mbox{ and }\alpha\mbox{ symmetric in }\rho_{1}
\mbox{ and }\rho_{2},
\vspace*{3pt}\cr
b=a,
\vspace*{3pt}\cr
\tilde{\alpha}=-\alpha,
\vspace*{3pt}\cr
\tilde{
\beta}(\tau,x,y,\rho_{1},\rho_{2})=\beta(\tau,x,y,\rho
_{2},\rho_{1}),
\vspace*{3pt}\cr
\rho_{1,0}=
\rho_{2,0},}
\]
then the processes $(x,y,\rho_{1},\rho_{2})$ and $(x,y,\rho_{2},\rho
_{1})$ have the same law. In particular, the processes $(x,y,\rho
_{1})$ and $(x,y,\rho_{2})$ have the same law.
\end{teo}

\begin{pf} Although this is almost trivial, we say a word about it. If
$\mathcal{L}$ is the generator of a diffusion $\omega_{\tau}$ on a
manifold $\mathcal{M}$ and $\pi\dvtx\mathcal{M}\to\mathcal{M}$ is such
that for any smooth function $\phi\dvtx\mathcal{M}\to\Re$,
\[
\mathcal{L}(\phi\circ\pi)=(\mathcal{L}\phi)\circ\pi,
\]
then uniqueness of the diffusion implies that $\omega$ and $\pi
(\omega)$ have the same law. This can be easily seen from the
martingale characterization of the law of the diffusion. We apply this
to the operator $\mathcal{L}$ from (\ref{tc2}) and the map $\pi
(x,y,\rho_{1},\rho_{2})=(x,y,\rho_{2},\rho_{1})$. The rest follows.
\end{pf}

Notice that [cf. (\ref{tc33})], the choices of $\beta$ and $\tilde
{\beta}$ from Theorem~\ref{tct1} are actually consistent with the
conditions of Theorem~\ref{tct2} under the assumptions that $a$ and
$b$ are equal and symmetric.

The ``middle particle'' process we are interested is
%
\begin{equation}
\label{tc14} z_{\tau}=\gamma_{x_{\tau},y_{\tau}}(\rho_{1,\tau}).
\end{equation}

The symmetry between $\rho_{1}$ and $\rho_{2}$ should be interpreted
as saying that the reflection of the process $z_{\tau}$ with respect
to the middle point of the geodesic $\gamma_{x_{\tau},y_{\tau}}$ has
the same law as $z_{\tau}$ itself.

Our next objective is the law of $z_{\tau}$. Before we jump into the
heart of the matter, we take up a discussion on the following class of
vector fields that are the main actors in our computation.

Assume we have a geodesic $\gamma$ from $x$ to $y$ with length $l$ and
consider a smooth, two-parameter geodesic perturbation $f\dvtx
(-\epsilon,\epsilon)\times(-\epsilon,\epsilon)\times[0,l]\to M$ of
$\gamma
$, that is, $f(0,0,s)=\gamma(s)$ and for each fixed choice of $u$ and
$v$, the curve $s\to f(u,v,s)$ is a geodesic. One of the things we want
to understand is the field
\[
\mathcal{H} (s)=\frac{D}{du}\frac{D}{dv}f(u,v,s)
\Big|_{u=0,v=0}.
\]

Let $J_{v}(s)=\frac{D}{dv}f(u,v,s)\mid_{u=v=0}$ be the Jacobi field
obtained by differentiating $f$ with respect to $v$ and we will use
$J_{v}(u,s)=\frac{D}{dv}f(u,v,s)\mid_{v=0}$ as the Jacobi field
which is still depending on $u$. Similarly, let $J_{u}(s)=\frac
{D}{du}f(u,v,s)\mid_{u=v=0}$ be the Jacobi field obtained by
differentiating $f$ with respect to $u$ and use $J_{u}(v,s)=\frac
{D}{du}f(u,v,s)\mid_{u=0}$. In order to determine the equation
satisfied by $\mathcal{H}
$, we recall here \cite{DoC}, Lemma 4.1, which asserts that for any
two-parameter family $g(a,b)$ and vector field $V$ along $g$,
%
\begin{equation}
\label{tc16} \frac{D}{da}\frac{D}{db}V-\frac{D}{db}
\frac{D}{da}V=-R \biggl(\frac
{Dg}{db},\frac{Dg}{da} \biggr)V.
\end{equation}

Now, what we want to do is to find a differential equation satisfied by
$\mathcal{H}
$. As pointed out already, $\mathcal{H}
(s)=\frac{D}{du}J_{v}(u,s)\mid_{u=0}$ and starting with (\ref
{tc17}) for $J_{v}(u)$, namely,
\[
\frac{D^{2}}{ds^2}J_{v}(u)+rJ_{v}(u)-r\langle\dot{
\gamma},J_{v} \rangle\dot{\gamma}=0
\]
we take the derivative with respect to $u$ at $u=0$ to arrive at
\[
\frac{D}{du}\frac{D^{2}}{ds^2}J_{v}+r\mathcal{H} -r \biggl
\langle\frac{D}{du}\dot{\gamma},J_{v} \biggr\rangle\dot{
\gamma}-r\langle\dot{\gamma},\mathcal{H} \rangle\dot{\gamma}-r\langle
\dot{
\gamma},J_{v} \rangle\frac
{D}{du}\dot{\gamma}=0.
\]
To move forward, use that $\frac{D}{du}\dot{\gamma}\mid
_{u=0}=\frac{D}{ds}\frac{D}{du}\gamma\mid_{u=0}=\dot{J}_{u}$ to
re-write the previous equation as
\[
\frac{D}{du}\frac{D^{2}}{ds^2}J_{v}+r\mathcal{H} -r\langle
\mathcal{H},\dot{\gamma} \rangle\dot{\gamma}-r\langle\dot{J}_{u},J_{v}
\rangle\dot{\gamma}-r\langle\dot{\gamma},J_{v} \rangle
\dot{J}_{u}=0.
\]

Our task now is to commute the derivatives with respect to $u$ and $s$.
For this, use~(\ref{tc16}) and (\ref{tc5}) to justify that at $u=0$,
{\renewcommand{\theequation}{*}
\begin{eqnarray}\label{eq*}
\frac{D}{du}\frac{D^{2}}{ds^2}J_{v}&=&
\frac{D}{ds}\frac{D}{du}\frac
{D}{ds}J_{v}-R \biggl(
\frac{Df}{ds},\frac{Df}{du} \biggr)\dot{J}_{v}\nonumber
\\
&=& \frac{D}{ds}\frac{D}{du}\frac{D}{ds}J_{v}-R (\dot{
\gamma},J_{u} )\dot{J}_{v}
\\
&=&\frac{D}{ds}\frac{D}{du}\frac{D}{ds}J_{v}+ r
\bigl(\langle\dot{\gamma},\dot{J}_{v} \rangle J_{u}-
\langle J_{u},\dot{J}_{v} \rangle\dot{\gamma} \bigr)\nonumber
\end{eqnarray}}\setcounter{equation}{65}%
and once again employing (\ref{tc16}),
{\renewcommand{\theequation}{**}
\begin{eqnarray}\label{eq**2}
\frac{D}{ds}\frac{D}{du}\frac{D}{ds}J_{v}&=&
\frac
{D^{2}}{ds^{2}}\frac{D}{du}J_{v}-\frac{D}{ds} \biggl(R
\biggl(\frac
{Df}{ds},\frac{Df}{du} \biggr)J_{v} \biggr)\nonumber
\\
&=& \ddot{\mathcal{H} }-\frac{D}{ds} \bigl(R (\dot{\gamma},J_{u}
)J_{v} \bigr)\nonumber
\nonumber\\[-8pt]\\[-8pt]\nonumber
&=&\ddot{\mathcal{H} }+r\frac{D}{ds} \bigl(\langle\dot{\gamma},
J_{v}\rangle J_{u}-\langle J_{u},
J_{v}\rangle\dot{\gamma} \bigr)
\\
&=&\ddot{\mathcal{H} }+r \bigl(\langle\dot{\gamma}, \dot{J}_{v}
\rangle J_{u}+\langle\dot{\gamma}, J_{v}\rangle
\dot{J}_{u}-\langle\dot{J}_{u}, J_{v}\rangle
\dot{\gamma}-\langle J_{u}, \dot{J}_{v}\rangle\dot{\gamma}
\bigr).\nonumber
\end{eqnarray}}\setcounter{equation}{65}%
Putting together (\ref{eq*}) and (\ref{eq**2}), we obtain
\[
\frac{D}{du}\frac{D^{2}}{ds^2}J_{v}=\ddot{\mathcal{H} }+r
\bigl(2\langle\dot{\gamma}, \dot{J}_{v}\rangle J_{u}+
\langle\dot{\gamma}, J_{v}\rangle\dot{J}_{u}-\langle
\dot{J}_{u}, J_{v}\rangle\dot{\gamma}-2\langle
J_{u}, \dot{J}_{v}\rangle\dot{\gamma} \bigr),
\]
and finally since the boundary conditions are pretty straightforward we
get the following:
%
\begin{equation}
\label{tc18} \qquad\cases{
\displaystyle \ddot{\mathcal{H}}+r\mathcal{H}-r\langle\mathcal
{H},\dot{\gamma} \rangle\dot{\gamma}+2r \bigl(\langle\dot{\gamma}, \dot{J}_{v}
\rangle J_{u}-\langle\dot{J}_{u}, J_{v}\rangle
\dot{\gamma}-\langle J_{u}, \dot{J}_{v}\rangle\dot{\gamma}
\bigr)=0,
\vspace*{3pt}\cr
\displaystyle\mathcal{H}(0)= \frac{D}{du}\frac{D}{dv}f(u,v,0)
\Big|_{u=v=0},
\vspace*{3pt}\cr
\displaystyle\mathcal{H}(l)= \frac{D}{du}\frac{D}{dv}f(u,v,l)
\Big|_{u=v=0}.}
\end{equation}

We discussed the case of a two-parameter perturbation of the geodesic
$\gamma$ in the form $f(u,v,s)$ but exactly the same argument works
also for the case where $f(u,s)$ is a perturbation with geodesics of
$\gamma$, and we consider the field
\[
\mathcal{H} (s)=\frac{D^{2}}{du^{2}}f(u,s)\Big|_{u=0}.
\]
The main result from the argument above then gives that
%
\begin{equation}
\label{tc19} \cases{
\displaystyle\ddot{\mathcal{H} }+r\mathcal{H} -r\langle\mathcal{H},
\dot{\gamma} \rangle\dot{\gamma}+2r \bigl(\langle\dot{\gamma}, \dot
{J}_{u}\rangle J_{u}-2\langle J_{u},
\dot{J}_{u}\rangle\dot{\gamma} \bigr)=0,
\vspace*{3pt}\cr
\displaystyle\mathcal{H}(0)=
\frac{D^{2}}{du^{2}}f(u,0)\Big|_{v=0},
\vspace*{3pt}\cr
\displaystyle\mathcal{H}(l)=
\frac{D^{2}}{du^{2}}f(u,l)\Big|_{u=0},}
\end{equation}
with $J_{u}(s)=\frac{D}{du}f(u,s)\mid_{u=0}$.

The perturbation $g(u,v,s)$ that will appear below is slightly
different from the perturbation $f(u,v,s)$ considered above. To
describe it, take a unit speed geodesic $\gamma$ defined on $[0,l]$
and consider two geodesic curves, $\eta_{1,u}$ with $\eta
_{1,0}=\gamma(0)$ and another, $\eta_{2,v}$ so that $\eta
_{2,0}=\gamma(l)$. Let $g(u,v,\cdot)$ be the geodesic run at unit
speed from $\eta_{1,u}$ to $\eta_{2,v}$. One problem immediately
arising with this choice is that the parameter in the geodesic
direction, namely $s$, is no longer running in the interval $[0,l]$ and
this is the reason we have to treat it separately. Consequently, the
above calculations do not apply in the same way as they were carried
out in the case of $f(u,v,s)$.

To fix this, let us denote by $l(u,v)$, the length of the geodesic
$\gamma_{u,v}=\gamma_{\eta_{1,u},\eta_{v,2}}$ and reparametrize
this geodesic such that it has constant speed equal to $l(u,v)/l$. More
precisely if $\tilde{\gamma}_{u,v}$ is the reparametrized geodesic,
then $\gamma_{u,v}(s)=\tilde{\gamma}_{u,v}(s l/l(u,v))$. Now let
$f(u,v,s)=\tilde{\gamma}_{u,v}(s)$. Clearly, now the parameter $s$
for $f(u,v,s)$ runs in the interval $[0,l]$ and
\[
g(u,v,s)=f\bigl(u,v,sl/l(u,v)\bigr).
\]
Our interest is again in the understanding of the field $\mathcal
{K}(s)=\frac{D}{d v}\frac{D}{d u}g(u,v,s)\mid_{u=v=0}$. We do this
via the fact that $f(u,v,s)=g(u,v,sl(u,v)/l)$ and upon differentiation
with respect to $u$ to get
%
\begin{eqnarray}\label{tct100}
\frac{D }{\partial u}f(u,v,s)
&=& \frac{D }{\partial
u}g\bigl(u,v,sl(u,v)/l
\bigr)
\nonumber\\[-8pt]\\[-8pt]\nonumber
&&{} +\frac{s}{l} \biggl(\frac{d}{du}l(u,v) \biggr)
\frac{D}{\partial s}g\bigl(u,v,sl(u,v)/l\bigr)
\end{eqnarray}
and from this and the first variation formula \cite{Chee}, equation
(1.3), page~5, to get the relation between Jacobi field $J_{u}^{g}$ generated
by $g$ and $J_{u}^{f}$ as
\[
J_{u}^{f}(s)=J_{u}^{g}(s)-
\frac{s}{l}\bigl\langle\dot{\eta}_{1,0},\dot{\gamma}(0) \bigr
\rangle\dot{\gamma}(s).
\]
Similarly,
\[
J_{v}^{f}(s)=J_{v}^{g}(s)+
\frac{s}{l}\bigl\langle\dot{\eta}_{2,0},\dot{\gamma}(l) \bigr
\rangle\dot{\gamma}(s).
\]
Now taking the derivative with respect to $v$ in (\ref{tct100}), set
$u=v=0$ to obtain
\begin{eqnarray*}
\mathcal{H}(s)&=&\mathcal{K}(s)+\frac{s}{l}\bigl\langle\dot{\eta
}_{2,0},\dot{\gamma}(l) \bigr\rangle\dot{J}_{u}^{g}(s)-
\frac
{s}{l}\bigl\langle\dot{\eta}_{1,0},\dot{\gamma}(0) \bigr
\rangle\dot{J}_{v}^{g}(s)
\\
&&{} +\frac{s}{l} \biggl(
\frac{d^{2}}{dv \,du}l(u,v) \biggr)\bigg|_{u=v=0}\dot{\gamma}(s).
\end{eqnarray*}

The case of interest in the sequel is the case of geodesics $\eta
_{1,u}$ and $\eta_{2,v}$ such that $\dot{\eta}_{1,0}=E(0)$ and $\dot
{\eta}_{2,0}(l)=E(l)$. In this case, the second and the third terms
vanish while the last term is computed using the second variation
formula which is \cite{Chee}, equation (1.14), page~20. We also learn
that $J^{g}_{u}=J^{f}_{u}$ and $J^{g}_{v}=J^{f}_{v}$ and the last term becomes
\[
\biggl(\frac{d^{2}}{dv \,du}l(u,v) \biggr)\bigg|_{u=v=0}=\mathcal{I}
\bigl(J^{f}_{u},J^{f}_{v}\bigr)=
\frac{1}{2}\bigl(\dot{w}_{1}(l)-\dot{w}_{2}(0)\bigr),
\]
where in between we used a polarization argument for (\ref{tc101})
together with (\ref{tc30}) and~(\ref{tc21}). Thus, we get
%
\begin{equation}
\label{tc102} \mathcal{K}(s)=\mathcal{H}(s)-\frac{s}{2l}\bigl(
\dot{w}_{1}(l)-\dot{w}_{2}(0)\bigr)\dot{\gamma}(s).
\end{equation}

Another situation we encounter below is the following. Take $\eta
_{1,u}$ a geodesic starting at $\gamma(0)$ such that $\dot{\eta
}_{1,0}=E(0)$. Then we take $g(u,s)$ to be the geodesic $\gamma
_{u}(s)$ run at unit speed from $\eta_{1,u}$ to $\gamma(l)$. The
field we are interested in is $\mathcal{K}(s)=\frac
{D^{2}}{du^{2}}g(u,s)\mid_{u=0}$. With a very similar argument, we
can show that
%
\begin{equation}
\label{tc103} \mathcal{K}(s)=\mathcal{H}(s)+\frac{s}{l}
\dot{w}_{1}(0)\dot{\gamma}(s),
\end{equation}
where\vspace*{1pt} $\mathcal{H}(s)=\frac{D^{2}}{du^{2}}g(u,s l/l(u))\mid_{u=0}$
with $l(u)$ being the length of the geodesic from $\eta_{1,u}$ to
$\gamma(l)$.

Similarly, if we take $\eta_{2,v}$ the geodesic starting at $\gamma
(l)$, such that $\dot{\eta}_{2,0}=E(l)$ and $g(v,s)$ being the unit
speed geodesic joining $\eta_{2,v}$ to $\gamma(0)$, and $\mathcal
{K}(s)=\frac{D^{2}}{dv^{2}}g(v,s)\mid_{v=0}$ then
%
\begin{equation}
\label{tc104} \mathcal{K}(s)=\mathcal{H}(s)-\frac{s}{l}
\dot{w}_{2}(l)\dot{\gamma}(s)
\end{equation}
with $\mathcal{H}(s)=\frac{D^{2}}{dv^{2}}g(v,s l/l(v))\mid_{v=0}$
and $l(v)$ the length of the geodesic from $\eta_{2,v}$ to $\gamma(0)$.

We are finally ready for the next result.

%
\begin{teo}\label{tct3}
Assume that
%
\begin{eqnarray}\label{tc24}
&& \alpha(\tau,x,y,\rho_{1},\rho_{2}) = a(\tau,x,y,\rho_{1},\rho_{2})w_{1}(
\rho_{1})+b(\tau,x,y,\rho_{1},\rho_{2})w_{2}(\rho_{1}),\nonumber
\\
&& \theta(\tau,x,y,\rho_{1},\rho_{2})\nonumber
\\
&&\qquad = \beta(\tau,x,y, \rho_{1},\rho_{2})\nonumber
\\
&&\quad\qquad{}+\frac{\rho_{1}}{\rho} \biggl(\frac{a(\tau,x,y,\rho_{1},\rho
_{2})^{2}+b(\tau,x,y,\rho_{1},\rho_{2})^{2}}{2}\dot{w}_{1}(0)
\\
&&\hspace*{79pt} {}-a(\tau,x,y,\rho_{1},\rho_{2})b(\tau,x,y,\rho_{1},
\rho_{2}) \dot{w}_{1}(l) \biggr)\nonumber
\\
&&\quad\qquad{}+r \biggl( \int_{0}^{\rho_{1}}\bigl(a(\tau,x,y,
\rho_{1},\rho_{2})w_{1}(\sigma)+b(\tau,x,y,
\rho_{1},\rho_{2})w_{2}(\sigma)
\bigr)^{2}\,d\sigma\nonumber
\\
&&\hspace*{23pt}\quad\qquad{}-\frac{\rho_{1}}{l}\int_{0}^{l}
\bigl(a(\tau,x,y,\rho_{1},\rho_{2})w_{1}(
\sigma)+b(\tau,x,y,\rho_{1},\rho_{2})w_{2}(
\sigma)\bigr)^{2}\,d\sigma\biggr)\nonumber
\end{eqnarray}
with $w_{1}$ and $w_{2}$ defined by (\ref{tc21}). With these choices,
the process $z_{\tau}=\gamma_{x_{\tau},y_{\tau}}(\rho_{1,\tau})$
has the property that, for any smooth function $\phi$ on $M$,
%
\begin{equation}
\label{tc22} \phi(z_{\tau})-\int_{0}^{\tau}
\biggl(\frac{\alpha
^{2}(u)}{2}[\Delta\phi](z_{u})+\theta(u)\bigl\langle
\nabla\phi(z_{u}),\dot{\gamma}_{x_{u},y_{u}}(\rho_{1,u})\bigr\rangle
\biggr) \,du
\end{equation}
is a martingale with respect to the filtration generated by $W_{1}$,
$W_{2}$ and $W_{3}$, where inside the integral, $\alpha(u)$ and
$\theta(u)$ are shorthand for $\alpha$ and $\theta$ evaluated at
$(u,x_{u},y_{u},\rho_{1,u},\rho_{2,u})$. In other words, $z_{\tau}$
is a time-changed Brownian motion (with the time change given by
$\alpha$) with a drift in the geodesic direction from $x_{\tau}$ to~$y_{\tau}$.
\end{teo}

\begin{pf}
The idea of the proof is to start with the generator of the diffusion
$(x,y,\rho_{1})$ and a function $\phi$ and look at the process $\phi
(z_{\tau})$. More precisely, we find the bounded variation part of
this. It is clear that, in terms of the generator (\ref{tc2}), we
need to compute the action of each term of this expression on $\phi
(\gamma_{x,y}(s))$. Notice that the part which involves derivatives of
$\rho_{2}$ simply drops out in this calculation.

For simplicity, we will drop the dependence on $\tau$, $x$ and $y$ in
the notation and let $l=d(x,y)$. Thus, the geodesic $\gamma_{x,y}$
will appear as $\gamma$ if we do not prescribe otherwise. Let $E$
denote the parallel vector field along $\gamma$ which is obtained by
parallel translation of $\xi_{1}$.

Before we start the proof, let us mention that all geodesics appearing
in this proof are geodesics run at unit speed.

Now we take the terms one by one. Again for simplicity in writing, we
use $s$ instead of $\rho_{1}$ as the parameter in the geodesic direction.
\begin{longlist}[(2)]
\item[(1)] We write the Laplacian term as
\[
\Delta_{x}\bigl[\phi\bigl(\gamma_{x,y}(s)\bigr)\bigr]=
\frac{d^{2}}{du^{2}}\phi\bigl(\gamma_{\eta_{1,u},y}(s)\bigr) +\frac
{d^{2}}{du^{2}}
\phi\bigl(\gamma_{\eta_{2,u},y}(s)\bigr),
\]
where $\eta_{1,u}$ and $\eta_{2,u}$ are geodesics starting at $x$ and
having derivatives given by $\dot{\eta}_{1,0}=\dot{\gamma
}_{x,y}(0)$ and $\dot{\eta}_{2,u}=\xi_{1}$. Then we continue with
%
\begin{eqnarray}
\label{tc7} \Delta_{x}\bigl[\phi\bigl(\gamma_{x,y}(s)
\bigr)\bigr]&=&\bigl\langle\operatorname{Hess}\phi\bigl(\gamma(s)\bigr)\dot
{\gamma}(s),
\dot{\gamma}(s) \bigr\rangle+\bigl\langle\operatorname{Hess}\phi\bigl(\gamma(s)
\bigr)J_{1}(s),J_{1}(s) \bigr\rangle\nonumber
\\
&&{} + \biggl\langle\nabla\phi\bigl(\gamma(s)\bigr),\frac
{D^{2}}{\partial
u^{2}}
\gamma_{\eta_{1,u},y}(s)\Big|_{u=0} \biggr\rangle
\\
&&{} + \biggl\langle\nabla
\phi\bigl(\gamma(s)\bigr),\frac{D^{2}}{\partial u^{2}}\gamma_{\eta_{2,u},y}(s)
\Big|_{u=0} \biggr\rangle,\nonumber
\end{eqnarray}
where $J_{1}$ is the Jacobi field along $\gamma$ given by
$J_{1}(s)=\frac{D}{du}\gamma_{\eta_{2,u},y}(s)\mid_{u=0}$, which can
also be characterized as the Jacobi field with the boundary conditions
$J_{1}(0)=\xi_{1}$ and $J_{1}(l)=0$ which is solved as $J_{1}(s)=w_{1}(s)E(s)$.

Now notice that the third term vanishes because $\gamma_{\eta
_{1,u}}(s)=\gamma(s+u)$ and $\gamma$ is a geodesic. Next,\vspace*{1pt} we look at
$\mathcal{K}(s)=\frac{D^{2}}{\partial u^{2}}\gamma_{\eta
_{2,u},y}(s)\mid_{u=0}$. Using (\ref{tc103}), we need to focus on
finding $\mathcal{H}$ now. Exploiting (\ref{tc19}), the equation for
$\mathcal{H}$ becomes
\[
\cases{ \ddot{\mathcal{H}}+r\mathcal{H}-r\langle\mathcal{H},\dot{\gamma
} \rangle
\dot{\gamma}+2r \bigl(\langle\dot{\gamma}, \dot{J}_{1}\rangle
J_{1}-2\langle J_{1},\dot{J}_{1}\rangle\dot{
\gamma} \bigr)=0,
\vspace*{3pt}\cr
\mathcal{H}(0)=0,
\vspace*{3pt}\cr
\mathcal{H}(l)=0.}
\]
Notice here that the boundary conditions follow from the fact that
$\eta_{2,u}$ is a geodesic and that $\gamma_{\eta_{2,u},y}(l(u))=y$,
where $l(u)$ is the length of the geodesic joining $\eta_{2,u}$ and~$y$.

Now, the Jacobi field $J_{1}$ is given by
\[
J_{1}(s)=w_{1}(s)E(s)
\]
and this in turn gives the equation of $\mathcal{H}$ as
\[
\cases{ \ddot{\mathcal{H}}+r\mathcal{H}-r\langle\mathcal{H},\dot{\gamma
} \rangle
\dot{\gamma}=4rw_{1}\dot{w}_{1}\dot{\gamma},
\vspace*{3pt}\cr
\mathcal{H}(0)=0,
\vspace*{3pt}\cr
\mathcal{H}(l)=0.}
\]
We solve this as
%
\begin{equation}
\label{tc20} \mathcal{H}=w_{1,0}\dot{\gamma}\qquad\mbox{with
}w_{1,0}(s)=2r\int_{0}^{s}w_{1}^{2}(
\sigma)\,d\sigma-\frac{2sr}{l}\int_{0}^{l}w_{1}^{2}(
\sigma)\,d\sigma.
\end{equation}

The conclusion is that
%
\begin{eqnarray}
\label{tce1} \Delta_{x}\bigl[\phi\bigl(\gamma_{x,y}(s)
\bigr)\bigr]&=&\bigl\langle\operatorname{Hess}\phi\bigl(\gamma(s)\bigr) \dot
{\gamma}(s),
\dot{\gamma}(s) \bigr\rangle\nonumber
\\
&&{} +w_{1}^{2}(s)\bigl\langle
\operatorname{Hess}\phi\bigl(\gamma(s)\bigr) E(s), E(s) \bigr\rangle
\\
&&{} + \biggl(w_{1,0}(s)+\frac{s}{l}\dot{w}_{1}(0) \biggr)
\bigl\langle\nabla\phi\bigl(\gamma(s)\bigr), \dot{\gamma}(s)\bigr
\rangle.\nonumber
\end{eqnarray}

\item[(2)] In the same vein, with very few changes, we can treat the next
term, which is the Laplacian $\Delta_{y}$ applied to $\phi(\gamma
(s))$. To this end, take $\eta_{1,u}$ a geodesic starting at $y$ with
initial speed given by $\dot{\gamma}(l)$, and $\eta_{2,u}$ a
geodesic starting at $y$ with initial speed $\xi_{2}=E(l)$ and write
\[
\Delta_{y}\bigl[\phi\bigl(\gamma_{x,y}(s)\bigr)\bigr]=
\frac{d^{2}}{du^{2}}\phi\bigl(\gamma_{x,\eta_{1,u}}(s)\bigr) +\frac
{d^{2}}{du^{2}}
\phi\bigl(\gamma_{x,\eta_{2,u}}(s)\bigr).
\]
Notice that $\gamma_{x,\eta_{1,u}}(s)=\gamma(s)$ for small $u$, and
thus the first derivative is $0$. Thus, we arrive at
%
\begin{eqnarray}
\label{tc7b} \Delta_{y}\bigl[\phi\bigl(\gamma(s)\bigr)\bigr]&=&
\bigl\langle\operatorname{Hess}\phi\bigl(\gamma(s)\bigr)J_{2}(s),J_{2}(s)
\bigr\rangle
\nonumber\\[-8pt]\\[-8pt]\nonumber
&&{} + \biggl\langle\nabla\phi\bigl(\gamma(s)\bigr),\frac
{D^{2}}{\partial u^{2}}
\gamma_{x,\eta
_{2,u}}(s)\Big|_{u=0} \biggr\rangle,
\end{eqnarray}
where $J_{2}$ is the Jacobi field which is $0$ at $0$ and $\xi_{2}$ at
$l$ which is exactly solved by $J_{2}(s)=w_{2}(s)E(s)$. The second term
in the equation above can be dealt with in a similar way to that
outlined above for $\Delta_{x}$. We skip the details and give the main
result. From (\ref{tc104}),
\[
\mathcal{K}(s)=\frac{D^{2}}{\partial u^{2}}\gamma_{x,\eta
_{2,u}}(s)\Big|_{u=0}=
\mathcal{H}-\frac{s}{l}\dot{w}_{2}(l)\dot{\gamma}(s).
\]
From (\ref{tc19}), the equation satisfied by $\mathcal{H}$ [with
$w_{2}$ given by (\ref{tc21})] is given by
\[
\cases{ \ddot{\mathcal{H}}+r\mathcal{H}-r\langle\mathcal{H},\dot{\gamma
} \rangle
\dot{\gamma}=4rw_{2}\dot{w}_{2}\dot{\gamma},
\vspace*{3pt}\cr
\mathcal{H}(0)=0,
\vspace*{3pt}\cr
\mathcal{H}(l)=0,}
\]
which is solved for
%
\begin{equation}
\label{tc20b} \mathcal{H}=w_{0,1}\dot{\gamma}\qquad\mbox{with
}w_{0,1}(s)=2r\int_{0}^{s}w_{2}^{2}(
\sigma)\,d\sigma-\frac{2sr}{l}\int_{0}^{l}w_{2}^{2}(
\sigma)\,d\sigma.
\end{equation}
Then we have
%
\begin{eqnarray}
\label{tce2} \Delta_{y}\bigl[\phi\bigl(\gamma_{x,y}(s)
\bigr)\bigr]&=&w_{2}^{2}(s)\bigl\langle\operatorname{Hess}\phi
\bigl(\gamma(s)\bigr) E(s), E(s) \bigr\rangle
\nonumber\\[-8pt]\\[-8pt]\nonumber
&&{} + \biggl(w_{0,1}(s)- \frac {s}{l}\dot{w}_{2}(l) \biggr)\bigl\langle\nabla\phi\bigl(
\gamma(s)\bigr), \dot{\gamma}(s)\bigr\rangle.
\end{eqnarray}

\item[(3)] For the next term, matters are fairly simple. Namely, because we
are differentiating with respect to the geodesic parameter $s$,
%
\begin{equation}
\label{tce3} \partial_{s}^{2}\bigl[\phi\bigl(\gamma(s)
\bigr)\bigr]=\bigl\langle\operatorname{Hess}\phi\bigl(\gamma(s)\bigr)\dot
{\gamma}(s),
\dot{\gamma}(s)\bigr\rangle.
\end{equation}

\item[(4)] Next in line is
%
\begin{equation}
\label{tce4} \dot{\gamma}_{x,y}(0)\dot{\gamma}_{y,x}(0)
\bigl[\phi\bigl(\gamma(s)\bigr)\bigr]=0
\end{equation}
because
\[
\dot{\gamma}_{y,x}(0)\bigl[\phi\bigl(\gamma_{x,y}(s)\bigr)
\bigr]=0,
\]
which follows from the fact that perturbing $y$ along a curve $\eta
_{2,u}$ in the geodesic direction of $\gamma_{x,y}$ yields that
$\gamma_{x,\eta_{2,u}}(s)=\gamma_{x,y}(s)$, and thus is independent
of $u$.

\item[(5)] Now we deal with
\[
\xi_{1}\xi_{2}\bigl[\phi\bigl(\gamma_{x,y}(s)
\bigr)\bigr].
\]
To this end, consider the geodesics $\eta_{1,u}$ and $\eta_{2,v}$
which start at $x$ (resp., $y$) and have the initial tangent vectors
$\xi_{1}$ (resp., $\xi_{2}$). What we need to compute is
\begin{eqnarray*}
\frac{D}{du}\frac{D}{dv}\bigl[\phi\bigl(\gamma_{x,y}(s)
\bigr)\bigr]\Big|_{u=v=0}&=&\bigl\langle\operatorname{Hess}\phi\bigl(\gamma
_{x,y}(s)\bigr)J_{1}(s),J_{2}(s) \bigr\rangle
\\
&&{} +
\biggl\langle\nabla\phi\bigl(\gamma_{x,y}(s)\bigr),\frac{D}{du}
\frac{D}{dv}\gamma_{\eta_{1,u},\eta
_{2,v}}(s)\Big|_{u=v=0} \biggr\rangle
\end{eqnarray*}
with $J_{1}=w_{1}E$ and $J_{2}=w_{2}E$. If we let
\[
\mathcal{K}(s)=\frac{D}{du}\frac{D}{dv}\gamma_{\eta_{1,u},\eta
_{2,v}}(s)
\Big|_{u=v=0},
\]
from (\ref{tc102}), we have $\mathcal{K}=\mathcal{H}-\frac
{s}{2l}(\dot{w}_{1}(l)-\dot{w}_{2}(0))\dot{\gamma}(s)$. Now, from
(\ref{tc18}), we obtain
\[
\cases{ \ddot{\mathcal{H}}+r\mathcal{H}-r\langle\mathcal{H},\dot{\gamma
} \rangle
\dot{\gamma}=2r(w_{1}\dot{w}_{2}+w_{2}
\dot{w_{1}})\dot{\gamma},
\vspace*{3pt}\cr
\mathcal{H}(0)=0,
\vspace*{3pt}\cr
\mathcal{H}(l)=0,}
\]
which we solve as
%
\begin{eqnarray}\label{tc23}
\mathcal{H} =w_{1,1}\dot{\gamma}
\nonumber\\[-8pt]\\[-8pt]
\eqntext{\displaystyle\mbox{with }w_{1,1}(s)=2r\int_{0}^{s}w_{1}(
\sigma)w_{2}(\sigma)\sigma-\frac{2sr}{l}\int_{0}^{l}w_{1}(
\sigma)w_{2}(\sigma)\,d\sigma.}
\end{eqnarray}

We conclude that
%
\begin{eqnarray}\label{tce5}
\qquad \xi_{1}\xi_{2}\bigl[\phi\bigl(
\gamma_{x,y}(s)\bigr)\bigr]&=&w_{1}(s)w_{2}(s)\bigl
\langle\operatorname{Hess}\phi\bigl(\gamma(s)\bigr)E(s),E(s) \bigr\rangle
\nonumber\\[-8pt]\\[-8pt]\nonumber
&&{} + \biggl(w_{1,1}(s)-\frac{s}{2l}\bigl(\dot{w}_{1}(l)-
\dot{w}_{2}(0)\bigr) \biggr)\bigl\langle\nabla\phi\bigl(\gamma(s)
\bigr),\dot{\gamma}(s) \bigr\rangle.
\end{eqnarray}

\item[(6)] Next is
\begin{eqnarray*}
\dot{\gamma}_{x,y}(0)\partial_{s}\bigl[\phi\bigl(
\gamma_{x,y}(s)\bigr)\bigr]&=&\dot{\gamma}(0)\bigl\langle\nabla\phi
\bigl(\gamma(s)\bigr), \dot{\gamma}(s) \bigr\rangle= \bigl\langle
\operatorname{Hess}
\phi\bigl(\gamma(s)\bigr)\dot{\gamma}(s),\dot{\gamma}(s) \bigr\rangle.
\end{eqnarray*}

\item[(8)] Now,
%
\begin{equation}
\label{tce7} \dot{\gamma}_{y,x}(0)\partial_{s}\bigl[\phi
\bigl(\gamma_{x,y}(s)\bigr)\bigr]=0,
\end{equation}
as can be easily seen from the fact that perturbing $y$ in the geodesic
direction (say, along $\eta_{v}$) reveals that $\gamma_{x,\eta
_{v}}(s)=\gamma_{x,y}(s)$, and thus the derivative with respect to $v$
vanishes.

\item[(8)] The last term is easy to deal with and gives
%
\begin{equation}
\label{tce8} \partial_{s}\bigl[\phi\bigl(\gamma_{x,y}(s)
\bigr)\bigr]=\bigl\langle\nabla\phi\bigl(\gamma(s)\bigr), \dot{\gamma
}(s) \bigr
\rangle.
\end{equation}
\end{longlist}

Putting together all the results from (\ref{tce1})--(\ref{tce8}) and
using that $w_{2}(s)=w_{1}(l-s)$, we arrive at
%
\begin{eqnarray}
\mathcal{L}\bigl[\phi\bigl(\gamma(s)\bigr)\bigr]&=&\frac{\alpha
^{2}}{2}\bigl\langle
\operatorname{Hess}\phi\bigl(\gamma(s)\bigr)\dot{\gamma}(s),\dot{\gamma}(s)
\bigr
\rangle\nonumber
\\
&&{} +\frac{(aw_{1}(s)+bw_{2}(s))^{2}}{2}\bigl\langle\operatorname{Hess}\phi\bigl(\gamma(s)
\bigr)E(s),E(s) \bigr\rangle
\nonumber\\[-8pt]\\[-8pt]\nonumber
&&{}+ \biggl(\beta+\frac{s}{l} \biggl(\frac{a^{2}+b^{2}}{2}\dot
{w}_{1}(0)-ab \dot{w}_{1}(l) \biggr)
\\
&&\hspace*{16pt}\hspace*{17pt}{} +\frac
{a^{2}w_{1,0}+b^{2}w_{0,1}+2abw_{1,1}}{2}
\biggr)\bigl\langle\nabla\phi\bigl(\gamma(s)\bigr),\dot{\gamma}(s)
\bigr\rangle.\nonumber
\end{eqnarray}
A little simplification follows from
\begin{eqnarray*}
&& a^{2}w_{1,0}+b^{2}w_{0,1}+2abw_{1,1}
\\
&&\qquad = 2r
\biggl( \int_{0}^{s}\bigl(a(s)w_{1}(
\tau)+b(s)w_{2}(\tau)\bigr)^{2}\,d\tau
\\
&&\hspace*{48pt}{} -\frac{s}{l}\int
_{0}^{l}\bigl(a(s)w_{1}(
\tau)+b(s)w_{2}(\tau)\bigr)^{2}\,d\tau\biggr)
\end{eqnarray*}
which then gives the result for the choice of $\alpha$ as in (\ref{tc24}).
\end{pf}

We close this section with the following result summarizing all of the
important findings of this section which is used in the next section.

%
\begin{Cor}\label{tcc}
Assume that the entries of (\ref{tc1}) satisfy
%
\begin{eqnarray}\label{tc40}
\cases{ a\mbox{ is symmetric in }\rho_{1}\mbox{ and }
\rho_{2},
\vspace*{3pt}\cr
b=a,
\vspace*{3pt}\cr
\tilde{\alpha}=-\alpha,
\vspace*{3pt}\cr
\alpha(\tau,x,y,
\rho_{1},\rho_{2})=a(\tau,x,y,\rho_{1},\rho
_{2})w(\rho_{1}),
\vspace*{3pt}\cr
\beta(\tau,x,y,\rho_{1},
\rho_{2})=\frac{1}{2}a^{2}(\tau,x,y,\rho
_{1},\rho_{2}) \bigl(w(\rho_{1})\dot{w}(
\rho_{1})-\dot{w}(0)\bigr),
\vspace*{3pt}\cr
\tilde{\beta}(\tau,x,y,
\rho_{1},\rho_{2})=\frac{1}{2}a^{2}(\tau,x,y,\rho_{1},\rho_{2}) \bigl
(w(\rho_{2})
\dot{w}(\rho_{2})-\dot{w}(0)\bigr),
\vspace*{3pt}\cr
\rho_{1,0}=
\rho_{2,0}=\rho_{0}/2}
\nonumber\\[-8pt]\\[-8pt]
\eqntext{\displaystyle \mbox{with } \cases{
\ddot{w}+rw=0,
\vspace*{3pt}\cr
w(0)=1,
\vspace*{3pt}\cr
w\bigl(d(x,y)\bigr)=1.}}
\end{eqnarray}
Then:
\begin{longlist}[(2)]
\item[(1)]$\rho_{1,\tau}+\rho_{2,\tau}=\rho_{\tau}$ almost surely.
\item[(2)] The diffusions $(x_{\tau},y_{\tau},\rho_{1,\tau},\rho
_{2,\tau})$ and $(x_{\tau},y_{\tau},\rho_{2,\tau},\rho_{1,\tau
})$ have the same law. In particular, $(x_{\tau},y_{\tau},\rho
_{1,\tau})$ and $(x_{\tau},y_{\tau},\rho_{2,\tau})$ have the same law.
\item[(3)] If $z_{\tau}=\gamma_{x_{\tau},y_{\tau}}(\rho_{1,\tau})$,
then for any smooth function $\phi$ on $M$,
%
\begin{equation}
\label{tc22b} \phi(z_{\tau})-\int_{0}^{\tau}
\biggl(\frac{\alpha
^{2}(u)}{2}[\Delta\phi](z_{u})+\theta(u)\bigl\langle
\nabla\phi(z_{u}),\dot{\gamma}_{x_{u},y_{u}}(\rho_{1,u})
\bigr\rangle\biggr) \,du
\end{equation}
is a martingale with respect to the filtration generated by $W_{1}$,
$W_{2}$ and $W_{3}$, where
\begin{eqnarray*}
&& \theta(\tau,x,y,\rho_{1},\rho_{2})
\\
&&\qquad = \beta(\tau,x,y,
\rho_{1},\rho_{2})
\\
&&\qquad\quad{}+a^{2}(\tau,x,y,\rho_{1},\rho_{2})
\\
&&\qquad\qquad{}\times \biggl(
\frac{\rho
_{1}}{d(x,y)}\dot{w}(0)+r \biggl( \int_{0}^{\rho_{1}}w^{2}(
\sigma)\,d\sigma-\frac{\rho_{1}}{d(x,y)}\int_{0}^{d(x,y)}w^{2}(
\sigma)\,d\sigma\biggr) \biggr).
\end{eqnarray*}
\end{longlist}
\end{Cor}

A word is in place here. The statement of Theorem~\ref{tct2} requires
the symmetry of $\alpha$ with respect to $\rho_{1}$ and $\rho_{2}$.
This is not satisfied by the choice in (\ref{tc40}) for arbitrary
$\rho_{1}$ and $\rho_{2}$. However, because of the choice of $\beta$
and $\tilde{\beta}$ and Theorem~\ref{tct1}, we know that (almost
surely) $\rho_{1,\tau}+\rho_{2,\tau}=\rho_{\tau}$. So it suffices
to ensure the symmetry of $\alpha$ and $\tilde{\alpha}$ with respect
to $\rho_{1}$ and $\rho_{2}$ only in the case that $\rho_{1}+\rho
_{2}=\rho=d(x,y)$, which follows from the fact that $w(s)=w(d(x,y)-s)$
for $s\in[0,d(x,y)]$.

For a given $l$, the solution $w$ to (\ref{tc40}) is
%
\begin{equation}
\label{tcws} w(s)=\cases{ 1, &\quad $r=0$,
\vspace*{3pt}\cr
\displaystyle\frac{\cosh((l-2s)/2 )}{\cosh(l/2)}, &\quad$r=-1$,
\vspace*{3pt}\cr
\displaystyle\frac{\cos((l-2s)/2 )}{\cos(l/2)}, &\quad$r=1$.}
\end{equation}
In particular, if $l$ is small snough, $w(s)$ and all its derivatives
stay bounded. In addition to this $\dot{w}(0)=O(l)$, a property which
will play an important role in the coming section. Thus, if $a$ is a
bounded function, then
%
\begin{equation}
\label{tc110} \sup_{\tau\in[0,d(x,y)]}\theta(\tau, x, y,
\rho_{1},\rho_{2})=O(\rho_{1}).
\end{equation}

\section{Estimates on the Hessian decay for \texorpdfstring{$\chi(M)\le0$}{$chi(M)<=0$}}\label{s9}

For Euler characteristic less than or equal to $0$, we know that $\bar
{p}_{t}$ and $\nabla\bar{p}_{t}$ decay exponentially fast. Our goal
is now to extend this to the Hessian of $\bar{p}_{t}$, resulting in
the convergence of the metric to the constant curvature metric in
$C^2$. In particular, the curvature converges to a constant.

To estimate the Hessian decay, we proceed in a similar way to the
estimation of the gradient, only that now we need to use the coupling
procedure for three points rather than two.

Let us denote, for $t>0$,
\[
H(t)=\sup_{x\in M} \bigl\llvert\operatorname{Hess}
\bar{p}_{t}(x)\bigr\rrvert.
\]
What we want to show is that $H(t)$ decays to $0$ exponentially fast.

%
\begin{teo}\label{het}
For the case $\chi(M)\le0$, $H(t)$ converges to $0$ exponentially
fast as $t\to\infty$.
\end{teo}

\begin{pf}
To begin with, notice that
%
\begin{equation}
\label{he0} \bigl\langle\operatorname{Hess} \bar{p}_{t}(z)\xi,\xi\bigr
\rangle=\lim_{\rho
_{0}\to0}\frac{\bar{p}_{t}(\gamma(-\rho_{0}))-2\bar{p}_{t}(z)+\bar
{p}_{t}(\gamma(\rho_{0}))}{\rho_{0}^{2}},
\end{equation}
where $\gamma$ is the unique geodesic passing through $z$ and having
the initial velocity given by $\xi$. Thus, similarly to the case of
the gradient estimate, we will use the three particle coupling to get a
handle on the right-hand side of the above quantity, for sufficiently
small $\rho_{0}$.

For convenience, fix a time $t>0$ and let $s\in[0,1\wedge t]$. Pick
two points $x,y\in M$, with $d(x,y)=\rho_{0}$ small enough, and let
$z$ be the middle point on the geodesic between $x$ and $y$ such that
$d(x,z)=d(z,y)=\rho_{0}/2$. Consider the triple coupling described by
(\ref{tc1}) with the choices from Corollary~\ref{tcc}. All the data
there is completely described by the choice of the time change $a$ of
the processes $x_{\tau}$ and $y_{\tau}$. In this section, we choose
%
\begin{equation}
\label{her1} a(\tau,x,y,\rho_{1},\rho_{2})=
\sqrt{2}e^{-\bar{p}_{t-\tau
}(\lambda_{x,y})},
\end{equation}
where $\lambda_{x,y}$ is the middle point on the geodesic between $x$
and $y$. This choice does not depend on $\rho_{1}$ or $\rho_{2}$, and
consequently it is symmetric in $\rho_{1}$ and $\rho_{2}$, as
required by Corollary~\ref{tcc}. Other choices are possible for the
argument here, but we stick with this one because it is symmetric with
respect to $x$ and $y$ and makes some of the estimates look more natural.

Now, we consider $\bar{p}_{t-\OldTau}(z_{\OldTau})$, where $z_{\tau
}$ is defined in the previous section. Again invoking Corollary~\ref
{tcc}, we learn that
\begin{eqnarray*}
&& d\bar{p}_{t-\OldSigma}(z_{\OldSigma})
\\
&&\qquad =M_{1,\tau}+ \biggl(-\partial
_{t}\bar{p}_{t-\OldSigma}(z_{\OldSigma})+\frac{\alpha^{2}(\tau
)}{2}
\Delta\bar{p}_{t-\OldSigma}(z_{\OldSigma})+\theta(\tau)\bigl\langle
\nabla
\bar{p}_{t-\tau}(z_\tau),\dot{\gamma}_{\tau} \bigr
\rangle\biggr)\,d\OldSigma,
\end{eqnarray*}
where $M_{1,\tau}$ is a martingale. From the Ricci flow equation,
$\partial_{t}\bar{p}_{t-\OldSigma}(z_{\OldSigma})=  e^{-2\bar
{p}_{t-\OldSigma}(z_{\OldSigma})}\Delta\bar{p}_{t-\OldSigma
}(z_{\OldSigma})+r(1-e^{-2\bar{p}_{t-\OldSigma}(z_{\OldSigma})})$
so we continue with
%
\begin{eqnarray}\label{he2-1}
&& d\bar{p}_{t-\OldSigma}(z_{\OldSigma})\nonumber
\\
&&\qquad =M_{1,\tau}+
\biggl( \biggl(\frac{\alpha^{2}(\OldSigma)}{2}-e^{-2\bar{p}_{t-\OldSigma
}(z_{\OldSigma})} \biggr)\Delta
\bar{p}_{t-\OldSigma}(z_{\OldSigma
})
\\
&&\hspace*{74pt}{} +\theta(\tau)\bigl\langle\nabla
\bar{p}_{t-\tau}(z_\tau),\dot{\gamma}_{\tau} \bigr
\rangle-r\bigl(1-e^{-2\bar{p}_{t-\OldSigma
}(z_{\OldSigma})}\bigr) \biggr)\,d\OldSigma.\nonumber
\end{eqnarray}

For the semi-martingale $\bar{p}_{t-\OldSigma}(x_{\OldSigma})$ we
have from (\ref{tc1}) and the Ricci flow equation that
%
\begin{eqnarray}\label{he2-2}
d\bar{p}_{t-\OldSigma}(x_{\OldSigma})
&=& M_{2,\tau}+
\bigl(\bigl(e^{-2\bar{p}_{t-\OldSigma}(\lambda_{\tau})}-e^{-2\bar
{p}_{t-\OldSigma
}(x_{\OldSigma})}\bigr)\Delta\bar{p}_{t-\OldSigma}(x_{\OldSigma})
\nonumber\\[-8pt]\\[-8pt]\nonumber
&&\hspace*{113pt}{} - r\bigl(1-e^{-2\bar{p}_{t-\OldSigma}(x_{\OldSigma})}\bigr) \bigr
)\,d\OldSigma,
\end{eqnarray}
where $\lambda_{\tau}$ is the middle point of the geodesic joining
$x_{\tau}$ and $y_{\tau}$. Similarly, for $\bar{p}_{t-\OldSigma
}(y_{\OldSigma})$,
%
\begin{eqnarray}\label{he2-3}
d\bar{p}_{t-\OldSigma}(y_{\OldSigma})&=&M_{3,\tau}+
\bigl(\bigl(e^{-2\bar{p}_{t-\OldSigma}(\lambda_{\tau})}-e^{-2\bar
{p}_{t-\OldSigma
}(y_{\OldSigma})}\bigr)\Delta\bar{p}_{t-\OldSigma}(y_{\OldSigma})
\nonumber\\[-8pt]\\[-8pt]\nonumber
&&\hspace*{113pt}{} -
r\bigl(1-e^{-2\bar{p}_{t-\OldSigma}(y_{\OldSigma})}\bigr) \bigr
)\,d\OldSigma.
\end{eqnarray}

Now, putting these together,
%
\begin{eqnarray}\label{he3}
&& \bar{p}_{t-\OldSigma}(x_{\OldSigma})-2
\bar{p}_{t-\OldSigma
}(z_{\OldSigma})+\bar{p}_{t-\OldSigma}(y_{\OldSigma})\nonumber
\\
&&\qquad = \bar{p}_{t}(x)-2\bar{p}_{t}(z)+\bar{p}_{t}(y)+M_{\OldSigma}\nonumber
\\
&&\quad\qquad{} -2\int_{0}^{\tau} \biggl( \biggl(
\frac{\alpha^{2}(u)}{2}- e^{-2\bar{p}_{t-u}(z_{u})} \biggr)\Delta\bar
{p}_{t-u}(z_{u})
\biggr)\,du\nonumber
\\
&&\quad\qquad{}-2 \int_{0}^{\tau}\theta(u)\bigl\langle\nabla
\bar{p}_{t-u}(u),\dot{\gamma}_{u} \bigr\rangle \,du
\\
&&\quad\qquad{}+\int_{0}^{\OldSigma} \bigl(\bigl(e^{-2\bar{p}_{t-u}(\lambda
_{u})}-e^{-2\bar{p}_{t-u}(x_{u})}
\bigr)\Delta\bar{p}_{t-u}(x_{u})\nonumber
\\
&&\hspace*{30pt}\quad\qquad{} + \bigl(e^{-2\bar{p}_{t-u}(\lambda
_{u})}-e^{-2\bar{p}_{t-u}(y_{u})}
\bigr)\Delta\bar{p}_{t-u}(y_{u}) \bigr)\,du\nonumber
\\
&&\quad\qquad{}+r\int_{0}^{\OldSigma} \bigl(e^{-2\bar{p}_{t-u}(x_{u})}-2e^{-2\bar
{p}_{t-u}(z_{u})}+e^{-2\bar{p}_{t-u}(y_{u})}
\bigr)\,du,\nonumber
\end{eqnarray}
where $M_{\tau}$ is a martingale.

From the definition of $\alpha$ in Corollary~\ref{tcc} and the fact
that we stop the processes before the distance between $x$ and $y$ hits
some small number $r_{0}$, it is not hard to prove [e.g.,
directly from (\ref{tcws})] that there is a constant $C>0$ such that
\[
\bigl\llvert\alpha(u)-a(u)\bigr\rrvert\le C\rho_{u},
\]
which in turn, using the gradient decay estimates and the fact that
$d(z_{u},\lambda_{u})\le d(x_{u},y_{u})/2=\rho_{u}/2$, leads to
(notice that $t-u\ge t-1$ because $u\in[0,1\wedge t]$)
\begin{eqnarray*}
\biggl\llvert\frac{\alpha^{2}(u)}{2}- e^{-2\bar{p}_{t-u}(z_{u})} \biggr
\rrvert&\le& C \rho_{u}+\bigl\llvert e^{-2\bar{p}_{t-u}(z_{u})}-e^{-2\bar
{p}_{t-u}(\lambda_{u})}\bigr\rrvert
\\
&\le& C\rho_{u}+Ce^{-C t}\rho_{u}\le C
\rho_{u}.
\end{eqnarray*}
Observe here that we do not need the full power of the exponential
decay of the gradient; just the boundedness suffices for this
particular estimate, but used in conjunction with (\ref{tc110}), for
any $u\in[0,1\wedge t]$,
\[
\bigl\llvert\theta(u)\bigl\langle\nabla\bar{p}_{t-u}(z_{u}),
\dot{\gamma}_{u} \bigr\rangle\bigr\rrvert\le c\rho_{u}e^{-C t}.
\]
Finally, from the exponential decay of the gradient and elementary
arguments, as long as $u\in[0,1\wedge t]$,
\[
\bigl\llvert e^{-2\bar{p}_{t-u}(x_{u})}-e^{-2\bar{p}_{t-u}(z_{u})} \bigr
\rrvert+\bigl\llvert
e^{-2\bar{p}_{t-u}(y_{u})}-e^{-2\bar{p}_{t-u}(z_{u})} \bigr\rrvert\le
c\rho_{u}e^{-C t}
\]
and also
\[
\bigl\llvert e^{-2\bar{p}_{t-u}(x_{u})}-e^{-2\bar{p}_{t-u}(\lambda
_{u})}\bigr\rrvert+\bigl\llvert
e^{-2\bar{p}_{t-u}(y_{u})}-e^{-2\bar{p}_{t-u}(\lambda_{u})}\bigr\rrvert
\le c\rho_{u}e^{-C t}.
\]

Now, let $\sigma$ be the first time $u$ when $\rho_{1,u}$ or $\rho
_{2,u}$ becomes $0$, and let $\zeta$ be the first time $u$ when either
$\rho_{1,u}$ or $\rho_{2,u}$ hits $r_{0}$, a small number (less than
half of the injectivity radius). Replacing $\OldSigma$ by $\OldSigma
\wedge\OldTau\wedge\zeta$ in (\ref{he3}) and then taking the
expectation at $\OldSigma=0$ and $\OldSigma=s$, combined with the
above estimates, lead to
%
\begin{eqnarray}
\label{he4}
\qquad&& \bigl\llvert\bar{p}_{t}(x)-2\bar{p}_{t}(z)+
\bar{p}_{t}(y)\bigr\rrvert\nonumber
\\
&&\qquad \le \bigl\llvert\E\bigl[\bar{p}_{t-s\wedge
\OldTau\wedge\zeta}(x_{s\wedge\OldTau\wedge\zeta
})-2\bar{p}_{t-s\wedge\OldTau\wedge\zeta}(z_{s\wedge\OldTau
\wedge\zeta})+
\bar{p}_{t-s\wedge\OldTau\wedge\zeta}(y_{s\wedge
\OldTau\wedge\zeta})\bigr]\bigr\rrvert
\\
&&\quad\qquad{}+ce^{-Ct}\E\biggl[\int_{0}^{s\wedge\sigma\wedge\zeta}\rho
_{u}\,du \biggr]+c\E\biggl[\int_{0}^{s\wedge\OldTau\wedge\zeta}
\rho_{u}H(t-u)\,du \biggr]\nonumber
\end{eqnarray}
for any $s\in[0,1\wedge t]$.

Next, the stopping time $\OldTau$ is $T_{1}\wedge T_{2}$, where
$T_{1}$ and $T_{2}$ are, respectively, the first time $\rho_{1}$ hits
$0$ and the first time $\rho_{2}$ hits $0$. Now we can write
%
\begin{eqnarray}
\label{he4100} &&\E\bigl[\bar{p}_{t-s\wedge\OldTau\wedge\zeta}(x_{s\wedge
\OldTau \wedge\zeta})-2
\bar{p}_{t-s\wedge\OldTau\wedge\zeta}(z_{s\wedge
\OldTau\wedge\zeta})+\bar{p}_{t-s\wedge\OldTau\wedge\zeta
}(y_{s\wedge\OldTau\wedge\zeta})
\bigr]\nonumber
\\
&&\qquad = \E\bigl[\bar{p}_{t-s\wedge\zeta}(x_{s\wedge\zeta})-2\bar
{p}_{t-s\wedge\zeta}(z_{s\wedge\zeta})+
\bar{p}_{t-s\wedge\zeta
}(y_{s\wedge\zeta}),\zeta<\OldTau\bigr]\nonumber
\\
&&\quad\qquad{} + \E\bigl[\bar{p}_{t-s\wedge\OldTau}(x_{s\wedge\OldTau})-2\bar
{p}_{t-s\wedge\OldTau}(z_{s\wedge\OldTau})+
\bar{p}_{t-s\wedge
\OldTau}(y_{s\wedge\OldTau}),\OldTau\le\zeta\bigr]\nonumber
\\
&&\qquad =\E\bigl[\bar{p}_{t-s\wedge\zeta}(x_{s\wedge\zeta})-2\bar{p}_{t-s\wedge
\zeta}(z_{s\wedge\zeta})+
\bar{p}_{t-s\wedge\zeta
}(y_{s\wedge\zeta}),\zeta<\OldTau\bigr]\nonumber
\\
&&\quad\qquad{}+\E\bigl[\bar{p}_{t-T_{1}}(y_{T_{1}})-\bar{p}_{t-T_{1}}(x_{T_{1}}),T_{1}<
T_{2}\le s\wedge\zeta\bigr]
\\
&&\quad\qquad{} +\E\bigl[\bar{p}_{t-T_{2}}(x_{T_{2}})-
\bar{p}_{t-T_{2}}(y_{T_{2}}),T_{2}< T_{1}\le s
\wedge\zeta\bigr]\nonumber
\\
&&\quad\qquad{} +\E\bigl[\bar{p}_{t-T_{1}}(y_{T_{1}})-\bar{p}_{t-T_{1}}(x_{T_{1}}),T_{1}
\le s\wedge\zeta< T_{2}\bigr]\nonumber
\\
&&\quad\qquad{} +\E\bigl[\bar{p}_{t-T_{2}}(x_{T_{2}})-
\bar{p}_{t-T_{2}}(y_{T_{2}}),T_{2}\le s\wedge\zeta<
T_{1}\bigr]\nonumber
\\
&&\quad\qquad{} +\E\bigl[\bar{p}_{t-s}(x_{s})-2\bar{p}_{t-s}(z_{s})+
\bar{p}_{t-s}(y_{s}),s\le\OldTau<\zeta\bigr].\nonumber
\end{eqnarray}
Here, we bear to fruit the work done in the previous section and argue
that due to the symmetry with respect to $\rho_{1}$ and $\rho_{2}$
from Corollary~\ref{tcc}, we have the crucial cancellations
%
\begin{eqnarray}\label{he4-2}
&& \E\bigl[\bar{p}_{t-T_{1}}(y_{T_{1}})-
\bar{p}_{t-T_{1}}(x_{T_{1}}),T_{1}< T_{2}\le s
\wedge\zeta\bigr]
\nonumber\\[-8pt]\\[-8pt]\nonumber
&&\qquad{} +\E\bigl[\bar{p}_{t-T_{2}}(x_{T_{2}})-\bar
{p}_{t-T_{2}}(y_{T_{2}}),T_{2}< T_{1}\le s
\wedge\zeta\bigr] =0
\end{eqnarray}
and also
%
\begin{eqnarray}\label{he4-3}
&& \E\bigl[\bar{p}_{t-T_{1}}(y_{T_{1}})-\bar
{p}_{t-T_{1}}(x_{T_{1}}),T_{1}\le s\wedge\zeta<
T_{2}\bigr]
\nonumber\\[-8pt]\\[-8pt]\nonumber
&&\qquad{} +\E\bigl[\bar{p}_{t-T_{2}}(x_{T_{2}})-
\bar{p}_{t-T_{2}}(y_{T_{2}}),T_{2}\le s\wedge\zeta<
T_{1}\bigr] =0.
\end{eqnarray}

Furthermore, from the exponential decay of $\bar{p}$ and $\nabla\bar
{p}$, for any $s\in[0,1\wedge t]$ we have
\begin{eqnarray*}
&&\bigl\llvert\E\bigl[\bar{p}_{t-s\wedge\OldTau\wedge\zeta}(x_{s\wedge
\OldTau
\wedge\zeta})-2
\bar{p}_{t-s\wedge\OldTau\wedge\zeta}(z_{s\wedge
\OldTau\wedge\zeta})+\bar{p}_{t-s\wedge\OldTau\wedge\zeta
}(y_{s\wedge\OldTau\wedge\zeta})
\bigr]\bigr\rrvert
\\
&&\qquad \le\E\bigl[\bigl\llvert\bar{p}_{t-s\wedge\zeta}(x_{s\wedge\zeta
})-2\bar{p}_{t-s\wedge\zeta}(z_{s\wedge\zeta})+\bar{p}_{t-s\wedge\zeta
}(y_{s\wedge\zeta})
\bigr\rrvert,\zeta<\OldTau\bigr]
\\
&&\quad\qquad{} + \E\bigl[\bigl\llvert\bar{p}_{t-s}(x_{s})-2
\bar{p}_{t- s}(z_{s})+\bar{p}_{t-s}(y_{s})
\bigr\rrvert,s\le\OldTau\le\zeta\bigr]
\\
&&\qquad \le ce^{-Ct}\Prob(\zeta\le s\wedge\OldTau)+ce^{-Ct}\E[\rho
_{s},s\le\OldTau\wedge\zeta],
\end{eqnarray*}
where we used the following inequalities:
\begin{eqnarray*}
&& \bigl\llvert\E\bigl[\bar{p}_{t-s\wedge\zeta} (x_{s\wedge\zeta})-2\bar
{p}_{t-s\wedge\zeta}(z_{s\wedge\zeta})+\bar{p}_{t-s\wedge\zeta
}(y_{s\wedge\zeta}),
\zeta<\OldTau\bigr]\bigr\rrvert
\\
&&\qquad \le\E\bigl[\bigl\llvert\bar{p}_{t-s}(x_{s})-2
\bar{p}_{t-s}(z_{s})+\bar{p}_{t-s}(y_{s})
\bigr\rrvert,s<\zeta<\OldTau\bigr]
\\
&&\quad\qquad{} +\E\bigl[\bigl\llvert\bar
{p}_{t-\zeta}(x_{\zeta
})-2\bar{p}_{t-\zeta}(z_{\zeta})+
\bar{p}_{t-\zeta}(y_{\zeta
})\bigr\rrvert,\zeta\le s\wedge\OldTau
\bigr]
\\
&&\qquad \le ce^{-Ct}\E[\rho_{s},s<\OldTau\wedge
\zeta]+ce^{-Ct}\Prob(\zeta\le s\wedge\OldTau).
\end{eqnarray*}
Putting these together into (\ref{he4}), plus a little
simplification, gives that for any $s\in[0,1\wedge t]$
\begin{eqnarray*}
&& \bigl\llvert\bar{p}_{t}(x)-2\bar{p}_{t}(z)+
\bar{p}_{t}(y)\bigr\rrvert
\\
&&\qquad \le ce^{-Ct}\Prob(\zeta\le s
\wedge\OldTau)+ce^{-C t}\E[\rho_{s},s\le\OldTau\wedge\zeta]
\\
&&\quad\qquad{} +ce^{-Ct}\int_{0}^{s}\E[
\rho_{u},u\le\OldTau\wedge\zeta]\,du+c\int_{0}^{s}H(t-u)
\E[\rho_{u},u\le\OldTau\wedge\zeta]\,du.
\end{eqnarray*}
A further simplification is due to the symmetry with respect to $\rho
_{1}$ and $\rho_{2}$ from Corollary~\ref{tcc}, which has the effect that
\[
\E[\rho_{u},u<\OldTau\wedge\zeta]=2\E[\rho_{1,u},u<\OldTau
\wedge\zeta],
\]
and thus for $s\in[0,1\wedge t]$,
%
\begin{eqnarray}\label{he5}
&& \bigl\llvert\bar{p}_{t}(x)-2\bar{p}_{t}(z)+
\bar{p}_{t}(y)\bigr\rrvert\nonumber
\\
&&\qquad \le Ce^{-Ct}\Prob(\zeta<s\wedge
\OldTau)+Ce^{-C t}\E[\rho_{1,s},s<\OldTau\wedge\zeta]
\nonumber\\[-8pt]\\[-8pt]\nonumber
&&\quad\qquad{}+e^{-Ct}\int_{0}^{s}\E[
\rho_{1,u},u<\OldTau\wedge\zeta]\,du
\\
&&\quad\qquad{} +C\int_{0}^{s}H(t-u)
\E[\rho_{1,u},u<\OldTau\wedge\zeta]\,du.\nonumber
\end{eqnarray}
The key step forward is the following result.

\begin{teo}\label{hep1}
Let $W^{1}$, $W^{2}$ and $W^{3}$ be three independent, one-dimensional
Brownian motions, and let $\tilde{\rho}_{1}$ and $\tilde{\rho}_{2}$
be two processes such that $\tilde{\rho}_{1,0}=\tilde{\rho
}_{2,0}=\tilde{\rho}_{0}>0$ and
%
\begin{equation}
\label{he6} \cases{ d\tilde{\rho}_{1,\OldSigma}=\bigl(1+O(\tilde{
\rho}_{1,\OldSigma
})\bigr) \bigl(A_{\tau}\,dW^{1}_{\OldSigma}+B_{\tau}\,dW^{2}_{\OldSigma
}
\bigr)
\vspace*{3pt}\cr
\hspace*{36pt}{} +\bigl(1+O(\tilde{\rho}_{1,\OldSigma})\bigr)\,dW^{3}_{\OldSigma
}+O(1)\,d
\OldSigma,
\vspace*{3pt}\cr
d\tilde{\rho}_{2,\OldSigma}=\bigl(1+O(\tilde{
\rho}_{2,\OldSigma
})\bigr) \bigl(A_{\tau}\,dW^{1}_{\OldSigma}+B_{\tau}\,dW^{2}_{\OldSigma
}
\bigr)
\vspace*{3pt}\cr
\hspace*{36pt}{}-\bigl(1+O(\tilde{\rho}_{2,\OldSigma})\bigr)\,dW^{3}_{\OldSigma
}+O(1)\,d
\OldSigma,}
\end{equation}
with $A^{2}_{\tau}+B^{2}_{\tau}=1$.

Let $\tilde{\OldTau}$ be the first hitting time of $0$ for the
process $\tilde{\rho}_{1}\tilde{\rho}_{2}$ and $\tilde{\zeta}$
the first time either $\tilde{\rho}_{1}$ or $\tilde{\rho}_{2}$ hits
some value $\tilde{r}_{0}$. Assume that (\ref{he6}) is valid for
$\OldSigma\in[0,\tilde{\OldTau}\wedge\tilde{\zeta}]$, and in
addition that for some constant $C>0$
%
\begin{equation}
\label{he7} \E[\tilde{\rho}_{2,s},s<\tilde{\OldTau}\wedge\tilde{
\zeta}]\le C\E[\tilde{\rho}_{1,s},s<\tilde{\OldTau}\wedge\tilde{\zeta}]
\qquad\mbox{for all }s\in[0,1\wedge t].
\end{equation}
Then there is a constant $C>0$ such that, for all $s\in[0,1\wedge t]$
and sufficiently small $\tilde{\rho}_{0}>0$,
%
\begin{equation}
\label{he20} \E[\tilde{\rho}_{1,s},s<\tilde{\OldTau}\wedge\tilde{
\zeta}]\le C\tilde{\rho}_{0}^{2}/\sqrt{s}
\end{equation}
and
%
\begin{equation}
\label{he21} \Prob(\tilde{\zeta}<s\wedge\tilde{\OldTau})\le C\tilde{
\rho}^{2}_{0}.
\end{equation}
\end{teo}

\begin{pf}
If we regard the process $(\tilde{\rho}_{1,\OldSigma},\tilde{\rho
}_{2,\OldSigma})$ as a process in the first quadrant, the equations in
(\ref{he6}) tell us that near the axes the process is near $\sqrt
{2}$ times a two-dimensional Brownian motion which certainly satisfies
both properties (\ref{he20}) and (\ref{he21}). Consequently, what
we will do is to compare $\E[\tilde{\rho}_{1,s},s<\tilde{\OldTau
}]$ with the analogous quantity in which $\tilde{\rho}_{1}$ and
$\tilde{\rho}_{2}$ run as independent Brownian motions.

In the simplest case in which $(\tilde{\rho}_{1},\tilde{\rho}_{2})$
is $\sqrt{2}$
times a planar Brownian motion started at $(\tilde{\rho}_{0},\tilde
{\rho}_{0})$ the
quantity $\E[f(\tilde{\rho}_{1,s},\tilde{\rho}_{2,s}),s<\tilde
{\OldTau}]$ is simply
$\phi(s,\tilde{\rho}_{0},\tilde{\rho}_{0})$, with $\phi$ being
the solution to the
following PDE on the upper-right quadrant $\Omega=\{ (x,y)\in\Re
^{2},x,y>0\}$:
%
\begin{equation}
\label{he8} \cases{ \partial_{t} \phi=\Delta\phi,
\vspace*{3pt}\cr
\phi
\bigl(t,(x,y)\bigr)=0, &\quad $(x,y)\in\partial\Omega$,
\vspace*{3pt}\cr
\phi\bigl(0,(x,y)
\bigr)=f(x,y), &\quad $(x,y)\in\Omega$.}
\end{equation}
This solution can be written in an explicit form in terms of the heat
kernel, which we discuss now. On the half line, the heat kernel for the
Laplacian with the Dirichlet boundary condition is given by
\[
h_{t}(x,y)=\frac{1}{\sqrt{4\pi t}} \bigl(e^{-\sfrac
{(x-y)^{2}}{4t}}-e^{-\sfrac{(x+y)^{2}}{4t}}
\bigr)
\]
for all $x,y,t>0$. On $\Omega$, the heat kernel with the Dirichlet
boundary condition is simply
\[
\mathbf{h}_{t}\bigl((x_{1},x_{2}),(y_{1},y_{2})
\bigr)=h_{t}(x_{1},y_{1})h_{t}(x_{2},y_{2}).
\]
Turning back to the PDE (\ref{he8}), the solution is given by
\[
\phi(t,x,y)=\int_{0}^{\infty}\int
_{0}^{\infty}\mathbf{h}_{t}
\bigl((x,y),(x_{1},y_{1})\bigr)f(x_{1},y_{1})\,dx_{1}
\,dy_{1}.
\]
For the case we are most interested in, namely $f(x,y)=x$, the solution
above can be computed as
\[
\phi(s,x,y)=x \Phi\biggl(\frac{y}{\sqrt{s}} \biggr)\qquad\mbox{with
}\Phi(x)=
\frac{1}{\sqrt{\pi}}\int_{0}^{y}e^{-u^{2}/4}\,du.
\]

Now we go back to the system (\ref{he6}) and take $\phi(s-\OldSigma
,\tilde{\rho}_{1,\OldSigma},\tilde{\rho}_{2,\OldSigma})$ as a
semi-martingale which, from It\^o's formula and $\partial_{t}\phi
=\Delta\phi$, becomes
\begin{eqnarray*}
&& d\phi(s-\OldSigma,\tilde{\rho}_{1,\OldSigma},\tilde{\rho}_{2,\OldSigma})
\\
&&\qquad =
\partial_{x}\phi \,d\tilde{\rho}_{1,\OldSigma
}+\partial_{y}
\phi \,d\tilde{\rho}_{2,\OldSigma}-\partial_{t}\phi \,d\OldSigma
\\
&&\quad\qquad{}+
\frac{1}{2}\partial_{xx}^{2}\phi \,d\langle\tilde{\rho
}_{1} \rangle_{\OldSigma}+\partial_{xy}^{2}
\phi \,d\langle\tilde{\rho}_{1},\tilde{\rho}_{2}
\rangle_{\OldSigma}+\frac
{1}{2}\partial_{yy}^{2}
\phi \,d\langle\tilde{\rho}_{2} \rangle_{\OldSigma}
\\
&&\qquad =M_{\OldSigma} + O(1) \biggl(\Phi\biggl(\frac{\tilde{\rho
}_{2,\OldSigma}}{\sqrt{s-\OldSigma}} \biggr)+
\frac{\tilde{\rho
}_{1,\OldSigma}}{\sqrt{s-\OldSigma}}\Phi' \biggl(\frac{\tilde{\rho
}_{2,\OldSigma}}{\sqrt{s-\OldSigma}} \biggr) \biggr)\,d
\OldSigma
\\
&&\quad\qquad{} + \frac{\tilde{\rho}_{1,\OldSigma}O(\tilde{\rho}_{2,\OldSigma
})\Phi''(\sfrac{\tilde{\rho}_{2,\OldSigma}}{\sqrt{s-\OldSigma
}})}{s-\OldSigma}\,d\OldSigma+\frac{O(\tilde{\rho}_{1,\OldSigma
}+\tilde{\rho}_{2,\OldSigma})}{\sqrt{s-\OldSigma}}\Phi' \biggl(
\frac{\tilde{\rho}_{2,\OldSigma}}{\sqrt{s-\OldSigma}} \biggr)\,d\OldSigma,
\end{eqnarray*}
where $M_{\OldSigma}$ is a martingale. Since $\Phi'$ and $y\Phi
''(y)$ are bounded, we deduce that the drift in the above is bounded in
absolute value by $\frac{C(\tilde{\rho}_{1,\OldSigma}+\tilde{\rho
}_{2,\OldSigma})}{\sqrt{s-\OldSigma}}$. Now replacing $\OldSigma$
by $\OldSigma\wedge\tilde{\OldTau}\wedge\tilde{\zeta}$ and
evaluating at $\OldSigma=0$ and $\OldSigma=s$, we are led to
\begin{eqnarray*}
\E[\tilde{\rho}_{1,s},s<\tilde{\OldTau}\wedge\tilde{\zeta}]&\le&\E
\bigl[\phi(s-s\wedge\tilde{\OldTau}\wedge\tilde{\zeta},\tilde{
\rho}_{1,s\wedge\tilde{\OldTau}\wedge\tilde{\zeta}},\tilde{\rho
}_{2,s\wedge\tilde{\OldTau}\wedge\tilde{\zeta}})\bigr]
\\
&\le& \phi(s,\tilde{
\rho}_{0},\tilde{\rho}_{0})+C\E\biggl[ \int
_{0}^{s\wedge\tilde{\OldTau}\wedge\tilde{\zeta}}\frac{\tilde
{\rho}_{1,\OldSigma}+\tilde{\rho}_{2,\OldSigma}}{\sqrt
{s-\OldSigma}}\,d\OldSigma\biggr]
\\
&\le& C\tilde{\rho}_{0} \Phi\biggl(\frac{\tilde{\rho}_{0}}{\sqrt
{s}} \biggr)+C\int
_{0}^{s}\frac{\E[\tilde{\rho}_{1,\OldSigma
}+\tilde{\rho}_{2,\OldSigma},\OldSigma<\tilde{\OldTau}\wedge
\tilde{\zeta}]}{\sqrt{s-\OldSigma}}\,d\OldSigma.
\end{eqnarray*}
Denote for simplicity $f(s)=\E[\tilde{\rho}_{1,s},s<\tilde{\OldTau
}\wedge\tilde{\zeta}]$ and $g(s)=C\tilde{\rho}_{0} \Phi
(\frac{\tilde{\rho}_{0}}{\sqrt{s}} )$. Now condition (\ref
{he7}) implies for all $s\in[0,1\wedge t]$ that
%
\begin{equation}
\label{he9} f(s)\le g(s)+C\int_{0}^{s}
\frac{f(\OldSigma)}{\sqrt{s-\OldSigma
}}\,d\OldSigma.
\end{equation}

This functional inequality is interesting enough to be treated
separately, and so we do this formally in the following result.
Incidentally, this also appears in renewal theory, but we were not able
to pinpoint exactly this statement in the literature.

%
\begin{Lemma}\label{hel1}
Assume $f,g\dvtx[0,t]\to[0,\infty)$ are bounded, continuous functions
such that for all $s\in[0,1\wedge t]$
%
\begin{equation}
\label{he10} f(s)\le g(s)+C\int_{0}^{s}
\frac{f(\OldSigma)}{\sqrt{s-\OldSigma
}}\,d\OldSigma.
\end{equation}
If $g(s)\le C\rho^{2}/\sqrt{s}$ for all $s\in[0,1\wedge t]$, then
\[
f(s)\le C\rho^{2}/\sqrt{s}\qquad\mbox{for all }s\in(0,1\wedge t].
\]
\end{Lemma}

\begin{pf} Rewrite (\ref{he10}) in the form
\[
f(s)\le g(s)+C\int_{0}^{s}\frac{f(\OldSigma)}{\sqrt{s-\OldSigma
}}\,d
\OldSigma=g(s)+C\sqrt{s}\int_{0}^{1}
\frac{f(sw)}{\sqrt{1-w}}\,dw.
\]
Now introduce the random variable $W$ with density $\frac{1}{2\sqrt
{1-w}}$ and observe that the right-hand side of the above equation
becomes $g(s)+C\sqrt{s}\E[f(sW)]$. Hence, the inequality at hand can
be rewritten as
\[
f(s)\le g(s)+C\sqrt{s}\E\bigl[f(sW)\bigr].
\]
Iterating this inequality, one can prove that if we pick an i.i.d.
sequence $W_{1},W_{2},\dots$ with the same distribution as $W$, then
for any $n\ge1$,
\begin{eqnarray*}
f(s)&\le&\sum_{k=0}^{n}(C
\sqrt{s})^{k}\E\bigl[\sqrt{W_{1}}\sqrt{W_{1}W_{2}}
\cdots\sqrt{W_{1}W_{2}\cdots W_{k-1}}g(sW_{1}W_{2}
\cdots W_{k})\bigr]
\\
&&{}+(C\sqrt{s})^{n+1}\E\bigl[\sqrt{W_{1}}
\sqrt{W_{1}W_{2}}\cdots\sqrt{W_{1}W_{2}
\cdots W_{n}}f(sW_{1}W_{2}\cdots W_{n+1})
\bigr].
\end{eqnarray*}
The random variable $W$ has moments
\[
\E\bigl[W^{k}\bigr]=\frac{\sqrt{\pi}\Gamma(k+1)}{2\Gamma(k+3/2)}\qquad
\mbox{for all }k>-1.
\]
Particularly important is the case of $k=-1/2$, so that $\frac
{1}{\sqrt{W}}$ is integrable, and in fact $\E[1/\sqrt{W}]=\pi/2$.
It is an elementary task to obtain from this that, for some constant $C>0$,
\[
\E\bigl[W^{k}\bigr]\le C/\sqrt{k}\qquad\mbox{for all }k>0.
\]
Since $g$ is bounded, the series
\[
\sum_{k=0}^{\infty}(C\sqrt{s})^{k}
\E\bigl[\sqrt{W_{1}}\sqrt{W_{1}W_{2}}\cdots
\sqrt{W_{1}W_{2}\cdots W_{k-1}}g(sW_{1}W_{2}
\cdots W_{k})\bigr]
\]
is absolutely convergent and
\[
(C\sqrt{s})^{n}\E\bigl[\sqrt{W_{1}}\sqrt
{W_{1}W_{2}}\cdots\sqrt{W_{1}W_{2}\cdots W_{n}}f(sW_{1}W_{2}\cdots
W_{n+1})\bigr]
\]
goes to $0$ as $n\to\infty$. Consequently,
\[
f(s)\le\sum_{k=0}^{\infty}(C
\sqrt{s})^{k}\E\bigl[\sqrt{W_{1}}\sqrt{W_{1}W_{2}}
\cdots\sqrt{W_{1}W_{2}\cdots W_{k-1}}g(sW_{1}W_{2}
\cdots W_{k})\bigr].
\]
If $g(s)\le C\rho^{2}/\sqrt{s}$, the above yields
\[
f(s)\le C\frac{\rho^{2}}{\sqrt{s}}\sum_{k=0}^{\infty}(C
\sqrt{s})^{k}\E\biggl[\frac{\sqrt{W_{1}}\sqrt{W_{1}W_{2}}\cdots\sqrt
{W_{1}W_{2}\cdots W_{k-1}}}{\sqrt{W_{1}W_{2}\cdots W_{k}}} \biggr]=\frac
{C\rho^{2}}{\sqrt{s}},
\]
where we used the decay of the moments of $W$ together with the fact
that $1/\sqrt{W}$ is integrable to justify that the series is convergent.
\end{pf}

The rest of the proof of (\ref{he20}) follows now from Lemma~\ref{hel1}.

We now turn our attention to (\ref{he21}) and observe that, from
(\ref{he6}), we easily obtain that
\begin{eqnarray*}
d(\tilde{\rho}_{1}\tilde{\rho}_{2})&=&\tilde{
\rho}_{1}\,d\tilde{\rho}_{2}+\tilde{\rho}_{2}\,d
\tilde{\rho}_{1}+d\langle\tilde{\rho}_{1},\tilde{
\rho}_{2} \rangle_{\OldSigma}
\\
&=&dM_{\OldSigma}+O(\tilde{\rho}_{1}+\tilde{\rho}_{2})\,d
\OldSigma
\end{eqnarray*}
with $M_{\OldSigma}$ a martingale. Using this at the times $\OldSigma
=0$ and $\OldSigma=s\wedge\tilde{\OldTau}\wedge\tilde{\zeta}$
with $0\le s\le1\wedge t$ and integrating, we get
\begin{eqnarray*}
\tilde{r}_{0}^{2} \Prob(\tilde{\zeta}<s\wedge\tilde{
\OldTau}) &\le&\E[\tilde{\rho}_{1,s\wedge\tilde{\OldTau}\wedge\tilde
{\zeta}}\tilde{\rho}_{2,s\wedge\tilde{\OldTau}\wedge\tilde
{\zeta}}] \le
\tilde{\rho}^{2}_{0}+C\E\biggl[\int_{0}^{s\wedge
\tilde{\OldTau}\wedge\tilde{\zeta}}(
\tilde{\rho}_{1,\OldSigma
}+\tilde{\rho}_{2,\OldSigma})\,d\OldSigma\biggr]
\\
&\le&\tilde{\rho}^{2}_{0}+C\int_{0}^{s}
\E\bigl[(\tilde{\rho}_{1,\OldSigma}+\tilde{\rho}_{2,\OldSigma
}),\OldSigma<
\tilde{\OldTau}\wedge\tilde{\zeta}\bigr]\,d\OldSigma
\\
&&\hspace*{-33pt} \stackrel{\mbox{\fontas{(\ref{he7}) and (\ref{he20})}}}{\le} \tilde{\rho
}^{2}_{0}+C\int_{0}^{s}
\frac{\tilde{\rho}_{0}^{2}}{\sqrt{\OldSigma
}} \,d\OldSigma=C\tilde{\rho}_{0}^{2},
\end{eqnarray*}
which is what we needed.
\end{pf}

Now we go back to (\ref{he5}). We cannot use Theorem~\ref{hep1} to
conclude that $\E[\rho_{1,s},s<\OldTau\wedge\zeta]\le C\rho
_{0}^{2}/\sqrt{s}$ because the equations satisfied by $\rho_{1}$ and
$\rho_{2}$ are not of the form (\ref{he6}). However, if we take
$\tilde{\rho}_{1,s}=\rho_{1,s}e^{\bar{p}_{t-s}(\lambda_{s})}$,
$\tilde{\rho}_{2,s}=\rho_{2,s}e^{\bar{p}_{t-s}(\lambda_{s})}$,
then (\ref{tc1}) and an application of It\^o's formula (followed by
several rearrangements) show that $\tilde{\rho}_{1}$ and $\tilde
{\rho}_{2}$ do satisfy (\ref{he6}). In addition, Corollary~\ref
{tcc} combined with the fact that $e^{\bar{p}_{t-s}(\lambda_{s})}$
is bounded shows that (\ref{he7}) is also satisfied. Therefore,
according to Theorem~\ref{hep1}, $\E[\tilde{\rho}_{1,s},s<\OldTau
\wedge\zeta]\le C\rho_{0}^{2}/\sqrt{s}$ and this in turn implies
\[
\E[\rho_{1,s},s<\OldTau\wedge\zeta]\le C\rho^{2}_{0}/
\sqrt{s}\quad\mbox{and}\quad\int_{0}^{s}\E[
\rho_{1,u},u<\OldTau\wedge\zeta]\,du\le C\rho^{2}_{0}
\sqrt{s}.
\]

Using the preceding in (\ref{he5}), we write the resulting equation as
%
\begin{eqnarray}
\bigl\llvert\bar{p}_{t}(x)-2\bar{p}_{t}(z)+
\bar{p}_{t}(y)\bigr\rrvert\le c\rho_{0}^{2}
\frac{e^{-Ct}}{\sqrt{s}}+c\rho_{0}^{2}\int_{0}^{s}
\frac{H(t-u)}{\sqrt{u}}\,du \nonumber
\\
\eqntext{\mbox{for any }s\in[0,1\wedge t].}
\end{eqnarray}
Now dividing both sides by $\rho_{0}^{2}$ and then letting $\rho_{0}$
tend to $0$, we arrive at
\[
H(t)\le c\frac{e^{-Ct}}{\sqrt{s}}+c\int_{0}^{s}
\frac{H(t-u)}{\sqrt
{u}}\,du \qquad\mbox{for any } s\in[0,1\wedge t].
\]
From here, the rest is taken care of by the following lemma.

%
\begin{Lemma}\label{hel2} If $H\dvtx[0,\infty)\to[0,\infty)$ is a
continuous function such that, for some constant $C>0$,
%
\begin{equation}
\label{he11} H(t)\le c \biggl(\frac{e^{-Ct}}{\sqrt{s}}+\int_{0}^{s}
\frac
{H(t-u)}{\sqrt{u}}\,du \biggr), \qquad0<s\le1\wedge t,
\end{equation}
then there are constants $k,K>0$ such that
\[
H(t)\le Ke^{-kt} \qquad\mbox{for all }t>0.
\]
\end{Lemma}

\begin{pf} It suffices to concentrate on the case $t\ge1$. The
strategy is similar to the one for proving Lemma~\ref{gel1} with a
few tweaks.

Let $m_{n}=\sup_{t\in[n,n+1]}H(t)$ and $M_{n}=\sup_{t\in
[n-1,n+1]}H(t)$. Clearly, $m_{n}\le M_{n}$ and $M_{n}$ is either
$m_{n}$ or $m_{n-1}$.

Now, if we take the $t$ which maximizes $H(t)$ on $[n,n+1]$ and use
(\ref{he11}), we get that for some constant $C>0$ and any $s\in[0,1]$,
\[
m_{n}\le c \biggl(\frac{e^{-Cn}}{\sqrt{s}}+\sqrt{s}M_{n}
\biggr).
\]
We want to minimize the right-hand side of the above expression over
$s\in[0,1]$. For any $a,b>0$, the minimum of $a/\sqrt{s}+b\sqrt{s}$
with $s\in[0,1]$ is attained at $\frac{a}{b}\wedge1$. Hence,
\[
m_{n}\le c \biggl( \frac{e^{-Cn}}{\sqrt{\sfrac{e^{-Cn}}{M_{n}}}\wedge
1}+ M_{n} \biggl(
\sqrt{\frac{e^{-Cn}}{M_{n}}}\wedge1 \biggr) \biggr).
\]

Now, for each given $n$, we have one of the following two cases:
\begin{longlist}[(2)]
\item[(1)] \textit{Case}: $e^{-Cn/2}\le M_{n}$. This leads first to
$e^{-Cn}/M_{n}<e^{-Cn/2}<1$, and then to
\[
m_{n}\le2ce^{-Cn/2}\sqrt{M_{n}}\le2c
e^{-Cn/4}M_{n}.
\]
This is enough to conclude that for a large $n_{1}$ (e.g., such
that $2ce^{-Cn_{1}/4}<1/2$) and $n\ge n_{1}$ one gets $m_{n}\le
M_{n}/2$, which means that we cannot have $M_{n}=m_{n}$ unless
$m_{n}=m_{n-1}=0$. Hence $M_{n}=m_{n-1}$, which in turn implies that
for some $k>0$
{\renewcommand{\theequation}{*}
\begin{equation}\label{eq*3}
m_{n}\le e^{-k}m_{n-1}
\qquad\mbox{if }n\ge n_{1}.
\end{equation}}\setcounter{equation}{109}%

\item[(2)] \textit{Case}: $M_{n}\le e^{-Cn/2}$. This already yields
{\renewcommand{\theequation}{**}
\begin{equation}\label{eq**3}
m_{n}\le e^{-kn}.
\end{equation}}\setcounter{equation}{109}%
Notice that we can arrange the constant $k>0$ to be the same in (\ref{eq*3}) and
(\ref{eq**3}) simply by taking the smaller.
\end{longlist}

By combining (\ref{eq*3}) and (\ref{eq**3}), we can show that $m_{n}$ decays
exponentially fast. Indeed, if there is $n_{2}\ge n_{1}$ for which the
second alternative holds, then $m_{n_{2}}\le e^{-k n_{2}}$. Then an
easy induction and use of both alternatives yields that $m_{n}\le
e^{-kn}$ for all $n\ge n_{2}$. On the other hand, if there is no such
$n_{2}$, that means the first alternative holds, and this means that
$m_{n}\le m_{n-1}e^{-k}$ for all $n\ge n_{1}$. This then results in
$m_{n}\le m_{n_{1}}e^{-k(n-n_{1})}$, and thus the exponential decay
follows again.
\end{pf}

This completes the proof of Theorem~\ref{het}.
\end{pf}

\section{$C^{k}$ convergence of \texorpdfstring{$\bar{p}$}{$bar{p}$} on surfaces with \texorpdfstring{$\chi(M)\le0$}{$chi(M)<=0$}}\label{s10}

In the previous two sections, using the same notation and assumptions,
we proved there exists a constant $C>0$ such that
%
\begin{eqnarray}\label{ho1}
\sup_{x\in M}\bigl\llvert\bar{p}_{t}(x)
\bigr\rrvert+\sup_{x\in
M}\bigl\llvert\nabla\bar{p}_{t}(x)\bigr\rrvert+\sup_{x\in M}\bigl
\llvert
\operatorname{Hess}\bar{p}_{t}(x)\bigr\rrvert\le ce^{-Ct}
\nonumber\\[-10pt]\\[-10pt]
\eqntext{\mbox{for all }t>0.}
\end{eqnarray}
Alternatively stated, $\bar{p}$ converges to $0$ exponentially fast in
the $C^{2}$-norm. In particular, this proves that the metric $g_{t}$
converges to the constant curvature metric $h$ in the $C^{2}$-topology,
and thus the curvature of $g_t$ converges uniformly to a constant.

We now complete our discussion of the convergence to the constant
curvature metric by extending this to $C^\infty$-convergence. The
culmination of the last several sections is the following theorem.

%
\begin{teo}\label{teoCInfyConv}
Let $M$ be a smooth, compact surface with $\chi(M)\le0$, with a
reference metric $h$ of constant curvature $0$ or $-1$, and let $g_0$
be a smooth initial metric in the same conformal class as $h$ and with
the same area. Then if we let $\bar{p}_t$ for $t\in[0,\infty)$ be
the associated solution to the normalized Ricci flow [as given in
equation (\ref{EqnNRFp})], we have that
\[
\bar{p}_t \rightarrow0 \qquad\mbox{in }C^\infty, \mbox{ exponentially fast,}
\]
in the sense that this convergence takes place exponentially fast in
the $C^k$-norm for all positive integers $k$. Stated differently, if
$g_t$ for $t\in[0,\infty)$ is the family of solution metrics to the
normalized Ricci flow (and so the metrics corresponding to~$\bar
{p}_t$), then $g_t\rightarrow h$ in $C^\infty$, exponentially fast.
\end{teo}

\begin{pf}
We start with the equation
\[
\partial_{t}\bar{p}=e^{-2\bar{p}_{t}}\Delta\bar{p}_{t}+r
\bigl(1-e^{-2\bar{p}_{t}}\bigr).
\]

Now we can assume, by induction, that all derivatives of $\bar{p}_{t}$
of order $l$ with $0\le l\le k-1$ decay to 0 exponentially fast as $t$
goes to infinity. In light of the $C^{2}$-convergence, we may assume
that $k\ge3$.

Taking the $k$th derivative $\bar{p}_{t}^{(k)}=\nabla^{(k)}\bar
{p}_{t}$, after commuting the Laplacian with the covariant derivative
we obtain
%
\begin{equation}
\label{ho2} \partial_{t}\bar{p}_{t}^{(k)}=e^{-2\bar{p}_{t}}
\Delta\bar{p}^{(k)}_{t}+2re^{-2\bar{p}_{t}}
\bar{p}_{t}^{(k)}+Q^{(k)}_{t},
\end{equation}
where $Q^{k}$ depends on the lower order derivatives of $\bar{p}_{t}$,
and thus we may assume by induction that for $k\ge2$,
%
\begin{equation}
\bigl\llvert Q^{(k)}_{t}\bigr\rrvert\le ce^{-Ct}.
\end{equation}

The idea now is to write a Feynman--Kac formula for the solution to
(\ref{ho2}) and get the estimates from this. Indeed, notice that if
$x_{\sigma}$ is the time changed Brownian motion starting at $x$ which
is defined by (\ref{e1b}), then
%
\begin{eqnarray}\label{ho3}
&& \exp\biggl(2r\int_{0}^{\sigma}e^{-2\bar{p}_{t-u}(x_{u})}\,du
\biggr)\mathcal{T}_{\sigma}\bar{p}^{(k)}_{t-\sigma}(x_{\sigma})
\nonumber\\[-8pt]\\[-8pt]\nonumber
&&\qquad{} +
\int_{0}^{\sigma}\exp\biggl(2r\int
_{0}^{u}e^{-2\bar{p}_{t-v}(x_{v})}\,dv \biggr)
\mathcal{T}_{u}Q^{(k)}_{t-u}(x_{u})\,du
\end{eqnarray}
is a martingale, where $\mathcal{T}_{u}$ is the extension to tensors
of the parallel transport (with respect to the underlying metric $h$)
along the path $x\mid_{[u,0]}$ from $x_{u}$ to $x_{0}=x$. From the
technical side, this expression can be seen in a clear way by lifting
the equation~(\ref{ho2}) to the orthonormal frame bundle, where the
lift of $\bar{p}_{t}^{(k)}$ takes values in a tensor product space of
a fixed 2-dimensional Euclidean space. This is standard in stochastic
analysis and we do not belabor it.

One result of equation (\ref{ho3}) is that evaluation at $\sigma=0$
and $\sigma=t$ yields
%
\begin{eqnarray}
\label{ho4} \bar{p}_{t}^{(k)}(x)&=&\E\biggl[\exp
\biggl(2r\int_{0}^{t}e^{-2\bar{p}_{t-u}(x_{u})}\,du \biggr)
\mathcal{T}_{t}\bar{p}^{(k)}_{0}(x_{t})
\biggr]
\nonumber\\[-8pt]\\[-8pt]\nonumber
&&{} +\E\biggl[\int_{0}^{t}\exp\biggl(2r\int
_{0}^{u}e^{-2\bar{p}_{t-v}(x_{v})}\,dv \biggr)\mathcal
{T}_{u}Q^{(k)}_{t-u}(x_{u})\,du \biggr].
\end{eqnarray}
Notice the first consequence of this, namely that $\llvert\bar
{p}^{(k)}_{t}\rrvert$ is bounded for $r\le0$ (which is the case
under consideration).
We consider separately the cases $r=-1$ and $r=0$.\vspace*{6pt}

\textit{Case}: $r=-1$. From the exponential decay of $\bar{p}_{t}$ and the
induction hypothesis (the decay of $Q_{t}^{(k)}$), it is easy to see that
\[
\bigl\llvert\bar{p}_{t}^{(k)}(x)\bigr\rrvert\le
ce^{-Ct}\qquad\mbox{for all }t\ge0,
\]
and thus the induction is done.\vspace*{6pt}

\textit{Case}: $r=0$. For the flat case, we still learn from (\ref{ho4}) that
$\bar{p}_{t}^{(k)}(x)$ is uniformly bounded in $t$ and $x$. Since the
curvature of the underlying metric $h$ is $0$, we know (cf. \cite{KN1},
Theorem 8.1, Chapter II) that the holonomy groups are trivial
(perhaps after lifting to the orientation cover). Stated differently,
the parallel transport along loops is the identity.

To finish the argument, we are going to use the coupling technique we
already exploited for the gradient estimates. Start with a fixed point
$x\in M$ and a unit vector $\xi$, and write
%
\begin{equation}
\label{ho7} \bar{p}^{k}_{t}(x)\xi=\nabla_{\xi}
\bar{p}^{(k-1)}_{t}=\lim_{h\to
0}
\frac{\mathcal{T}_{h}\bar{p}^{(k-1)}_{t}(\gamma(h))-\bar{p}^{(k-1)}_{t}(x)}{h},
\end{equation}
where $\mathcal{T}_{h}$ is the parallel transport from $T_{\gamma
(h)}$ to $T_{x}$ along the geodesic $\gamma$ started at $x$ with
initial velocity $\xi$.

Now we use the martingale representation (\ref{ho3}) with $k$
replaced by $(k-1)$ to see that, for $x$ and $y$ close enough and
$\mathcal{T}$ the parallel transport from $T_{y}$ to $T_{x}$ along the
minimizing geodesic,
\begin{eqnarray*}
&& \mathcal{T}\bar{p}^{(k-1)}_{t}(y)-\bar{p}^{(k-1)}_{t}(x)
\\
&&\qquad =
\E\bigl[ \mathcal{T} \mathcal{T}_{\sigma}\bar{p}^{(k-1)}_{t-\sigma
}(y_{\sigma})-
\mathcal{T}_{\sigma}\bar{p}^{(k-1)}_{t-\sigma
}(x_{\sigma})
\bigr]
\\
&&\qquad\quad{} -\E\biggl[\int_{0}^{\sigma} \bigl(\mathcal{T}
\mathcal{T}_{u}Q^{(k-1)}_{t-u}(y_{u})-
\mathcal{T}_{u}Q^{(k-1)}_{t-u}(x_{u})
\bigr)\,du \biggr].
\end{eqnarray*}
Take $t\ge1$ and let $\sigma$ be $1\wedge\tau$ with $\tau$ the
coupling time of $x_{u}$ and $y_{u}$ which run mirror coupled. Now,
because the holonomy group is trivial, it follows that
\begin{eqnarray*}
&& \E\bigl[ \mathcal{T} \mathcal{T}_{1\wedge\tau}\bar
{p}^{(k-1)}_{t-1\wedge\tau}(y_{1\wedge\tau})-
\mathcal{T}_{1\wedge
\tau}\bar{p}^{(k-1)}_{t-1\wedge\tau}(x_{1\wedge\tau})
\bigr]
\\
&&\qquad =\E\bigl[ \mathcal{T} \mathcal{T}_{1}\bar{p}^{(k-1)}_{t-1}(y_{1})-
\mathcal{T}_{1}\bar{p}^{(k-1)}_{t-1}(x_{1}),1<
\tau\bigr].
\end{eqnarray*}
From this and the exponential decay of $\bar{p}_{t}^{(k-1)}$ and
$Q_{t}^{(k-1)}$, we have
\[
\bigl\llvert\mathcal{T}\bar{p}^{(k-1)}_{t}(y)-\bar
{p}^{(k-1)}_{t}(x)\bigr\rrvert\le e^{-Ct}\Prob(1<
\tau) +e^{-Ct}\int_{0}^{1}\Prob(u<\tau)
\,du.
\]
Finally, using the estimate (\ref{ge20}), we get
\begin{eqnarray*}
&& \bigl\llvert\mathcal{T}\bar{p}^{(k-1)}_{t}(y)-\bar
{p}^{(k-1)}_{t}(x)\bigr\rrvert
\\
&&\qquad \le e^{-Ct}\,d(x,y)
+e^{-Ct}\int_{0}^{1}\frac{d(x,y)}{\sqrt{u}}
\,du=Ce^{-Ct}\,d(x,y).
\end{eqnarray*}
Now taking $y=\gamma(h)$ and considering the limit as $h$ goes to $0$
leads to
\[
\bigl\llvert\bar{p}^{(k)}_{t}(x)\xi\bigr\rrvert\le
ce^{-Ct}
\]
for any unit vector $\xi$, which implies the exponential convergence
of $\bar{p}^{(k)}_{t}$.
\end{pf}

\section*{Acknowledgements} We would like to express our true
appreciation and gratitude for the scholarly, careful, pertinent and
sharp remarks of the anonymous reviewers which transformed the present
paper into a much better one.

Ionel Popescu thanks Sergiu Moroianu for very useful discussions on
the geometry of surfaces particularly enlightening being the
uniformization of surfaces arguments from the manuscript \cite{sergiu}.





%

\printaddresses

\begin{thebibliography}{44}
\bibitem{ACT}
%
\begin{barticle}[mr]
\bauthor{\bsnm{Arnaudon},~\bfnm{Marc}\binits{M.}},
\bauthor{\bsnm{Coulibaly},~\bfnm{Kolehe~Abdoulaye}\binits{K.~A.}} \AND
\bauthor{\bsnm{Thalmaier},~\bfnm{Anton}\binits{A.}}
(\byear{2008}).
\btitle{Brownian motion with respect to a metric depending on time:
Definition, existence and applications to {R}icci flow}.
\bjournal{C. R. Math. Acad. Sci. Paris}
\bvolume{346}
\bpages{773--778}.
\bid{doi={10.1016/j.crma.2008.05.004}, issn={1631-073X}, mr={2427080}}
\end{barticle}
%
\bptok{imsref}%
\endbibitem

\bibitem{r3}
%
\begin{bincollection}[mr]
\bauthor{\bsnm{Arnaudon},~\bfnm{Marc}\binits{M.}},
\bauthor{\bsnm{Coulibaly},~\bfnm{Kol{\'e}h{\`e}~Abdoulaye}\binits
{K.~A.}} \AND
\bauthor{\bsnm{Thalmaier},~\bfnm{Anton}\binits{A.}}
(\byear{2011}).
\btitle{Horizontal diffusion in {$C\sp1$} path space}.
In \bbooktitle{S\'eminaire de {P}robabilit\'es {XLIII}}.
\bseries{Lecture Notes in Math.}
\bvolume{2006}
\bpages{73--94}.
\bpublisher{Springer},
\blocation{Berlin}.
\bid{doi={10.1007/978-3-642-15217-7_2}, mr={2790368}}
\end{bincollection}
%
\bptok{imsref}%
\endbibitem

\bibitem{coup10}
%
\begin{barticle}[mr]
\bauthor{\bsnm{Arnaudon},~\bfnm{Marc}\binits{M.}},
\bauthor{\bsnm{Thalmaier},~\bfnm{Anton}\binits{A.}} \AND
\bauthor{\bsnm{Wang},~\bfnm{Feng-Yu}\binits{F.-Y.}}
(\byear{2006}).
\btitle{Harnack inequality and heat kernel estimates on manifolds with
curvature unbounded below}.
\bjournal{Bull. Sci. Math.}
\bvolume{130}
\bpages{223--233}.
\bid{doi={10.1016/j.bulsci.2005.10.001}, issn={0007-4497}, mr={2215664}}
\end{barticle}
%
\bptok{imsref}%
\endbibitem

\bibitem{coup5}
%
\begin{barticle}[mr]
\bauthor{\bsnm{Ba{\~n}uelos},~\bfnm{Rodrigo}\binits{R.}} \AND
\bauthor{\bsnm{Burdzy},~\bfnm{Krzysztof}\binits{K.}}
(\byear{1999}).
\btitle{On the ``hot spots'' conjecture of {J}. {R}auch}.
\bjournal{J. Funct. Anal.}
\bvolume{164}
\bpages{1--33}.
\bid{doi={10.1006/jfan.1999.3397}, issn={0022-1236}, mr={1694534}}
\end{barticle}
%
\bptok{imsref}%
\endbibitem

\bibitem{coup3}
%
\begin{barticle}[mr]
\bauthor{\bsnm{Ba{\~n}uelos},~\bfnm{Rodrigo}\binits{R.}},
\bauthor{\bsnm{Pang},~\bfnm{Michael}\binits{M.}} \AND
\bauthor{\bsnm{Pascu},~\bfnm{Mihai}\binits{M.}}
(\byear{2004}).
\btitle{Brownian motion with killing and reflection and the
``hot-spots'' problem}.
\bjournal{Probab. Theory Related Fields}
\bvolume{130}
\bpages{56--68}.
\bid{doi={10.1007/s00440-003-0323-x}, issn={0178-8051}, mr={2092873}}
\end{barticle}
%
\bptok{imsref}%
\endbibitem

\bibitem{coup6}
%
\begin{barticle}[mr]
\bauthor{\bsnm{Bass},~\bfnm{Richard~F.}\binits{R.~F.}} \AND
\bauthor{\bsnm{Hsu},~\bfnm{Pei}\binits{P.}}
(\byear{1991}).
\btitle{Some potential theory for reflecting {B}rownian motion in {H}\"
older and {L}ipschitz domains}.
\bjournal{Ann. Probab.}
\bvolume{19}
\bpages{486--508}.
\bid{issn={0091-1798}, mr={1106272}}
\end{barticle}
%
\bptok{imsref}%
\endbibitem

\bibitem{Burdzy-Benjamini}
%
\begin{barticle}[mr]
\bauthor{\bsnm{Benjamini},~\bfnm{Itai}\binits{I.}},
\bauthor{\bsnm{Burdzy},~\bfnm{Krzysztof}\binits{K.}} \AND
\bauthor{\bsnm{Chen},~\bfnm{Zhen-Qing}\binits{Z.-Q.}}
(\byear{2007}).
\btitle{Shy couplings}.
\bjournal{Probab. Theory Related Fields}
\bvolume{137}
\bpages{345--377}.
\bid{doi={10.1007/s00440-006-0008-3}, issn={0178-8051}, mr={2278461}}
\end{barticle}
%
\bptok{imsref}%
\endbibitem

\bibitem{MR3077525}
%
\begin{barticle}[mr]
\bauthor{\bsnm{Bramson},~\bfnm{Maury}\binits{M.}},
\bauthor{\bsnm{Burdzy},~\bfnm{Krzysztof}\binits{K.}} \AND
\bauthor{\bsnm{Kendall},~\bfnm{Wilfrid}\binits{W.}}
(\byear{2013}).
\btitle{Shy couplings, {$\rm CAT(0)$} spaces, and the {L}ion and {M}an}.
\bjournal{Ann. Probab.}
\bvolume{41}
\bpages{744--784}.
\bid{doi={10.1214/11-AOP723}, issn={0091-1798}, mr={3077525}}
\end{barticle}
%
\bptok{imsref}%
\endbibitem

\bibitem{MR1768241}
%
\begin{barticle}[mr]
\bauthor{\bsnm{Burdzy},~\bfnm{Krzysztof}\binits{K.}} \AND
\bauthor{\bsnm{Kendall},~\bfnm{Wilfrid~S.}\binits{W.~S.}}
(\byear{2000}).
\btitle{Efficient {M}arkovian couplings: Examples and counterexamples}.
\bjournal{Ann. Appl. Probab.}
\bvolume{10}
\bpages{362--409}.
\bid{doi={10.1214/aoap/1019487348}, issn={1050-5164}, mr={1768241}}
\end{barticle}
%
\bptok{imsref}%
\endbibitem

\bibitem{Cao}
%
\begin{barticle}[mr]
\bauthor{\bsnm{Cao},~\bfnm{Huai~Dong}\binits{H.~D.}}
(\byear{1985}).
\btitle{Deformation of {K}\"ahler metrics to {K}\"ahler--{E}instein
metrics on compact {K}\"ahler manifolds}.
\bjournal{Invent. Math.}
\bvolume{81}
\bpages{359--372}.
\bid{doi={10.1007/BF01389058}, issn={0020-9910}, mr={0799272}}
\end{barticle}
%
\bptok{imsref}%
\endbibitem

\bibitem{Chee}
%
\begin{bbook}[mr]
\bauthor{\bsnm{Cheeger},~\bfnm{Jeff}\binits{J.}} \AND
\bauthor{\bsnm{Ebin},~\bfnm{David~G.}\binits{D.~G.}}
(\byear{1975}).
\btitle{Comparison Theorems in {R}iemannian Geometry}.
\bseries{North-Holland Mathematical Library}
\bvolume{9}
\bpublisher{North-Holland},
\blocation{Amsterdam}.
\bid{mr={0458335}}
\end{bbook}
%
\bptok{imsref}%
\endbibitem

\bibitem{CheriditoST}
%
\begin{barticle}[mr]
\bauthor{\bsnm{Cheridito},~\bfnm{Patrick}\binits{P.}},
\bauthor{\bsnm{Soner},~\bfnm{H.~Mete}\binits{H.~M.}},
\bauthor{\bsnm{Touzi},~\bfnm{Nizar}\binits{N.}} \AND
\bauthor{\bsnm{Victoir},~\bfnm{Nicolas}\binits{N.}}
(\byear{2007}).
\btitle{Second-order backward stochastic differential equations and
fully nonlinear parabolic {PDE}s}.
\bjournal{Comm. Pure Appl. Math.}
\bvolume{60}
\bpages{1081--1110}.
\bid{doi={10.1002/cpa.20168}, issn={0010-3640}, mr={2319056}}
\end{barticle}
%
\bptok{imsref}%
\endbibitem

\bibitem{ChKn1}
%
\begin{bbook}[mr]
\bauthor{\bsnm{Chow},~\bfnm{Bennett}\binits{B.}} \AND
\bauthor{\bsnm{Knopf},~\bfnm{Dan}\binits{D.}}
(\byear{2004}).
\btitle{The {R}icci Flow: An Introduction}.
\bseries{Mathematical Surveys and Monographs}
\bvolume{110}.
\bpublisher{Amer. Math. Soc.},
\blocation{Providence, RI}.
\bid{doi={10.1090/surv/110}, mr={2061425}}
\end{bbook}
%
\bptok{imsref}%
\endbibitem

\bibitem{CLN}
%
\begin{bbook}[mr]
\bauthor{\bsnm{Chow},~\bfnm{Bennett}\binits{B.}},
\bauthor{\bsnm{Lu},~\bfnm{Peng}\binits{P.}} \AND
\bauthor{\bsnm{Ni},~\bfnm{Lei}\binits{L.}}
(\byear{2006}).
\btitle{Hamilton's {R}icci Flow}.
\bseries{Graduate Studies in Mathematics}
\bvolume{77}.
\bpublisher{Amer. Math. Soc.},
\blocation{Providence, RI}.
\bid{mr={2274812}}
\end{bbook}
%
\bptok{imsref}%
\endbibitem

\bibitem{KACP}
%
\begin{barticle}[mr]
\bauthor{\bsnm{Coulibaly-Pasquier},~\bfnm{Kol{\'e}h{\`e}~A.}\binits{K.~A.}}
(\byear{2011}).
\btitle{Brownian motion with respect to time-changing {R}iemannian
metrics, applications to {R}icci flow}.
\bjournal{Ann. Inst. Henri Poincar\'e Probab. Stat.}
\bvolume{47}
\bpages{515--538}.
\bid{doi={10.1214/10-AIHP364}, issn={0246-0203}, mr={2814421}}
\end{barticle}
%
\bptok{imsref}%
\endbibitem

\bibitem{CranstonJFA}
%
\begin{barticle}[mr]
\bauthor{\bsnm{Cranston},~\bfnm{M.}\binits{M.}}
(\byear{1991}).
\btitle{Gradient estimates on manifolds using coupling}.
\bjournal{J. Funct. Anal.}
\bvolume{99}
\bpages{110--124}.
\bid{doi={10.1016/0022-1236(91)90054-9}, issn={0022-1236}, mr={1120916}}
\end{barticle}
%
\bptok{imsref}%
\endbibitem

\bibitem{CranstonTripleCouple}
%
\begin{barticle}[mr]
\bauthor{\bsnm{Cranston},~\bfnm{M.}\binits{M.}}
(\byear{1992}).
\btitle{A probabilistic approach to gradient estimates}.
\bjournal{Canad. Math. Bull.}
\bvolume{35}
\bpages{46--55}.
\bid{doi={10.4153/CMB-1992-007-6}, issn={0008-4395}, mr={1157463}}
\end{barticle}
%
\bptok{imsref}%
\endbibitem

\bibitem{MR1001020}
%
\begin{barticle}[mr]
\bauthor{\bsnm{Cranston},~\bfnm{M.}\binits{M.}} \AND
\bauthor{\bsnm{Le Jan},~\bfnm{Y.}\binits{Y.}}
(\byear{1989}).
\btitle{On the noncoalescence of a two point {B}rownian motion
reflecting on a circle}.
\bjournal{Ann. Inst. Henri Poincar\'e Probab. Stat.}
\bvolume{25}
\bpages{99--107}.
\bid{issn={0246-0203}, mr={1001020}}
\end{barticle}
%
\bptok{imsref}%
\endbibitem

\bibitem{MR1080491}
%
\begin{barticle}[mr]
\bauthor{\bsnm{Cranston},~\bfnm{M.}\binits{M.}} \AND
\bauthor{\bsnm{Le Jan},~\bfnm{Y.}\binits{Y.}}
(\byear{1990}).
\btitle{Noncoalescence for the {S}korohod equation in a convex domain
of {${\mathbf R}\sp2$}}.
\bjournal{Probab. Theory Related Fields}
\bvolume{87}
\bpages{241--252}.
\bid{doi={10.1007/BF01198431}, issn={0178-8051}, mr={1080491}}
\end{barticle}
%
\bptok{imsref}%
\endbibitem

\bibitem{DoC}
%
\begin{bbook}[mr]
\bauthor{\bsnm{do Carmo},~\bfnm{Manfredo~Perdig{\~a}o}\binits{M.~P.}}
(\byear{1992}).
\btitle{Riemannian Geometry}.
\bpublisher{Birkh\"auser},
\blocation{Boston, MA}.
\bid{doi={10.1007/978-1-4757-2201-7}, mr={1138207}}
\end{bbook}
%
\bptok{imsref}%
\endbibitem

\bibitem{Touzi}
%
\begin{barticle}[mr]
\bauthor{\bsnm{Fahim},~\bfnm{Arash}\binits{A.}},
\bauthor{\bsnm{Touzi},~\bfnm{Nizar}\binits{N.}} \AND
\bauthor{\bsnm{Warin},~\bfnm{Xavier}\binits{X.}}
(\byear{2011}).
\btitle{A probabilistic numerical method for fully nonlinear parabolic {PDE}s}.
\bjournal{Ann. Appl. Probab.}
\bvolume{21}
\bpages{1322--1364}.
\bid{doi={10.1214/10-AAP723}, issn={1050-5164}, mr={2857450}}
\end{barticle}
%
\bptok{imsref}%
\endbibitem

\bibitem{JY}
%
\begin{barticle}[mr]
\bauthor{\bsnm{G{\"o}ing-Jaeschke},~\bfnm{Anja}\binits{A.}} \AND
\bauthor{\bsnm{Yor},~\bfnm{Marc}\binits{M.}}
(\byear{2003}).
\btitle{A survey and some generalizations of {B}essel processes}.
\bjournal{Bernoulli}
\bvolume{9}
\bpages{313--349}.
\bid{doi={10.3150/bj/1068128980}, issn={1350-7265}, mr={1997032}}
\end{barticle}
%
\bptok{imsref}%
\endbibitem

\bibitem{sergiu}
%
\begin{bmisc}[auto:parserefs-M02]
\bauthor{\bsnm{Guillarmou},~\bfnm{C.}\binits{C.}} \AND
\bauthor{\bsnm{Moroianu},~\bfnm{Sergiu}\binits{S.}}
\bhowpublished{Surfaces. Available at \surl{http://www.imar.ro/\\\%7Esergium/fisiere/rs.pdf}.}
\end{bmisc}
%
\bptok{imsref}%
\endbibitem

\bibitem{Ham1}
%
\begin{bincollection}[mr]
\bauthor{\bsnm{Hamilton},~\bfnm{Richard~S.}\binits{R.~S.}}
(\byear{1988}).
\btitle{The {R}icci flow on surfaces}.
In \bbooktitle{Mathematics and General Relativity ({S}anta {C}ruz,
{CA}, 1986)}.
\bseries{Contemp. Math.}
\bvolume{71}
\bpages{237--262}.
\bpublisher{Amer. Math. Soc.},
\blocation{Providence, RI}.
\bid{doi={10.1090/conm/071/954419}, mr={0954419}}
\end{bincollection}
%
\bptok{imsref}%
\endbibitem

\bibitem{Hatcher}
%
\begin{bbook}[mr]
\bauthor{\bsnm{Hatcher},~\bfnm{Allen}\binits{A.}}
(\byear{2002}).
\btitle{Algebraic Topology}.
\bpublisher{Cambridge Univ. Press},
\blocation{Cambridge}.
\bid{mr={1867354}}
\end{bbook}
%
\bptok{imsref}%
\endbibitem

\bibitem{Elton}
%
\begin{bbook}[mr]
\bauthor{\bsnm{Hsu},~\bfnm{Elton~P.}\binits{E.~P.}}
(\byear{2002}).
\btitle{Stochastic Analysis on Manifolds}.
\bseries{Graduate Studies in Mathematics}
\bvolume{38}.
\bpublisher{Amer. Math. Soc.},
\blocation{Providence, RI}.
\bid{mr={1882015}}
\end{bbook}
%
\bptok{imsref}%
\endbibitem

\bibitem{coup8}
%
\begin{barticle}[mr]
\bauthor{\bsnm{Kendall},~\bfnm{Wilfrid~S.}\binits{W.~S.}}
(\byear{1986}).
\btitle{Nonnegative {R}icci curvature and the {B}rownian coupling property}.
\bjournal{Stochastics}
\bvolume{19}
\bpages{111--129}.
\bid{doi={10.1080/17442508608833419}, issn={0090-9491}, mr={0864339}}
\end{barticle}
%
\bptok{imsref}%
\endbibitem

\bibitem{coup9}
%
\begin{barticle}[mr]
\bauthor{\bsnm{Kendall},~\bfnm{Wilfrid~S.}\binits{W.~S.}}
(\byear{1986}).
\btitle{Stochastic differential geometry, a coupling property, and
harmonic maps}.
\bjournal{J. Lond. Math. Soc. (2)}
\bvolume{33}
\bpages{554--566}.
\bid{doi={10.1112/jlms/s2-33.3.554}, issn={0024-6107}, mr={0850971}}
\end{barticle}
%
\bptok{imsref}%
\endbibitem

\bibitem{coup7}
%
\begin{barticle}[mr]
\bauthor{\bsnm{Kendall},~\bfnm{Wilfrid~S.}\binits{W.~S.}}
(\byear{1989}).
\btitle{Coupled {B}rownian motions and partial domain monotonicity for
the {N}eumann heat kernel}.
\bjournal{J. Funct. Anal.}
\bvolume{86}
\bpages{226--236}.
\bid{doi={10.1016/0022-1236(89)90053-0}, issn={0022-1236}, mr={1021137}}
\end{barticle}
%
\bptok{imsref}%
\endbibitem

\bibitem{coup11}
%
\begin{barticle}[mr]
\bauthor{\bsnm{Kendall},~\bfnm{Wilfrid~S.}\binits{W.~S.}}
(\byear{2009}).
\btitle{Brownian couplings, convexity, and shy-ness}.
\bjournal{Electron. Commun. Probab.}
\bvolume{14}
\bpages{66--80}.
\bid{doi={10.1214/ECP.v14-1417}, issn={1083-589X}, mr={2481667}}
\end{barticle}
%
\bptok{imsref}%
\endbibitem

\bibitem{KN1}
%
\begin{bbook}[mr]
\bauthor{\bsnm{Kobayashi},~\bfnm{Shoshichi}\binits{S.}} \AND
\bauthor{\bsnm{Nomizu},~\bfnm{Katsumi}\binits{K.}}
(\byear{1996}).
\btitle{Foundations of Differential Geometry}.
\bseries{Wiley Classics Library}
\bvolume{I}.
\bpublisher{Wiley},
\blocation{New York}.
\bid{mr={1393940}}
\end{bbook}
%
\bptok{imsref}%
\endbibitem

\bibitem{r4}
%
\begin{barticle}[mr]
\bauthor{\bsnm{Kuwada},~\bfnm{Kazumasa}\binits{K.}}
(\byear{2012}).
\btitle{Convergence of time-inhomogeneous geodesic random walks and its
application to coupling methods}.
\bjournal{Ann. Probab.}
\bvolume{40}
\bpages{1945--1979}.
\bid{doi={10.1214/11-AOP676}, issn={0091-1798}, mr={3025706}}
\end{barticle}
%
\bptok{imsref}%
\endbibitem

\bibitem{r2}
%
\begin{barticle}[mr]
\bauthor{\bsnm{Kuwada},~\bfnm{Kazumasa}\binits{K.}} \AND
\bauthor{\bsnm{Philipowski},~\bfnm{Robert}\binits{R.}}
(\byear{2011}).
\btitle{Coupling of {B}rownian motions and {P}erelman's {$\mathcal
L$}-functional}.
\bjournal{J. Funct. Anal.}
\bvolume{260}
\bpages{2742--2766}.
\bid{doi={10.1016/j.jfa.2011.01.017}, issn={0022-1236}, mr={2772350}}
\end{barticle}
%
\bptok{imsref}%
\endbibitem

\bibitem{r1}
%
\begin{barticle}[mr]
\bauthor{\bsnm{Kuwada},~\bfnm{Kazumasa}\binits{K.}} \AND
\bauthor{\bsnm{Philipowski},~\bfnm{Robert}\binits{R.}}
(\byear{2011}).
\btitle{Non-explosion of diffusion processes on manifolds with
time-dependent metric}.
\bjournal{Math. Z.}
\bvolume{268}
\bpages{979--991}.
\bid{doi={10.1007/s00209-010-0704-7}, issn={0025-5874}, mr={2818739}}
\end{barticle}
%
\bptok{imsref}%
\endbibitem

\bibitem{LMa}
%
\begin{bbook}[auto:parserefs-M02]
\bauthor{\bsnm{Li},~\bfnm{Ma}\binits{M.}}
(\byear{2004}).
\btitle{Ricci--{H}amilton Flow on Surfaces}.
\bpublisher{Global Scientific Pub.},
\blocation{Singapore}.
\end{bbook}
%
\bptok{imsref}%
\endbibitem

\bibitem{Lin-Rog}
%
\begin{barticle}[mr]
\bauthor{\bsnm{Lindvall},~\bfnm{Torgny}\binits{T.}} \AND
\bauthor{\bsnm{Rogers},~\bfnm{L.~C.~G.}\binits{L.~C.~G.}}
(\byear{1986}).
\btitle{Coupling of multidimensional diffusions by reflection}.
\bjournal{Ann. Probab.}
\bvolume{14}
\bpages{860--872}.
\bid{issn={0091-1798}, mr={0841588}}
\end{barticle}
%
\bptok{imsref}%
\endbibitem



\bibitem{fixeddist}
%
\begin{bmisc}[auto:parserefs-M02]
\bauthor{\bsnm{Pascu},~\bfnm{Mihai}\binits{M.}} \AND
\bauthor{\bsnm{Popescu},~\bfnm{Ionel}\binits{I.}}
\bhowpublished{Shy and fixed-distance couplings of brownian motions on
manifolds. Preprint. Available at \arxivurl{arXiv:1210.7217}.}
\end{bmisc}
%
\bptok{imsref}%
\endbibitem

\bibitem{coup4}
%
\begin{barticle}[mr]
\bauthor{\bsnm{Pascu},~\bfnm{Mihai~N.}\binits{M.~N.}}
(\byear{2002}).
\btitle{Scaling coupling of reflecting {B}rownian motions and the hot
spots problem}.
\bjournal{Trans. Amer. Math. Soc.}
\bvolume{354}
\bpages{4681--4702 (electronic)}.
\bid{doi={10.1090/S0002-9947-02-03020-9}, issn={0002-9947}, mr={1926894}}
\end{barticle}
%
\bptok{imsref}%
\endbibitem

\bibitem{coup1}
%
\begin{barticle}[mr]
\bauthor{\bsnm{Pascu},~\bfnm{Mihai~N.}\binits{M.~N.}}
(\byear{2011}).
\btitle{Mirror coupling of reflecting {B}rownian motion and an
application to {C}havel's conjecture}.
\bjournal{Electron. J. Probab.}
\bvolume{16}
\bpages{504--530}.
\bid{doi={10.1214/EJP.v16-859}, issn={1083-6489}, mr={2781844}}
\end{barticle}
%
\bptok{imsref}%
\endbibitem

\bibitem{coup2}
%
\begin{barticle}[mr]
\bauthor{\bsnm{Pascu},~\bfnm{Mihai~N.}\binits{M.~N.}} \AND
\bauthor{\bsnm{Gageonea},~\bfnm{Maria~E.}\binits{M.~E.}}
(\byear{2011}).
\btitle{Monotonicity properties of the {N}eumann heat kernel in the ball}.
\bjournal{J. Funct. Anal.}
\bvolume{260}
\bpages{490--500}.
\bid{doi={10.1016/j.jfa.2010.08.014}, issn={0022-1236}, mr={2737410}}
\end{barticle}
%
\bptok{imsref}%
\endbibitem

\bibitem{RevYor}
%
\begin{bbook}[mr]
\bauthor{\bsnm{Revuz},~\bfnm{Daniel}\binits{D.}} \AND
\bauthor{\bsnm{Yor},~\bfnm{Marc}\binits{M.}}
(\byear{1999}).
\btitle{Continuous Martingales and {B}rownian Motion},
\bedition{3rd} ed.
\bseries{Grundlehren der Mathematischen Wissenschaften [Fundamental
Principles of Mathematical Sciences]}
\bvolume{293}.
\bpublisher{Springer},
\blocation{Berlin}.
\bid{doi={10.1007/978-3-662-06400-9}, mr={1725357}}
\end{bbook}
%
\bptok{imsref}%
\endbibitem

\bibitem{ST2}
%
\begin{barticle}[mr]
\bauthor{\bsnm{Soner},~\bfnm{H.~Mete}\binits{H.~M.}} \AND
\bauthor{\bsnm{Touzi},~\bfnm{Nizar}\binits{N.}}
(\byear{2002}).
\btitle{A stochastic representation for the level set equations}.
\bjournal{Comm. Partial Differential Equations}
\bvolume{27}
\bpages{2031--2053}.
\bid{doi={10.1081/PDE-120016135}, issn={0360-5302}, mr={1941665}}
\end{barticle}
%
\bptok{imsref}%
\endbibitem

\bibitem{SonerTouzi}
%
\begin{barticle}[mr]
\bauthor{\bsnm{Soner},~\bfnm{H.~Mete}\binits{H.~M.}} \AND
\bauthor{\bsnm{Touzi},~\bfnm{Nizar}\binits{N.}}
(\byear{2003}).
\btitle{A stochastic representation for mean curvature type geometric flows}.
\bjournal{Ann. Probab.}
\bvolume{31}
\bpages{1145--1165}.
\bid{doi={10.1214/aop/1055425773}, issn={0091-1798}, mr={1988466}}
\end{barticle}
%
\bptok{imsref}%
\endbibitem

\bibitem{Stroock}
%
\begin{bbook}[mr]
\bauthor{\bsnm{Stroock},~\bfnm{Daniel~W.}\binits{D.~W.}}
(\byear{2000}).
\btitle{An Introduction to the Analysis of Paths on a {R}iemannian Manifold}.
\bseries{Mathematical Surveys and Monographs}
\bvolume{74}.
\bpublisher{Amer. Math. Soc.},
\blocation{Providence, RI}.
\bid{mr={1715265}}
\end{bbook}
%
\bptok{imsref}%
\endbibitem
\end{thebibliography}
\end{document}